\documentclass[reqno]{amsproc}
\usepackage{amssymb}
\usepackage{pgf,pgfautomata,tikz}
\usepackage{euscript}
\usepackage{extarrows}
\usepackage{mathrsfs} 
\usepackage{units}   
\usepackage{color}

\makeatletter
\@namedef{subjclassname@2010}{%
\textup{2010} Mathematics Subject Classification}
\makeatother

\numberwithin{equation}{subsection}

\newtheorem{thm}{Theorem}[subsection]
\newtheorem{cor}[thm]{Corollary}
\newtheorem{lem}[thm]{Lemma}
\newtheorem{pro}[thm]{Proposition}

\newtheorem*{thm*}{Theorem}

\theoremstyle{remark}
\newtheorem{rem}[thm]{Remark}

\newtheorem{proc}[thm]{Procedure}
\newtheorem{opq}[thm]{Problem}

\theoremstyle{definition}
\newtheorem{exa}[thm]{Example}

\DeclareMathOperator{\D}{d\hspace{-0.25ex}}

\DeclareMathOperator{\paa}{{\mathsf{par}}}

\newcommand*{\ascr}{\mathscr A}

\newcommand*{\borel}[1]{{\mathfrak B}(#1)}

\newcommand*{\card}[1]{\mathrm{card}(#1)}
\newcommand*{\cbb}{\mathbb C}

\newcommand*{\efi}{\mathsf{E}_{\phi}}

\newcommand*{\dz}[1]{{\EuScript D}(#1)}

\newcommand*{\dzn}[1]{{\EuScript D}^\infty(#1)}

\newcommand*{\Ge}{\geqslant}
\newcommand*{\gammab}{\boldsymbol \gamma}
\newcommand*{\gcal}{{\mathscr G}}
\newcommand*{\hh}{\mathcal H}

\newcommand*{\hfi}{{\mathsf h}_{\phi}}

\newcommand*{\hfin}[1]{{\mathsf h}_{\phi^{#1}}}

\newcommand*{\I}{\mathrm{i}}
\newcommand*{\is}[2]{\langle#1,#2\rangle}
\newcommand*{\ind}[2]{\mathrm{ind}_{#1}(#2)}

\newcommand*{\kk}{\mathcal K}

\newcommand*{\lambdab}{{\boldsymbol\lambda}}

\newcommand*{\Le}{\leqslant}
\newcommand*{\msc}{\mathscr M}

\newcommand*{\nbb}{\mathbb N}

\newcommand*{\ogr}[1]{\boldsymbol B(#1)}

\newcommand*{\pa}[1]{\paa(#1)}

\newcommand*{\rbb}{\mathbb R}
\newcommand*{\rbop}{{\overline{\rbb}_+}}
\newcommand*{\slam}{S_{\boldsymbol \lambda}}
\newcommand*{\smalloplus}{\raise0pt
\hbox{$\scriptscriptstyle \oplus$}}

\newcommand*{\supp}[1]{\mathrm{supp}(#1)}
\newcommand*{\tcal}{{\mathscr T}}

\newcommand*{\zbb}{\mathbb Z}
\begin{document}
   \title[Subnormality of composition operators: exotic examples]
{Subnormality of unbounded composition operators over
one-circuit directed graphs: exotic examples}
   \author[P.\ Budzy\'{n}ski]{Piotr Budzy\'{n}ski}
   \address{Katedra Zastosowa\'{n} Matematyki,
Uniwersytet Rolniczy w Krakowie, ul.\ Balicka 253c,
PL-30198 Krak\'ow}
\email{piotr.budzynski@ur.krakow.pl}
   \author[Z.\ J.\ Jab{\l}o\'nski]{Zenon Jan
Jab{\l}o\'nski}
   \address{Instytut Matematyki,
Uniwersytet Jagiello\'nski, ul.\ \L ojasiewicza 6,
PL-30348 Kra\-k\'ow, Poland}
\email{Zenon.Jablonski@im.uj.edu.pl}
   \author[I.\ B.\ Jung]{Il Bong Jung}
   \address{Department of Mathematics, Kyungpook
National University, Daegu 702-701, Korea}
\email{ibjung@knu.ac.kr}
   \author[J.\ Stochel]{Jan Stochel}
\address{Instytut Matematyki, Uniwersytet
Jagiello\'nski, ul.\ \L ojasiewicza 6, PL-30348
Kra\-k\'ow, Poland} \email{Jan.Stochel@im.uj.edu.pl}
   \thanks{The research of the first author was
supported by the Ministry of Science and Higher
Education of the Republic of Poland. The research of
the second and fourth authors was supported by the NCN
(National Science Center), decision No.
DEC-2013/11/B/ST1/03613. The research of the third
author was supported by Basic Science Research Program
through the National Research Foundation of Korea(NRF)
funded by the Ministry of Science, ICT and future
Planning (KRF-2015R1A2A2A01006072).}
   \begin{abstract}
A recent example of a non-hyponormal injective
composition operator in an $L^2$-space generating
Stieltjes moment sequences, invented by three of the
present authors, was built over a non-locally finite
directed tree. The main goal of this paper is to solve
the problem of whether there exists such an operator
over a locally finite directed graph and, in the
affirmative case, to find the simplest possible graph
with these properties (simplicity refers to local
valency). The problem is solved affirmatively for the
locally finite directed graph $\gcal_{2,0}$, which
consists of two branches and one loop. The only
simpler directed graph for which the problem remains
unsolved consists of one branch and one loop. The
consistency condition, the only efficient tool for
verifying subnormality of unbounded composition
operators, is intensively studied in the context of
$\gcal_{2,0}$, which leads to a constructive method of
solving the problem. The method itself is partly based
on transforming the Krein and the Friedrichs measures
coming either from shifted Al-Salam-Carlitz
\mbox{$q$-polynomials} or from a quartic birth and
death process.
   \end{abstract}
   \maketitle
   \newpage
\setcounter{tocdepth}{2} \tableofcontents
   \section{{\bf Preliminaries}}
   \subsection{Introduction} \label{subsec1}
   The theory of bounded subnormal operators was
initiated by Halmos (cf.\ \cite{hal1}). The definition
and the first characterization of their unbounded
counterparts were given independently by Bishop (cf.\
\cite{bis}) and Foia\c{s} (cf.\ \cite{foi}). The
foundations of the theory of unbounded (i.e., not
necessarily bounded) subnormal operators were
developed by the fourth-named author and Szafraniec
(cf.\ \cite{StSz1,StSz2,StSz3,StSz}). The study of
this topic turned out to be highly successful. It led
to a number of challenging problems and nontrivial
results in various branches of mathematics including
functional analysis and mathematical physics (see,
e.g., \cite{Con,Con2,c-f} for the case of bounded
operators and \cite{m-s,jor,c-j-k,kou,k-t1,k-t2,j-s}
for unbounded ones). This area of interest still plays
a vital role in operator theory.

The first characterization of bounded subnormal
operators was given by Halmos himself. It was
successively simplified by Bram (cf.\ \cite{bra}),
Embry (cf.\ \cite{emb}) and Lambert (see \cite{lam};
see also \cite[Theorem 7]{StSz2} where the assumption
of injectivity was removed). The Lambert
characterization states that a bounded Hilbert space
operator is subnormal if and only if it generates
Stieltjes moment sequences (see Section \ref{n&t} for
definitions). It turns out that this characterization
also works for unbounded operators which have
sufficiently many analytic vectors (cf.\ \cite[Theorem
7]{StSz2}). However, it is no longer true for
arbitrary unbounded operators (cf.\ \cite[Section
3.2]{b-j-j-sA}). Recall that subnormal operators with
dense set of $C^{\infty}$-vectors always generate
Stieltjes moment sequences (see \cite[Proposition
3.2.1]{b-j-j-sC}). It is also worth pointing out that
subnormal composition operators in $L^2$-spaces, as
opposed to abstract subnormal operators, are always
injective (see \cite[Corollary 6.3]{b-j-j-sC}). Hence,
there arises the question whether or not composition
operators in $L^2$-spaces generating Stieltjes moment
sequences are injective (see Problem \ref{open2}).

In a recent paper \cite{b-j-j-sS}, we have developed a
completely new, even in the bounded case, approach to
studying subnormality of composition operators (in
$L^2$-spaces over $\sigma$-finite measure spaces)
which involves measurable families of probability
measures satisfying the so-called consistency
condition. This approach provides a criterion (read:
sufficient condition) for subnormality of composition
operators, which does not refer to the density of
domains of powers. The corresponding technique for
weighted shifts on directed trees worked out in
\cite{b-d-j-s} (see also \cite{b-j-j-sW}) enabled us
to construct an unexpected example of a subnormal
composition operator whose square has trivial domain
(cf.\ \cite{b-j-j-sSq}).

As shown in \cite[Example 4.2.1]{j-j-s0}, there are
unbounded injective operators generating Stieltjes
moment sequences which are not even hyponormal, and
thus not subnormal. In fact, it was proved there that
if $\tcal$ is a leafless directed tree which has
exactly one branching vertex and if the branching
vertex itself has {\em infinite}
valency\footnote{\;The valency of a vertex $v$ is
understood as the number of outgoing edges at $v$.},
then there exists a non-hyponormal (injective)
weighted shift on $\tcal$ with nonzero weights
generating Stieltjes moment sequences. Up to
isomorphism, there is only one rootless directed tree
of this kind, denoted in \cite[p.\ 67]{j-j-s} by
$\tcal_{\infty,\infty}$. A weighted shift on
$\tcal_{\infty,\infty}$ with nonzero weights is
unitarily equivalent to an injective composition
operator in an $L^2$-space over a discrete measure
space (see \cite[Lemma 4.3.1]{j-j-s0} and
\cite[Theorem 3.2.1]{j-j-s}). Since the directed graph
induced by the symbol of such a composition operator
coincides with $\tcal_{\infty,\infty}$ (see Section
\ref{sect-gr} for the definition), it is {\em not
locally finite}. This raises the question as to
whether there exists a non-hyponormal injective
composition operator over\footnote{\;A composition
operator $C$ is {\em over} a directed graph $\gcal$ if
$\gcal$ is induced by the symbol of $C$.} a {\em
locally finite} connected directed graph generating
Stieltjes moment sequences, and, if this is the case,
how {\em simple} such a directed graph can be, where
simplicity is understood with respect to local valency
(cf.\ Remark \ref{val1}). The present paper addresses
both of these questions. Taking into account the
simplicity leads to considering directed graphs
induced by self-maps whose vertices, all but one, say
$\omega$, have valency one, and the valency of
$\omega$ is greater than or equal to $1$. Such
directed graphs are described in Theorem
\ref{golfclub} and Remark \ref{val1}. In view of part
1) of Remark \ref{val1}, the situation in which the
valency of $\omega$ is equal to one is excluded by an
unbounded variant of Herrero's
characterization\footnote{\;The Herrero result (see
\cite{Herr}; see also \cite{cu-fi}) is a bilateral
analogue of the Berger-Gellar-Wallen characterization
of bounded subnormal injective unilateral weighted
shifts (cf.\ \cite{hal2,g-w}).} of subnormal injective
bilateral weighted shifts (see \cite[Theorem
5]{StSz2}; see also \cite[Theorem 3.2]{b-j-j-sB} for a
recent approach). If the valency of $\omega$ is
strictly greater than one, then, by Theorem
\ref{golfclub}, we have two cases. The first, which is
described in Theorem \ref{golfclub}\mbox{(ii-b)},
reduces to the directed tree $\tcal_{\infty,\infty}$,
the case studied in \cite{j-j-s0}. Unfortunately, the
method invented in \cite{j-j-s0} does not give any
hope of answering our questions. In the second case,
which is described in Theorem
\ref{golfclub}\mbox{(ii-a)}, the directed graph under
consideration has exactly one circuit of length
$\kappa+1$ starting at $\omega$ and $\eta$ branches of
infinite length attached to $\omega$, where $\eta \in
\{1, 2, 3, \ldots\} \cup \{\infty\}$ and $\kappa \in
\{0,1,2, \ldots\}$ (see Figure $2$ with
$\omega=x_{\kappa}$); denote it by
$\gcal_{\eta,\kappa}$. The culminating result of the
present paper, Theorem \ref{czwarte}, shows that there
exists a non-hyponormal injective composition operator
over the locally finite directed graph $\gcal_{2,0}$
generating Stieltjes moment sequences. This answers
our first question in the affirmative. Regarding
simplicity, the only simpler directed graph which
potentially may admit a composition operator with the
above-mentioned properties is $\gcal_{1,0}$ (the
subnormality over $\gcal_{1,0}$ was studied in
\cite[Section 3.4]{b-j-j-sS}). However, so far this
particular case remains unsolved because composition
operators over $\gcal_{1,0}$ obtained by our method
are automatically subnormal (cf.\ Theorem
\ref{drugie1}(iv)).

A large part of the present paper is devoted to the
study of subnormality of composition operators over
the directed graph $\gcal_{\eta,\kappa}$. They all
have the same symbol $\phi_{\eta,\kappa}$ whereas
masses attached to vertices that define the underlying
$L^2$-space are subject to changes. Since the only
known criterion\footnote{\;General criteria in
\cite{bis,foi,FHSz,Sz4} seem hardly to be applicable
to composition~ operators.} for subnormality of
unbounded composition operators relies on the
consistency condition \eqref{cc} (cf.\ \cite[Theorem
9]{b-j-j-sS}), we first characterize families of Borel
probability measures (on the positive half-line)
indexed by the vertices of $\gcal_{\eta,\kappa}$ which
satisfy \eqref{cc} (cf.\ Theorem \ref{cc1}). This
enables us to model all such families via collections
of measures indexed by the set $\{x_{i,1}\colon i\in
J_{\eta}\}$ which satisfy some natural conditions
(cf.\ Procedure \ref{proceed}), where $\{x_{i,1}\colon
i\in J_{\eta}\}$ are ends of edges outgoing from
$\omega=x_{\kappa}$ not lying on the circuit (see
Figure $2$) and $J_{\eta}$ is the set of all positive
integers less than or equal to $\eta$. The end $x_0$
of the edge that outgoes from $x_{\kappa}$ and lies on
the circuit also plays an important role in our
considerations. Namely, assuming both that the
Radon-Nikodym derivatives $\{\hfin{n}\}_{n=0}^\infty$
(see Section \ref{sec3.1}) calculated at $x_0$ and
$x_{i,1}$, $i\in J_\eta$, form Stieltjes moment
sequences and that appropriate sequences coming from
$\{\hfin{n}(x_0)\}_{n=0}^\infty$ are S-determinate, we
show that the corresponding composition operator over
$\gcal_{\eta,\kappa}$ is subnormal (cf.\ Theorem
\ref{suff}). The case of $\gcal_{1,\kappa}$ does not
require any determinacy assumption and may be written
purely in terms of the Hankel matrices
$[\hfin{i+j}(x_0)]_{i,j=0}^{\infty}$ and
$[\hfin{i+j+1}(x_0)]_{i,j=0}^{\infty}$ (cf.\
Proposition \ref{eta1}). The proofs of Theorem
\ref{suff} and Proposition \ref{eta1} rely on
constructing families of measures satisfying
\eqref{cc}. These two results are in the spirit of
Lambert's characterization of subnormality of bounded
composition operators (cf.\ \cite{lam1}) which is no
longer true for unbounded operators (cf.\
\cite[Theorem 4.3.3]{j-j-s0} and \cite[Section
11]{b-j-j-sC}). The case of bounded composition
operators over $\gcal_{\eta,\kappa}$, which is also
covered by Lambert's criterion, follows easily from
Theorem \ref{suff} (cf.\ Proposition \ref{wn0gr}). The
optimality of the assumptions of Propositions
\ref{eta1} and \ref{wn0gr} is illustrated by Examples
\ref{final} and \ref{prz2}.

It follows from \cite[Theorems 9 and 17]{b-j-j-sS}
(see also Theorem \ref{glowne}) that under the
assumption that $\hfin{n}$ takes finite values for all
positive integers $n$, any family of Borel probability
measures satisfying \eqref{cc} consists of
representing measures of Stieltjes moment sequences
$\{\hfin{n}(x)\}_{n=0}^\infty$, where $x$ varies over
the vertices of $\gcal_{\eta,\kappa}$. In Section
\ref{s4.4} we discuss the question of extending a
given family $\{P(x_{i,1},\cdot)\}_{i \in J_{\eta}}$
of Borel probability measures to a wider one (indexed
by $\gcal_{\eta,\kappa}$) satisfying the consistency
condition \eqref{cc}. According to Theorem \ref{main},
such extension exists if and only if for every $i \in
J_{\eta}$, $\{\hfin{n}(x_{i,1})\}_{n=0}^\infty$ is a
Stieltjes moment sequence represented by a measure
$P(x_{i,1},\cdot)$ satisfying the conditions
\mbox{(i-b)}, \mbox{(i-c)} and \mbox{(i-d)} of this
theorem. The condition \mbox{(i-b)} refers to moments
of the measures $P(x_{i,1},\cdot)$, $i \in J_{\eta}$.
The remaining two are of different nature, namely
\mbox{(i-c)} is a system of $\kappa$ equations (the
case of $\kappa=0$ is not excluded), while
\mbox{(i-d)} is a single inequality. In Theorem
\ref{6.2} we introduce the condition
\mbox{(i-d$^{\prime}$)} which is a weaker version of
\mbox{(i-d)}. This turns out to be the key idea that
leads to constructing exotic examples. Assuming the
S-determinacy of the sequence
$\{\hfin{n}(x_{\kappa})+c\}_{n=0}^\infty$ for any
$c\in (0,\infty)$, it is proved in Theorem \ref{6.2}
that the conditions \mbox{(i-d)} and
\mbox{(i-d$^{\prime}$)} are equivalent (provided the
remaining ones \mbox{(i-a)}, \mbox{(i-b)} and
\mbox{(i-c)} are satisfied). However, this is no
longer true if the S-determinacy assumption is
dropped. We show this by using Procedure
\ref{proceed2} that heavily depends on the existence
of a pair of N-extremal measures satisfying some
constraints (cf.\ Lemma \ref{pierwsze}). The task of
finding such a pair is challenging. It is realized by
transforming via special homotheties the Krein and the
Friedrichs measures (which are particular instances of
N-extremal measures). The crucial properties of these
transformations are described in Lemma \ref{epslem}.
The proof of the existence of the gap between
\mbox{(i-d)} and \mbox{(i-d$^{\prime}$)} is brought to
completion in Theorems \ref{drugie} (the case of $\eta
\Ge 2$) and \ref{drugie1} (the case of $\eta=1$).
Adapting the above technique, we show in Theorem
\ref{czwarte} that for any integer $\eta \Ge 2$, there
exists a non-hyponormal injective composition operator
over $\gcal_{\eta,0}$ which generates Stieltjes moment
sequences. The case of $\eta=\infty$ is treated in
Theorem \ref{trzecie}. The parallel question of
determinacy of moment sequences
$\{\hfin{n}(x)\}_{n=0}^\infty$, $x \in \{x_{\kappa}\}
\cup \{x_{i,j}\colon i \in J_{\eta}, \, j \in \nbb\}$,
is studied in Section \ref{sec5} by using the index of
H-determinacy introduced by Berg and Duran in
\cite{ber-dur}.

As noted above, the proofs of the main results of the
present paper (Theorems \ref{drugie}, \ref{drugie1},
\ref{trzecie} and \ref{czwarte}) essentially depend on
subtle properties of N-extremal measures. The question
of determinacy of moment sequences is of considerable
importance in our study as well. Therefore, for the
sake of completeness, we collect in Section \ref{Sec3}
basic concepts of the classical theory of moments and
include some new results in this field. Using
\cite[Theorem 3.6]{ber-dur}, we show that a measure
which comes from an N-extremal measure by removing an
infinite number of its atoms has infinite index of
H-determinacy (cf.\ Theorem \ref{b-d}). The Carleman
condition, which always guarantees the H-determinacy
of Stieltjes moment sequences, is investigated in
Section \ref{Sec2.5}. The process of transforming
moment sequences and their representing measures,
including N-extremal ones, via homotheties is
described in Section \ref{Sec2.3} (the particular case
of transformations induced by translations has already
been studied via different approaches in
\cite{Ped,sim}). Particular attention is paid to
transforming the Krein and the Friedrichs measures
(cf.\ Theorem \ref{taso}). As a consequence, a new way
of parametrizing N-extremal measures of
H-indeterminate Stieltjes moment sequences is invented
(cf.\ Theorem \ref{tinfsup}) and an unexpected
trichotomy property of N-extremal measures of
H-indeterminate Hamburger moment sequences is proven
(cf.\ Theorem \ref{trichotomy}). The N-extremal
measures used in the proofs of Theorems \ref{drugie},
\ref{drugie1} and \ref{trzecie} are derived from the
Krein and the Friedrichs measures of an
S-indeterminate Stieltjes moment sequence, first by
scaling them and then by transforming them via
carefully chosen homotheties (cf.\ Lemma
\ref{epslem}). As for the proof of Theorem
\ref{czwarte}, the above method requires the usage of
the Krein and the Friedrichs measures coming from
shifted Al-Salam-Carlitz \mbox{$q$-polynomials} (or,
alternatively, from a quartic birth and death process,
see Remark \ref{piate}). The existence, determinacy
and explicit form of orthogonalizing measures for
Al-Salam-Carlitz \mbox{$q$-polynomials}
$\{V^{(a)}_{n}(x;q)\}_{n=0}^{\infty}$ are discussed in
Section \ref{SecASC}. It is worth mentioning that
explicit examples of N-extremal measures such as those
used in the present paper are to the best of our
knowledge very rare (see, e.g.,
\cite{Ism-Mas,Ism-book}).

The necessary facts concerning composition operators
in $L^2$-spaces over discrete measure spaces,
including criteria for their hyponormality and
subnormality, are recapitulated in Section
\ref{sec3.1}. A variety of relations between
Radon-Nikodym derivatives $\{\hfin{n}\}_{n=0}^\infty$
calculated in different vertices of the directed graph
$\gcal_{\eta,\kappa}$ are established in Section
\ref{Sec3RN}.
   \subsection{Notation and terminology\label{n&t}}
Denote by $\cbb$, $\rbb$, $\rbb_+$, $\zbb$, $\zbb_+$
and $\nbb$ the sets of complex numbers, real numbers,
nonnegative real numbers, integers, nonnegative
integers and positive integers, respectively. Set
$\rbop=\rbb_+\cup\{\infty\}$ and $\nbb_2=\nbb\setminus
\{1\}$. Given $k \in \zbb_+ \cup \{\infty\}$, we write
$J_k = \{i \in \nbb\colon i \Le k\}$ (clearly $J_0 =
\emptyset$). The identity map on a set $X$ is denoted
by $\mathrm{id}_X$. We write $\card{X}$ for the
cardinality of a set $X$ and $\chi_{\varDelta}$ for
the characteristic function of a subset $\varDelta$ of
$X$. The symbol ``$\bigsqcup$'' denotes the disjoint
union of sets. A mapping from $X$ to $X$ is called a
{\em self-map} of $X$. The $\sigma$-algebra of all
Borel subsets of a topological space $X$ is denoted by
$\borel{X}$. All measures considered in this paper are
positive. Since any finite Borel measure $\nu$ on
$\rbb$ is automatically regular (cf.\ \cite[Theorem
2.18]{Rud}), we can consider its closed support; we
denote it by $\supp{\nu}$. Given $t\in \rbb$, we write
$\delta_t$ for the Borel probability measure on $\rbb$
such that $\supp{\delta_t} = \{t\}$. In this paper we
will use the notation $\int_a^b$ and $\int_a^{\infty}$
in place of $\int_{[a,b]}$ and $\int_{[a,\infty)}$
respectively ($a,b\in \rbb$). We also use the
convention that $0^0=1$. The ring of all complex
polynomials in one real variable $t$ (which in the
context of $L^2$-spaces are regarded as equivalence
classes) is denoted by $\cbb[t]$.

Let $A$ be an operator in a complex Hilbert space
$\hh$ (all operators considered in this paper are
linear). Denote by $\dz{A}$ the domain of $A$. If $A$
is closable, then the closure of $A$ is denoted by
$\bar A$. Set $\dzn{A} = \bigcap_{n=0}^\infty
\dz{A^n}$ with $A^0=I$, where $I$ is the identity
operator on $\hh$. We say that $A$ is {\em positive}
if $\is{Af}{f} \Ge 0$ for all $f\in \dz{A}$. $A$ is
said to be {\em normal} if it is densely defined,
$\dz{A}=\dz{A^*}$ and $\|A^*f\| = \|Af\|$ for all
$f\in \dz{A}$ (or, equivalently, if and only if $A$ is
closed, densely defined and $A^*A=AA^*$, see
\cite[Proposition, page 125]{Weid}). $A$ is called
{\em subnormal} if $A$ is densely defined and there
exists a normal operator $N$ in a complex Hilbert
space $\kk$ with $\hh\subseteq \kk$ (isometric
embedding) such that $\dz{A} \subseteq \dz{N}$ and $Af
= Nf$ for all $f \in \dz{A}$. $A$ is said to be {\em
hyponormal} if it is densely defined, $\dz{A}
\subseteq \dz{A^*}$ and $\|A^* f\| \Le \|Af\|$ for all
$f \in \dz{A}$. Following \cite{j-j-s0}, we say that
$A$ {\em generates Stieltjes moment sequences} if
$\dzn{A}$ is dense in $\hh$ and $\{\|A^n
f\|^2\}_{n=0}^\infty$ is a Stieltjes moment sequence
for every $f \in \dzn{A}$ (see Section \ref{secbc}
below for the definition and basic properties of
Stieltjes moment sequences). It is known that if $A$
is subnormal and $\dzn{A}$ is dense in $\hh$, then $A$
generates Stieltjes moment sequences (cf.\
\cite[Proposition 3.2.1]{b-j-j-sA}). However, the
reverse implication is not true in general (see
\cite{sto-a}; see also \cite{j-j-s0}).

In what follows $\ogr{\hh}$ stands for the
$C^*$-algebra of all bounded operators in $\hh$ whose
domains are equal to $\hh$.
   \section{{\bf Determinacy in moment problems}}
\label{Sec3}
   \subsection{Basic concepts\label{secbc}}
Denote by $\msc$ the set of all Borel measures $\nu$
on $\rbb$ such that $\int_{\rbb} |t|^n \D\nu(t) <
\infty$ for all $n\in \zbb_+$. Set $\msc^+ = \{\nu \in
\msc\colon \supp{\nu} \subseteq \rbb_+\}$. A sequence
$\gammab = \{\gamma_n\}_{n=0}^\infty \subseteq \rbb$
is said to be a {\em Hamburger} (resp.\ {\em
Stieltjes}) {\em moment sequence} if there exists $\nu
\in \msc$ (resp.\ $\nu \in \msc^+$) such that
   \begin{align*}
\gamma_n = \int_{\rbb} t^n \D \nu(t), \quad n \in
\zbb_+;
   \end{align*}
the set of all such measures, called {\em
H-representing} (resp.\ {\em S-representing}) measures
of $\gammab$, is denoted by $\msc(\gammab)$ (resp.\
$\msc^+(\gammab)$). A Hamburger (resp.\ Stieltjes)
moment sequence $\gammab$ is said to be {\em
H-determinate} (resp.\ {\em S-determinate}) if
$\card{\msc(\gammab)}=1$ (resp.\
$\card{\msc^+(\gammab)} = 1$); otherwise, we call it
{\em H-indeterminate} (resp.\ {\em S-indeterminate}).
We say that a measure $\nu \in \msc$ (resp.\ $\nu \in
\msc^+$) is H-determinate (resp.\ S-determinate) if
the sequence $\{\int_{\rbb} t^n
\D\nu(t)\}_{n=0}^\infty$ is H-determinate (resp.\
S-determinate). Similarly, we define H-indeterminacy
and S-indeterminacy of measures. Clearly, an
S-indeterminate Stieltjes moment sequence is
H-indeterminate. It is well-known that a Hamburger
moment sequence which has a compactly supported
H-representing measure is H-determinate (cf.\
\cite{fug}). Note that H-determinacy and S-determinacy
coincide for Stieltjes moment sequences having
S-representing measures vanishing on $\{0\}$ (see
\cite[Corollary, p.\ 481]{chi}; see also \cite[Lemma
2.2.5]{j-j-s0}).

   Let $\gammab=\{\gamma_n\}_{n=0}^\infty$ be a
Hamburger moment sequence. A measure $\nu \in
\msc(\gammab)$ is called an {\em N-extremal}
measure of $\gammab$ if $\gammab$ is
H-indeterminate and $\cbb[t]$ is dense in
$L^2(\nu)$. We say that $\nu \in \msc$ is an
N-extremal measure if $\nu$ is an N-extremal
measure of the Hamburger moment sequence
$\{\int_{\rbb} t^n \D \nu(t)\}_{n=0}^{\infty}$.
Denote by $\msc_e(\gammab)$ the set of all
N-extremal measures of $\gammab$ and put
$\msc_e^+(\gammab) = \msc_e(\gammab) \cap \msc^+$.

Note that if $\gammab=\{\gamma_n\}_{n=0}^\infty$ is an
H-indeterminate Hamburger moment sequence, then
$\card{\msc_e(\gammab)} = \mathfrak c$ (cf.\
\cite[Theorem 4 and Remark, p.\ 96]{sim}). Moreover,
we have:
   \begin{lem}[\mbox{\cite[Theorems 5 and 4.11]{sim}}]
\label{proN} If $\gammab=\{\gamma_n\}_{n=0}^\infty$ is
an H-indetermi\-nate Hamburger moment sequence, then
$\rbb = \bigsqcup_{\nu \in \msc_e(\gammab)}
\supp{\nu}$, and the closed support of any $\nu \in
\msc_e(\gammab)$ is countably infinite with no
accumulation point in $\rbb$.
   \end{lem}
Now we state the M. Riesz characterizations of
H-determinacy and N-extrem\-ality (see \cite[p.\
223]{M-R} or \cite[Theorem, p.\ 58]{fug}) and the
Berg-Thill characterization of S-determinacy (see
\cite[Theorem 3.8]{ber-thil1} or \cite[Proposition
1.3]{ber-thil2}).
   \begin{lem} \label{MRk}
{\em (i)} A measure $\nu\in \msc$ is H-determinate
$($resp.\ N-extremal\/$)$ if and only if $\cbb[t]$ is
dense in $L^2((1+t^2)\D \nu(t))$ $($resp.\ $\cbb[t]$
is dense in $L^2(\nu)$ and not dense in $L^2((1+t^2)\D
\nu(t))$$)$. {\em (ii)} A measure $\nu\in \msc^+$ is
S-determinate if and only if $\cbb[t]$ is dense in
both $L^2((1+t)\D \nu(t))$ and $L^2(t(1+t)\D \nu(t))$.
   \end{lem}
The above enables us to formulate a comparison test
for determinacy.
   \begin{pro}[Comparison test] \label{krytnd}
Let $\rho$ and $\nu$ be Borel measures on $\rbb$
$($resp.\ $\rbb_+$$)$ such that $\nu \in \msc$
$($resp.\ $\nu \in \msc^+$$)$ and $\rho(\sigma) \Le M
\, \nu(\sigma)$ for every $\sigma \in \borel{\rbb}$
$($resp.\ $\sigma \in \borel{\rbb_+}$$)$ and for some
$M \in \rbb_+$. Then $\rho\in \msc$ $($resp.\ $\rho\in
\msc^+$$)$. Moreover, $\rho$ is H-determinate
$($resp.\ S-determinate$)$ whenever $\nu$ is.
   \end{pro}
   \begin{proof}
We deal only with the case of H-determinacy; the other
case can be treated similarly. The standard
measure-theoretic argument implies that $\rho\in
\msc$. Since $\rho \Le M \, \nu$, we deduce from
\cite[Theorem 3.13]{Rud} that $L^2((1+t^2)\D\nu(t))
\ni f \mapsto f \in L^2((1+t^2)\D\rho(t))$ is a
well-defined bounded operator with dense range. Hence,
applying Lemma \ref{MRk} completes the proof.
   \end{proof}
   \begin{cor}\label{snieg}
Suppose that $\{\gamma_{1,n}\}_{n=0}^{\infty}$ and
$\{\gamma_{2,n}\}_{n=0}^{\infty}$ are Hamburger
$($resp.\ Stieltjes$)$ moment sequences such that
$\{\gamma_{1,n}+\gamma_{2,n}\}_{n=0}^{\infty}$ is
H-determinate $($resp.\ S-determinate$)$. Then both
$\{\gamma_{1,n}\}_{n=0}^{\infty}$ and
$\{\gamma_{2,n}\}_{n=0}^{\infty}$ are H-determinate
$($resp.\ S-determinate$)$.
   \end{cor}
   \begin{rem}  \label{snieg2}
It follows from Corollary \ref{snieg} that if
$\{\gamma_n\}_{n=0}^{\infty}$ is a Stieltjes moment
sequence such that $\{\gamma_n+c\}_{n=0}^{\infty}$ is
S-determinate for some $c\in (0,\infty)$, then
$\{\gamma_n\}_{n=0}^{\infty}$ is S-determinate. This
may suggest that if $\{\gamma_n\}_{n=0}^\infty$ is an
S-determi\-nate Stieltjes moment sequence, then so is
$\{\gamma_n + c\}_{n=0}^\infty$ for some $c \in
(0,\infty)$. However, in general, this is not true. In
fact, one can show more; namely there exists an
H-determinate Stieltjes moment sequence
$\{\gamma_n\}_{n=0}^{\infty}$ such that
$\{\gamma_n+c\}_{n=0}^{\infty}$ is S-indeterminate for
all $c\in (0,\infty)$. Indeed, as noticed by C. Berg
(private communication), if $\nu$ is an N-extremal
measure of an S-indeterminate Stieltjes moment
sequence such that $\inf\supp{\nu} = 1$ (e.g., the
orthogonalizing measure $\beta^{(a;q)}$ for the
Al-Salam-Carlitz \mbox{$q$-polynomials}
$\{V^{(a)}_{n}(x;q)\}_{n=0}^{\infty}$ with $0 < q < a
\Le 1$ is N-extremal and its closed support is equal
to $\{q^{-n}\}_{n=0}^\infty$, see Section
\ref{SecASC}), then the measure $\mu:=\nu - \nu(\{1\})
\delta_{1} \in \msc^+$ is H-determinate and for every
$c \in (0, \infty)$, the measure $\mu + c \delta_{1}$
is N-extremal (see \cite[Theorem 3.6 and Lemma
3.7]{ber-dur}) and consequently, since $\inf\supp{\mu
+ c \delta_{1}} > 0$, it is S-indeterminate (see
\cite[Corollary, p.\ 481]{chi} or \cite[Lemma
2.2.5]{j-j-s0}).
   \end{rem}
The following lemma will be used in Section
\ref{sect4.3}.
   \begin{lem} \label{kontr}
If $\{\gamma_n\}_{n=0}^\infty \subseteq \rbb_+$, then
the following conditions are equivalent{\em :}
   \begin{enumerate}
   \item[(i)] $\{\gamma_n\}_{n=0}^\infty$ is a Stieltjes
moment sequence which has an S-representing measure
vanishing on $[0,1)$,
   \item[(ii)] $0 \Le \sum_{i,j=0}^n \gamma_{i+j} \lambda_i \bar
       \lambda_j
\Le \sum_{i,j=0}^n \gamma_{i+j+1} \lambda_i \bar
\lambda_j$ for all finite sequences
$\{\lambda_i\}_{i=0}^n$ of complex numbers,
   \item[(iii)] $\{\gamma_n\}_{n=0}^\infty$ is a Stieltjes moment
sequence and $\sum_{i,j=0}^n (\gamma_{i+j+1} -
\gamma_{i+j}) \lambda_i \bar \lambda_j \Ge 0$ for all
finite sequences $\{\lambda_i\}_{i=0}^n$ of complex
numbers.
   \end{enumerate}
   \end{lem}
   \begin{proof}
(i)$\Rightarrow$(ii) Obvious.

(ii)$\Leftrightarrow$(iii) Apply the Stieltjes theorem
(cf.\ \cite[Theorem 6.2.5]{b-c-r}).

(ii)$\Rightarrow$(i) Let $\varLambda\colon \cbb[t] \to
\cbb$ be a linear functional such that
$\varLambda(t^n) = \gamma_n$ for all $n \in \zbb_+$.
Take $p\in \cbb[t]$ which is nonnegative on
$[1,\infty)$. Since $p(t+1)$ is nonnegative on
$[0,\infty)$, there exist $q_1, q_2 \in \cbb[t]$ such
that $p = (t-1)|q_1|^2 + |q_2|^2$ (cf.\ \cite[Problem
45, p.\ 78]{P-S}). Hence $\varLambda(p) \Ge 0$.
Applying the Riesz-Haviland theorem (cf.\ \cite{Hav2})
completes the proof.
   \end{proof}
The question of when $\msc_e^+(\gammab)$ is nonempty
has the following answer.
   \begin{lem} \label{taso+}
Suppose $\gammab$ is an H-indeterminate Stieltjes
moment sequence. Then $\msc_e^+(\gammab) \neq
\emptyset$. Moreover, $\gammab$ is S-determinate if
and only if $\card{\msc_e^+(\gammab)} = 1$.
   \end{lem}
   \begin{proof}
Let $A$ be a symmetric operator in a complex Hilbert
space $\hh$ and $e \in \dzn{A}$ be such that $\dz{A}$
is the linear span of $\{A^n e \colon n \in \zbb_+\}$,
and $\gamma_n = \is{A^n e}{e}$ for all $n\in \zbb_+$
(cf.\ \cite[(1.10)]{sim}). By assumption, $A$ is a
positive operator with deficiency indices $(1,1)$
(cf.\ \cite[Theorem 2 and Corollary 2.9]{sim}). Hence,
the Friedrichs extension $S$ of $A$ differs from $\bar
A$ (cf.\ \cite[Theorem 5.38]{Weid}). As a consequence,
$\is{E(\cdot)e}{e} \in \msc_e^+(\gammab)$, where $E$
is the spectral measure of $S$ (cf.\ \cite[p.\
3951]{j-j-s0}). This also proves necessity in the
``moreover'' part. The sufficiency is a direct
consequence of \cite[Theorem 4]{sim}.
   \end{proof}
Recall that if $\gammab=\{\gamma_n\}_{n=0}^\infty$ is
an S-indeterminate Stieltjes moment sequence, then
$\card{\msc^+_e(\gammab)} = \mathfrak c$ and there
exist distinct measures $\alpha, \beta \in
\msc^+_e(\gammab)$ (uniquely determined by
\eqref{KFm}) such that for every $\rho \in
\msc_e^+(\gammab) \setminus \{\alpha, \beta\}$,
   \begin{align} \label{KFm}
0 = \inf\supp{\alpha} < \inf \supp{\rho} < \inf
\supp{\beta};
   \end{align}
$\alpha$ and $\beta$ are called the {\em Krein} and
the {\em Friedrichs} measures of $\gammab$,
respectively. These two particular N-extremal measures
come from the Krein and the Friedrichs extensions of a
positive operator attached to $\gammab$. The reader is
referred to \cite{Ped} for the case of Friedrichs
extensions and to \cite{sim} for a complete and
up-to-date operator approach to moment problems (see
also \cite[Section 2]{j-j-s0}).
   \subsection{Transforming moments via homotheties}
   \label{Sec2.3} In this section we investigate
transformations acting on real sequences induced by
homotheties of $\rbb$. Such transformations are shown
to preserve many properties of Hamburger and Stieltjes
moment sequences. The particular case of
transformations induced by translations has been
considered in \cite[Section 3]{Ped} and \cite[p.\
96]{sim} (with different approaches).

Fix $\vartheta \in (0,\infty)$ and $a \in \rbb$. Let
us define the self-map $\psi_{\vartheta,a}$ of $\rbb$
by
   \begin{align} \label{defpsi}
\psi_{\vartheta,a}(t) = \vartheta(t + a), \quad
t\in\rbb.
   \end{align}
Note that $\psi_{\vartheta,a}$ is a strictly
increasing homeomorphism of $\rbb$ onto itself such
that
   \begin{align} \label{t1-1}
\quad \psi_{1,0} = \mathrm{id}_{\rbb}, \quad
\psi_{\tilde \vartheta,\tilde a} \circ
\psi_{\vartheta,a} = \psi_{\tilde \vartheta
\vartheta,\frac{\tilde a}{\vartheta} + a}, \quad
\psi_{\vartheta,a}^{-1} =
\psi_{\frac{1}{\vartheta},-a\vartheta}
   \end{align}
for all $\tilde \vartheta \in (0,\infty)$ and $\tilde
a \in \rbb$. Next, we define the linear self-map
$T_{\vartheta,a}$ of $\rbb^{\zbb_+}$ by
   \begin{align} \label{T1}
(T_{\vartheta,a}\gammab)_n = \sum_{j=0}^n \binom{n}{j}
a^{n-j} \vartheta^n \gamma_j, \quad n\in \zbb_+, \,
\gammab = \{\gamma_n\}_{n=0}^\infty \subseteq \rbb.
   \end{align}

The proof of Lemma \ref{hg} below, being elementary,
is omitted.
   \begin{lem} \label{hg}
The following hold for all $\vartheta, \tilde
\vartheta \in (0,\infty)$ and $a, \tilde a \in
\rbb${\em :} \allowdisplaybreaks
   \begin{align} \notag
& \text{$T_{\vartheta,a}$ is a bijection of
$\rbb^{\zbb_+}$ onto itself,}
   \\ \label{prod}
& T_{1,0} = \mathrm{id}_{\rbb^{\zbb_{+}}}, \quad
T_{\tilde \vartheta,\tilde a} T_{\vartheta,a} =
T_{\tilde \vartheta \vartheta,\frac{\tilde
a}{\vartheta} + a}, \quad T_{\vartheta,a}^{-1} =
T_{\frac{1}{\vartheta},-a\vartheta}.
   \end{align}
   \end{lem}
In view of \eqref{t1-1} and Lemma \ref{hg}, the
correspondence $\psi_{\vartheta,a} \mapsto
T_{\vartheta,a}$ defines a faithful representation of
the group of all strictly increasing homotheties of
$\rbb$. Moreover, by \eqref{T1} and \eqref{prod}, we
have, for all $\vartheta \in (0,\infty)$ and $a \in
\rbb$,
   \begin{align*}
(T_{\vartheta,a}^{-1}\gammab)_n = \sum_{j=0}^n
\binom{n}{j} (-a)^{n-j} \vartheta^{-j} \gamma_j, \quad
n\in \zbb_+, \, \gammab = \{\gamma_n\}_{n=0}^\infty
\subseteq \rbb.
   \end{align*}

In Lemma \ref{transH} and Theorem \ref{taso} below we
state properties of $T_{\vartheta,a}$ which are
relevant for further considerations. If $\nu$ is a
Borel measure on $\rbb$ and $\varphi$ is a
homeomorphism of $\rbb$ onto itself, then $\nu\circ
\varphi$ is the Borel measure on $\rbb$ given by
   \begin{align} \label{dnf}
\nu\circ \varphi(\sigma) = \nu(\varphi(\sigma)), \quad
\sigma \in \borel{\rbb}.
   \end{align}
   \begin{lem}\label{transH}
Let $\vartheta \in (0,\infty)$ and $a \in \rbb$. Then
   \begin{enumerate}
   \item[(i)] $T_{\vartheta,a}$ is a self-bijection
on the set of all Hamburger moment sequences,
   \item[(ii)] if $\gammab$ is a Hamburger moment sequence,
then the mapping
   \begin{align} \label{hurra}
\msc(\gammab) \ni \nu \mapsto \nu \circ
\psi_{\vartheta,a}^{-1} \in
\msc(T_{\vartheta,a}\gammab)
   \end{align}
is a well-defined bijection with the inverse given by
   \begin{align*}
\msc(T_{\vartheta,a}\gammab) \ni \nu \mapsto \nu \circ
\psi_{\vartheta,a} \in \msc(\gammab);
   \end{align*}
in particular, $\gammab$ is H-determinate if and only
if $T_{\vartheta,a}\gammab$ is H-determinate,
   \item[(iii)] if $\gammab$ is an H-indeterminate
Hamburger moment sequence, then so is
$T_{\vartheta,a}\gammab$ and the mapping defined by
\eqref{hurra} maps $\msc_e(\gammab)$ onto
$\msc_e(T_{\vartheta,a}\gammab)$,
   \item[(iv)]  if $\gammab$ is a nonzero Hamburger moment
sequence and $\nu \in \msc(\gammab)$, then
   \begin{gather} \label{2a}
\supp{\nu \circ \psi_{\vartheta,a}^{-1}} =
\psi_{\vartheta,a}(\supp{\nu}),
   \\ \label{2aa}
\inf \supp{\nu \circ \psi_{\vartheta,a}^{-1}} =
\psi_{\vartheta,a}(\inf\supp{\nu}),
   \end{gather}
with convention that
$\psi_{\vartheta,a}(-\infty)=-\infty$,
   \item[(v)] if  $a \Ge 0$, $\gammab$ is an
H-indeterminate Stieltjes moment sequence and
$\nu\in\msc_e^+(\gammab)$, then
$T_{\vartheta,a}\gammab$ is an H-indeterminate
Stieltjes moment sequence and $\nu \circ
\psi_{\vartheta,a}^{-1} \in
\msc_e^+(T_{\vartheta,a}\gammab)$.
   \end{enumerate}
   \end{lem}
   \begin{proof}
(i)\&(ii) If $\gammab$ is a Hamburger moment sequence
and $\nu \in \msc(\gammab)$, then, by the measure
transport theorem (cf.\ \cite[Theorem 1.6.12]{Ash}),
we have
   \begin{align*}
(T_{\vartheta,a}\gammab)_n = \int_{\rbb}
\big(\psi_{\vartheta,a}(t)\big)^n \D \nu(t) =
\int_{\rbb} t^n \D \nu\circ
\psi_{\vartheta,a}^{-1}(t), \quad n \in \zbb_+,
   \end{align*}
which means that $T_{\vartheta,a}\gammab$ is a
Hamburger moment sequence and
$\nu\circ\psi_{\vartheta,a}^{-1} \in
\msc(T_{\vartheta,a}\gammab)$. The above combined with
\eqref{t1-1} and \eqref{prod} completes the proof of
(i) and (ii).

(iii) Let $\gammab$ be a Hamburger moment sequence and
$\nu \in \msc(\gammab)$. Since
   \begin{align*}
\sup\nolimits_ {t\in\rbb} (1+\varphi(t)^2)/(1+t^2) <
\infty, \quad \varphi \in \big\{\psi_{\vartheta,a},
\psi_{\vartheta,a}^{-1}\big\},
   \end{align*}
we deduce from the measure transport theorem that the
mapping
   \begin{align*}
W\colon L^2((1+t^2)\D\nu\circ
\psi_{\vartheta,a}^{-1}(t)) \ni f \to f\circ
\psi_{\vartheta,a} \in L^2((1+t^2)\D\nu(t))
   \end{align*}
is a well-defined linear homeomorphism (with the
inverse $g \to g\circ \psi_{\vartheta,a}^{-1}$) such
that $W(\cbb[t]) = \cbb[t]$. Similarly, $V\colon
L^2(\nu\circ \psi_{\vartheta,a}^{-1}) \ni f \to f\circ
\psi_{\vartheta,a} \in L^2(\nu)$ is a unitary
isomorphism such that $V(\cbb[t]) = \cbb[t]$. This,
(ii) and Lemma \ref{MRk}(i) yield (iii).

(iv) The equality \eqref{2a} is a direct consequence
of the definition of the closed support of a measure
(see also \cite[Lemma 3.2]{2xSt2} for a more general
result). Clearly, \eqref{2a} implies \eqref{2aa}.

(v) Apply (iii) and (iv).
   \end{proof}
   \begin{thm} \label{taso}
Let $\vartheta\in (0,\infty)$ and $a \in \rbb$.
Suppose $\gammab=\{\gamma_n\}_{n=0}^\infty$ is an
S-indeterminate Stieltjes moment sequence and $\beta$
is its Friedrichs measure. Then $T_{\vartheta,a}
\gammab$ is an H-indeterminate Hamburger moment
sequence and the following holds{\em :}
   \begin{enumerate}
   \item[(i)] if $c > 0$,
then $T_{\vartheta,a} \gammab$ is an S-indeterminate
Stieltjes moment sequence and $\beta\circ
\psi_{\vartheta,a}^{-1}$ is the Friedrichs measure of
$T_{\vartheta,a} \gammab$,
   \item[(ii)]
if $c = 0$, then $T_{\vartheta,a} \gammab$ is an
S-determinate Stieltjes moment sequence,
   \item[(iii)] if $c < 0$,
then $T_{\vartheta,a} \gammab$ is not a Stieltjes
moment sequence,
   \end{enumerate}
where $c:= \psi_{\vartheta,a} \big(\inf
\supp{\beta}\big)$.
   \end{thm}
   \begin{proof}
By Lemma \ref{transH}, $T_{\vartheta,a} \gammab$ is an
H-indeterminate Hamburger moment sequence such that
   \begin{align} \label{dzth}
\beta \circ \psi_{\vartheta,a}^{-1} \in
\msc_e(T_{\vartheta,a}\gammab) \text{ and }
\inf\supp{\beta \circ \psi_{\vartheta,a}^{-1}} =
\psi_{\vartheta,a}(\inf\supp{\beta}) = c.
   \end{align}
Note that
   \begin{align} \label{kr1}
\inf\supp{\beta \circ \psi^{-1}_{\vartheta,a}} \Ge
\inf\supp{\rho}, \quad \rho \in
\msc_e(T_{\vartheta,a}\gammab).
   \end{align}
Indeed, otherwise, by \eqref{dzth} and Lemma
\ref{transH}, $\inf\supp{\beta} < \inf \supp{\rho
\circ \psi_{\vartheta,a}}$ and $\rho \circ
\psi_{\vartheta,a}\in \msc_e^+(\gammab)$, which
contradicts \eqref{KFm}.

(i) By \eqref{dzth} and \cite[Corollary, p.\ 481]{chi}
(see also \cite[Lemma 2.2.5]{j-j-s0}),
$T_{\vartheta,a} \gammab$ is an S-indetermi\-nate
Stieltjes moment sequence. Denote by $\rho$ its
Friedrichs measure. If $\beta\circ
\psi_{\vartheta,a}^{-1} \neq \rho$, then by
\eqref{KFm} and \eqref{dzth}, $\inf\supp{\beta \circ
\psi_{\vartheta,a}^{-1}} < \inf \supp{\rho}$, which
would contradict \eqref{kr1}.

(ii) By \eqref{dzth}, $T_{\vartheta,a} \gammab$ is a
Stieltjes moment sequence and $\beta \circ
\psi^{-1}_{\vartheta,a} \in
\msc_e^+(T_{\vartheta,a}\gammab)$. If
$T_{\vartheta,a}\gammab$ were S-indeterminate and
$\rho$ were its Friedrichs measure, then by
\eqref{dzth}, $\beta \circ \psi^{-1}_{\vartheta,a}
\neq \rho$, which, as in (i), would contradict
\eqref{kr1}.

(iii) If $T_{\vartheta,a} \gammab$ were a Stieltjes
moment sequence, then by Lemma \ref{taso+}, there
would exist $\rho \in
\msc_e^+(T_{\vartheta,a}\gammab)$, and thus by
\eqref{dzth}, $\inf\supp{\beta \circ
\psi^{-1}_{\vartheta,a}} < \inf\supp{\rho}$, which
would contradict \eqref{kr1}.
   \end{proof}
   \begin{cor} \label{transH+}
Let $\vartheta, a \in (0,\infty)$. Suppose $\gammab$
is an S-indeterminate Stieltjes moment sequence and
$\alpha$ and $\beta$ are its Krein and Friedrichs
measures, respectively. Then $T_{\vartheta,a} \gammab$
is an S-indeterminate Stieltjes moment sequence,
$\alpha\circ \psi_{\vartheta,a}^{-1} \in
\msc_e^+(T_{\vartheta,a} \gammab)$, $\beta\circ
\psi_{\vartheta,a}^{-1}$ is the Friedrichs measure of
$T_{\vartheta,a} \gammab$ and
   \begin{align} \label{ineqKF}
0 < \inf\supp{\alpha\circ \psi_{\vartheta,a}^{-1}} <
\inf \supp{\beta\circ \psi_{\vartheta,a}^{-1}}.
   \end{align}
In particular, $\alpha\circ \psi_{\vartheta,a}^{-1}$
is neither the Krein nor the Friedrichs measure of
$T_{\vartheta,a} \gammab$.
   \end{cor}
   \begin{proof}
In view of \eqref{KFm} and Lemma \ref{transH},
\eqref{ineqKF} holds and $\alpha\circ
\psi_{\vartheta,a}^{-1} \in
\msc_e^+(T_{\vartheta,a}\gammab)$. This together with
\eqref{KFm} and Theorem \ref{taso} completes the
proof.
   \end{proof}
The particular case of Theorem \ref{taso} with
$\vartheta=1$ (without the statement that $\beta\circ
\psi_{1,a}^{-1}$ is the Friedrichs measure of $T_{1,a}
\gammab = \gammab(a)$) appeared in \cite[Theorem
3.3]{sim} with a very brief outline of the proof based
on the von Neumann theory of selfadjoint extensions of
symmetric operators. In turn, the particular case of
Corollary \ref{transH+} with $\vartheta=1$ (without
any statement on the Krein measure) appeared in
\cite[Proposition 3.5]{Ped} with another approach
based on the Nevanlinna parametrization.

It follows from Corollary \ref{transH+} that if
$\vartheta, a > 0$, then the transformation
$T_{\vartheta,a}$ preserves S-indeterminate Stieltjes
moment sequences, and the corresponding mapping
defined by \eqref{hurra} preserves the Friedrichs
measures (but never the Krein ones). The situation
changes drastically when $a < 0$ (for example, when we
consider $T^{-1}_{\vartheta,b}$ with $b>0$; see
\eqref{prod}). This is because the quantity
$c=\psi_{\vartheta,a} (\inf \supp{\beta})$ may happen
to be negative (cf.\ Theorem \ref{taso}).

The above-mentioned properties of self-maps
$T_{\vartheta,a}$ enable us to parameterize N-extremal
measures of H-indeterminate Stieltjes moment sequences
in a new way.
   \begin{thm} \label{tinfsup}
Suppose $\gammab=\{\gamma_n\}_{n=0}^\infty$ is an
H-indeterminate Stieltjes moment sequence. Set
$t_0=\inf \supp{\beta}$, where $\beta$ is either the
Friedrichs measure of $\gammab$ if $\gammab$ is
S-indeterminate, or $\beta$ is the unique
S-representing measure of $\gammab$ otherwise. Then
   \begin{enumerate}
   \item[(i)] if $\gamma$ is S-determinate, then
$t_0=0$ and $\beta \in \msc_e^+(\gammab)$,
   \item[(ii)] if $\gamma$ is S-indeterminate, then
$t_0 > 0$,
   \item[(iii)] for every $t \in (-\infty,t_0]$ there
exists a unique $\nu_t \in \msc_e(\gammab)$ such that
   \begin{align} \label{infsup}
\inf \supp{\nu_t} = t,
   \end{align}
   \item[(iv)] the mapping $(-\infty,t_0] \ni t \mapsto
\nu_t \in \msc_e(\gammab)$ is a bijection,
   \item[(v)] $\msc_e^+(\gammab) = \{\nu_t\colon t \in
[0,t_0]\}$,
   \item[(vi)] the closed support of each N-extremal measure
of $\gammab$ is bounded from below,
   \item[(vii)] $\supp{\nu_t}
\cap (-\infty,t_0]=\{t\}$ for every $t\in
(-\infty,t_0]$,
   \item[(viii)] $[0,t_0] \varsubsetneq
\bigcup_{\mu \in \msc_e^+(\gammab)} \supp{\mu}$ and
$\mathrm{card}\big(\rbb_+ \setminus \bigcup_{\mu \in
\msc_e^+(\gammab)} \supp{\mu}\big) = \mathfrak c$.
   \end{enumerate}
   \end{thm}
   \begin{proof}
Assume that $\gammab$ is S-determinate. Then by
\cite[Corollary, p.\ 481]{chi} (see also \cite[Lemma
2.2.5]{j-j-s0}), $0 \in \supp{\beta}$ and thus
$t_0=0$. In view of Lemma \ref{taso+}, $\beta \in
\msc_e^+(\gammab)$ and the measure $\nu_0:=\beta$
satisfies \eqref{infsup}. Now take $t \in
(-\infty,0)$. Then by Lemma \ref{transH}, the sequence
$T_{1,-t}\gammab$ is H-indeterminate, $\beta\circ
\psi_{1,-t}^{-1} \in \msc_e(T_{1,-t}\gammab)$ and
$\inf\supp{\beta\circ \psi_{1,-t}^{-1}} = |t| > 0$.
Applying \cite[Corollary, p.\ 481]{chi} again, we see
that $T_{1,-t}\gammab$ is S-indeterminate. Let
$\rho_t$ be the Krein measure of $T_{1,-t}\gammab$.
Since $\gammab = T_{1,t}T_{1,-t}\gammab$, we infer
from Lemma \ref{transH} and \eqref{KFm} that
$\nu_t:=\rho_t \circ \psi_{1,t}^{-1} \in
\msc_e(\gammab)$ and $\inf\supp{\nu_t} = t$.

   Assume now that $\gammab$ is S-indeterminate. Let
$t\in (-\infty,t_0)$. Since
   \begin{align*}
c:=\psi_{1,-t}(\inf\supp{\beta}) = t_0-t > 0,
   \end{align*}
we deduce from Theorem \ref{taso}(i) that
$T_{1,-t}\gammab$ is S-indeterminate. Taking the Krein
measure $\rho_t$ of $T_{1,-t}\gammab$ and arguing as in
the previous paragraph, we see that $\nu_t:=\rho_t \circ
\psi_{1,t}^{-1} \in \msc_e(\gammab)$ satisfies
\eqref{infsup}. If $t=t_0$, then $\nu_{t_0}:=\beta$ does
the job.

In both cases, S-determinate and S-indeterminate, the
uniqueness of $\nu_t \in \msc_e(\gammab)$ satisfying
\eqref{infsup} follows from Lemma \ref{proN}.
Altogether this proves (i), (ii) and (iii). Clearly,
by \eqref{infsup}, the mapping $(-\infty,t_0] \ni t
\mapsto \nu_t \in \msc_e(\gammab)$ is injective. To
prove its surjectivity, take $\nu \in
\msc_e(\gammab)$. By Lemma \ref{proN} and \eqref{KFm},
there exists $t \in (-\infty,t_0] \cap \supp{\nu}$.
Then $t\in \supp{\nu_{t}} \cap \supp{\nu}$ and so, by
Lemma \ref{proN}, $\nu_{t}=\nu$. This proves (iv) and
consequently (v) and (vi). The condition (vii) is a
direct consequence of (iii) and Lemma \ref{proN}. To
prove (viii) take any $t\in (-\infty,0)$. By (iii) and
Lemma \ref{proN}, $\supp{\nu_t} \cap (t_0,\infty)$ is
a nonempty (in fact, a countably infinite) subset of
$\rbb_+$ which is disjoint with $\bigcup_{\mu \in
\msc_e^+(\gammab)} \supp{\mu}$ and the latter set
being unbounded properly contains $[0,t_0]$. It
follows from (iv) and Lemma \ref{proN} that the sets
$\supp{\nu_t} \cap (t_0,\infty)$, $t \in (-\infty,0)$,
are nonempty and disjoint. This completes the proof of
(viii) and the theorem.
   \end{proof}
   We conclude this section by stating the following
unexpected trichotomy property of N-extremal measures
of H-indeterminate Hamburger moment sequences.
   \begin{thm}[Trichotomy] \label{trichotomy}
Suppose that $\gammab=\{\gamma_n\}_{n=0}^\infty$ is an
H-indeterminate Hamburger moment sequence. Then
exactly one of the following three conditions
holds{\em :}
   \begin{enumerate}
   \item[(i)] for every $\nu\in \msc_e(\gammab)$,
$\inf\supp{\nu}=-\infty$ and $\sup\supp{\nu}=\infty$,
   \item[(ii)] for every $\nu\in \msc_e(\gammab)$,
$\inf\supp{\nu}> -\infty$ and $\sup\supp{\nu}=\infty$,
   \item[(iii)] for every $\nu\in \msc_e(\gammab)$,
$\inf\supp{\nu}=-\infty$ and $\sup\supp{\nu}< \infty$.
   \end{enumerate}
   \end{thm}
   \begin{proof} Suppose that  (i) does not hold.
Then there exists $\nu\in \msc_e(\gammab)$ such that
either $s:=\inf\supp{\nu}> -\infty$ or $\sup\supp{\nu}
< \infty$. In the former case, $T_{1,-s}\gammab$ is an
H-indeterminate Stieltjes moment sequence (cf.\ Lemma
\ref{transH}), and thus applying Theorem \ref{tinfsup}
to $T_{1,-s}\gammab$ and returning to $\gammab$ via
the self-map $T_{1,s}$ we get (ii). Considering the
H-indeterminate Hamburger moment sequence
$\widetilde\gammab:=\{(-1)^n\gamma_n\}_{n=0}^\infty$,
one can show that the latter case leads to (iii) (as
in the proof of Lemma \ref{transH}, we verify that the
mapping $\msc(\gammab) \ni \nu \mapsto \nu \circ
\psi^{-1} \in \msc(\widetilde\gammab)$ is a bijection
which maps $\msc_e(\gammab)$ onto
$\msc_e(\widetilde\gammab)$ and $\supp{\nu \circ
\psi^{-1}} = \psi(\supp{\nu})$ for every $\nu\in
\msc(\gammab)$, where $\psi$ is a self-map of $\rbb$
given by $\psi(t)=-t$ for $t \in \rbb$). This
completes the proof.
   \end{proof}
We refer the reader to \cite[pages 93 and 94]{Ism-Mas}
(see also \cite[Theorem 21.5.3]{Ism-book}), where all
N-extremal measures of the H-indeterminate Hamburger
moment sequence arising from \mbox{$q^{-1}$-Hermite}
polynomials with $q \in (0,1)$ are explicitly
calculated. Since their closed supports are bounded
neither from below nor from above, we see that none of
the conditions (i), (ii) and (iii) of Theorem
\ref{trichotomy} is redundant.
   \subsection{The Al-Salam-Carlitz moment problem}\label{SecASC}
   Orthogonal $q$-polyno\-mials introduced by Al-Salam
and Carlitz in \cite{als-c} give rise to examples of
S-indetermi\-nate Stieltjes moment sequences for which
some particular N-extremal measures are explicitly
known. These special measures help us to show that the
directed graph $\gcal_{2,0}$ admits a non-hyponormal
composition operator generating Stieltjes moment
sequences (cf.\ Theorem \ref{czwarte}).

We begin by recalling the definition of
\mbox{$q$-Pochhammer} symbol (called also
\mbox{$q$-shifted} factorial). For $z, q\in \cbb$, we
write
   \begin{align} \label{Poch}
(z;q)_n =
   \begin{cases}
1 & \text{ if } n=0,
   \\[.5ex]
\prod_{j=1}^n (1-zq^{j-1}) & \text{ if } n \in \nbb,
   \\[.7ex]
\prod_{j=1}^{\infty} (1-zq^{j-1}) & \text{ if }
n=\infty, |q| < 1.
   \end{cases}
   \end{align}
(See \cite[Section VII.1]{Sa-Zy} for more on infinite
products.) Following \cite[Section VI.10]{chi-b}, we
extend the original definition of $q$-polynomials of
Al-Salam and Carlitz to cover the case of $|q|>1$.
Given $a\in \cbb$ and $q\in \cbb \setminus \{0\}$, we
define $\{V^{(a)}_{n}(x;q)\}_{n=0}^{\infty}$, the
sequence of complex polynomials in one variable $x$,
by the recurrence formula
   \begin{align} \label{3termv}
   \begin{gathered}
V^{(a)}_{n+1}(x;q) = \Big(x-\frac{1+a}{q^n}\Big) V^{(a)}_{n}(x;q)
- a \frac{1-q^{n}}{q^{2n-1} } V^{(a)}_{n-1}(x;q), \quad n \in
\zbb_+,
   \\
V^{(a)}_{-1}(x;q) = 0, \qquad V^{(a)}_{0}(x;q) = 1.
   \end{gathered}
   \end{align}
The generating function for
$\{V^{(a)}_{n}(x;q)\}_{n=0}^{\infty}$ can be described
as follows (see \cite{als-c} and \cite[Section
VI.10]{chi-b}). If $a,q,z \in\cbb$, $x\in \rbb$ and
$q\neq 0$, then
   \begin{gather} \label{gf2}
   \sum_{n=0}^{\infty} V^{(a)}_n(x;q)
\frac{(-1)^nq^{\frac{n(n-1)}{2}}z^n}{(q;q)_n} =
   \begin{cases}
\frac{(xz;q)_{\infty}}{(z;q)_{\infty}(az;q)_{\infty}}
& \text{ if } |z|< r_a, |q|<1,
   \\[1ex]
\frac{\big(z\frac{1}{q};\frac{1}{q}\big)_{\infty}
\big(z\frac{a}{q};\frac{1}{q}\big)_{\infty}}
{\big(z\frac{x}{q}; \frac{1}{q}\big)_{\infty}} &
\text{ if } |z|< \frac{|q|}{|x|}, |q| >1,
   \end{cases}
   \end{gather}
where $r_a = \min\{1,\frac{1}{|a|}\}$ with the
convention that $\frac{1}{0}=\infty$. The function of
one complex variable $z$ given by the right-hand side
of \eqref{gf2} is called the generating function for
$\{V^{(a)}_{n}(x;q)\}_{n=0}^{\infty}$. Clearly, it is
meromorphic and has simple poles because
$(z;q)_{\infty}=0$ if and only if $1-zq^n=0$ for some
(unique) $n\in \zbb_+$.

By \eqref{3termv},
$\{V^{(a)}_{n}(x;q)\}_{n=0}^{\infty}$ satisfies the
following general recurrence relation
   \begin{align} \label{cokolw}
   \begin{gathered} P_{n+1}(x) = (x-c_n) P_n(x) - \lambda_n
P_{n-1}(x), \quad n \in \zbb_+,
   \\
P_{-1}(x) = 0, \qquad P_0(x)=1,
   \end{gathered}
   \end{align}
where $\{c_n\}_{n=0}^{\infty}$ and
$\{\lambda_n\}_{n=1}^{\infty}$ are sequences of
complex numbers ($\lambda_0$ can be chosen
arbitrarily) and $\{P_n(x)\}_{n=0}^{\infty}$ are
polynomials. Suppose that \eqref{cokolw} holds. Then
$\{P_n(x)\}_{n=0}^{\infty}$ is a Hamel basis of
$\cbb[x]$ and thus there exists a unique linear
functional $L\colon \cbb[x] \to \cbb$ such that
$L(x^0)=1$ and $L(P_n)=0$ for all $n\in \nbb$ (or
equivalently if and only if $L(x^0)=1$ and $L(P_m
P_n)=0$ for all $m,n \in \zbb_+$ such that $m\neq n$).
We say that $\mu \in \msc$ is an {\em orthogonalizing
measure} for $\{P_n(x)\}_{n=0}^{\infty}$ if $L(P) =
\int_{\rbb} P \D \mu$ for all $P \in \cbb[x]$. If this
is the case, then $\{L(x^n)\}_{n=0}^{\infty}$ is a
Hamburger moment sequence and $\mu$ is its
H-representing measure (clearly $\mu(\rbb) = 1$). By
Favard's theorem (cf.\ \cite[Theorems I.4.4 and
II.3.1]{chi-b}), the polynomials
$\{P_n(x)\}_{n=0}^{\infty}$ have an orthogonalizing
measure if and only if $c_n \in \rbb$ and
$\lambda_{n+1} > 0$ for all $n\in \zbb_+$.

Applying the above, we obtain the following
statement.
   \begin{align} \label{othm2}
   \begin{minipage}{72ex} {\em The polynomials
$\{V^{(a)}_{n}(x;q)\}_{n=0}^{\infty}$ have an
orthogonalizing measure if and only if either $a < 0$
and $q\in (-1,0) \cup (1,\infty)$, or $a > 0$ and $q
\in (0,1)$.}
   \end{minipage}
   \end{align}
As in \cite{als-c} we concentrate on the case of $q
\in (0,1)$. Then, by \eqref{othm2}, the polynomials
$\{V^{(a)}_{n}(x;q)\}_{n=0}^{\infty}$ have an
orthogonalizing measure if and only if $a>0$. For such
$a$ and $q$, the question of determinacy of
orthogonalizing measures for
$\{V^{(a)}_{n}(x;q)\}_{n=0}^{\infty}$ can be answered
completely. This is done below. Known orthogonalizing
measures for $\{V^{(a)}_{n}(x;q)\}_{n=0}^{\infty}$ are
discussed in detail as well. We refer the reader to
Figure $1$ which illustrates how determinacy depends
on the parameters $a$ and $q$.
  \vspace{1ex}
   \begin{center}
   \includegraphics[width=10cm]
   {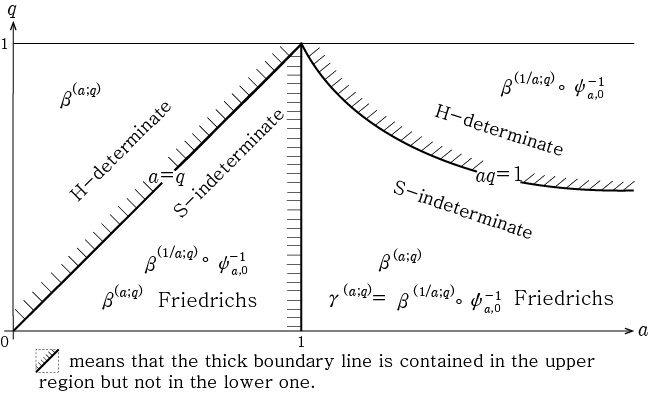}
   \\[1ex]
   \begin{minipage}{72ex}
{\small {\sf Figure $1$.} Determinacy of known
orthogonalizing measures for
$\{V^{(a)}_{n}(x;q)\}_{n=0}^{\infty}$.}
   \end{minipage}
   \end{center}
   \vspace{2ex}

Al-Salam and Carlitz showed in \cite{als-c} that if
$a>0$ and $aq<1$, then the measure
   \begin{align} \label{betaq}
\beta^{(a;q)} := (aq;q)_{\infty} \sum_{n=0}^{\infty}
\frac{a^n q^{n^2}}{(aq;q)_n (q;q)_n} \delta_{q^{-n}},
   \end{align}
is an orthogonalizing measure for
$\{V^{(a)}_{n}(x;q)\}_{n=0}^{\infty}$ (if $aq \Ge 1$,
then the right-hand side of \eqref{betaq} either does
not make sense or does not define a positive
measure),~ and
   \begin{align*}
{\sf m}_n[\beta^{(a;q)}]:=\int_{\rbb} t^n \D
\beta^{(a;q)}(t) = \sum_{k=0}^{n} \frac{(q;q)_n
q^{k(k-n)}}{(q;q)_k (q;q)_{n-k}} a^k, \quad n\in
\zbb_+.
   \end{align*}
(That $\beta^{(a;q)}$ is a probability measure was
proved in \cite[Theorem 5.1]{Ism}.) In turn, Chihara
essentially proved that if $0 < q < a \Le 1$, then
$\{{\sf m}_n[\beta^{(a;q)}]\}_{n=0}^{\infty}$ is an
S-indeterminate Stieltjes moment sequence and
$\beta^{(a;q)}$ is its Friedrichs measure (cf.\
\cite{chi}). This can be also deduced from Lemma
\ref{transH}, \cite[Proposition 4.5.1]{b-v},
\cite[Corollary, p.\ 481]{chi} (see also \cite[Lemma
2.2.5]{j-j-s0}), \cite[p.\ 483, (b)]{chi} and
\eqref{KFm} by considering the measure $\beta^{(a;q)}
\circ \psi_{1,-1}^{-1}$ which coincides with the
measure $\nu^{(a)}_0$ appearing in \cite[Proposition
4.5.1]{b-v}. A similar reasoning shows that if
$1<a<q^{-1}$, then $\{{\sf
m}_n[\beta^{(a;q)}]\}_{n=0}^{\infty}$ is an
S-indeterminate Stieltjes moment sequence (this was
first noticed by Chihara in \cite{chi}) and
$\beta^{(a;q)}$ is its N-extremal measure which is
neither the Krein nor the Friedrichs measure. Consider
now the measure
   \begin{align} \label{gamprezy}
\gamma^{(a;q)} := (q/a;q)_{\infty} \sum_{n=0}^{\infty}
\frac{a^{-n} q^{n^2}}{(q/a;q)_n (q;q)_n}
\delta_{aq^{-n}}, \quad 1<a<q^{-1}.
   \end{align}
Applying Lemmata \ref{hg} and \ref{transH},
\cite[Proposition 4.5.1]{b-v} and Theorem \ref{taso},
the latter to the measure $\gamma^{(a;q)} \circ
\psi_{1,-1}^{-1}$ which coincides with the measure
$\nu^{(a)}_{-1/\xi(a)}$ appearing in \cite[Proposition
4.5.1]{b-v} (consult also Remark \ref{fri}), we deduce
that if $1<a<q^{-1}$, then ${\sf
m}_n[\beta^{(a;q)}]={\sf m}_n[\gamma^{(a;q)}]$ for all
$n\in \zbb_+$, $\gamma^{(a;q)}$ is an orthogonalizing
measure for $\{V^{(a)}_{n}(x;q)\}_{n=0}^{\infty}$ and
$\gamma^{(a;q)}$ is the Friedrichs measure of $\{{\sf
m}_n[\beta^{(a;q)}]\}_{n=0}^{\infty}$. Finally, as
shown in \cite{chi}, if $0 < a \Le q < 1$ or $1 <
q^{-1} \Le a$, then the polynomials
$\{V^{(a)}_{n}(x;q)\}_{n=0}^{\infty}$ have an
H-determinate orthogonalizing measure. Clearly, in the
former case this measure coincides with
$\beta^{(a;q)}$ defined by \eqref{betaq}. To find the
orthogonalizing measure for
$\{V^{(a)}_{n}(x;q)\}_{n=0}^{\infty}$ in the latter
case, we follow an idea due to Ismail, which in fact
can be applied to wider set of parameters (cf.\
\cite[page 592]{Ism}). For this, note that if
$a,\tilde q \in \cbb \setminus \{0\}$, then the
polynomials $\{a^nV^{(1/a)}_{n}(x/a;\tilde
q)\}_{n=0}^{\infty}$ satisfy the same recurrence
relation as the polynomials $\{V^{(a)}_{n}(x;\tilde
q)\}_{n=0}^{\infty}$ (cf.\ \eqref{3termv}), and thus
   \begin{align} \label{1nada}
V^{(a)}_{n}(x;\tilde q) = a^nV^{(1/a)}_{n}(x/a;\tilde
q), \quad n \in \zbb_+, \, a,\tilde q \in \cbb
\setminus \{0\}.
   \end{align}

Now suppose that $a\in (0,\infty)$ and $a^{-1} q < 1$
(recall that $q\in (0,1)$). Then using the measure
transport theorem and the fact that $\beta^{(1/a;q)}$
is the orthogonalizing measure for
$\{V^{(1/a)}_{n}(x;q)\}_{n=0}^{\infty}$, we get
   \begin{align*}
\int_{\rbb} & V^{(a)}_m(x;q) V^{(a)}_n(x;q) \D
\beta^{(1/a;q)}\circ \psi_{a,0}^{-1}(x)
   \\
&\hspace{-1.8ex}\overset{\eqref{1nada}}= \int_{\rbb}
a^mV^{(1/a)}_m(\psi_{a,0}^{-1}(x);q) a^n
V^{(1/a)}_n(\psi_{a,0}^{-1}(x);q) \D
\beta^{(1/a;q)}\circ \psi_{a,0}^{-1}(x)
   \\
&= a^{m+n} \int_{\rbb} V^{(1/a)}_m(x;q)
V^{(1/a)}_n(x;q) \D \beta^{(1/a;q)}(x) = 0, \quad m,n
\in \zbb_+, \, m\neq n.
   \end{align*}
Since $\beta^{(1/a;q)} \circ \psi_{a,0}^{-1}$ is a
probability measure, we deduce that it is an
orthogonalizing measure for
$\{V^{(a)}_{n}(x;q)\}_{n=0}^{\infty}$. If additionally
$aq < 1$, then $\beta^{(a,q)}$ is another
orthogonalizing measure for
$\{V^{(a)}_{n}(x;q)\}_{n=0}^{\infty}$. Now there are
three possibilities. If $a=1$, then the measures
$\beta^{(a;q)}$ and $\beta^{(1/a;q)} \circ
\psi_{a,0}^{-1}$ coincide. If $a < 1$, then $\{{\sf
m}_n[\beta^{(a;q)}]\}_{n=0}^{\infty}$ is an
S-indeterminate Stieltjes moment sequence,
$\beta^{(a;q)}$ is its Friedrichs measure (because $0
< q < a < 1$) and $\beta^{(1/a;q)} \circ
\psi_{a,0}^{-1}$ is its N-extremal measure which is
neither the Krein nor the Friedrichs measure (because
of $1 < a^{-1} < q^{-1}$, \eqref{KFm} and Lemma
\ref{transH}). In turn, if $a > 1$, then $\{{\sf
m}_n[\beta^{(a;q)}]\}_{n=0}^{\infty}$ is an
S-indeterminate Stieltjes moment sequence,
$\beta^{(1/a;q)} \circ \psi_{a,0}^{-1}$ is its
Friedrichs measure (because of $0 < q < a^{-1} < 1$
and Theorem \ref{taso}) which coincides with
$\gamma^{(a;q)}$ defined by \eqref{gamprezy}, and
$\beta^{(a;q)}$ is its N-extremal measure which is
neither the Krein nor the Friedrichs measure (because
$1 < a < q^{-1}$).

Finally, we note that if $a \in (0,\infty)$ and $1 <
q^{-1} \Le a$ (i.e., $aq \Ge 1$), then $a^{-1} q <1$,
and thus, by the above considerations, the measure
$\beta^{(1/a;q)} \circ \psi_{a,0}^{-1}$ is the unique
orthogonalizing measure for
$\{V^{(a)}_{n}(x;q)\}_{n=0}^{\infty}$.

One more observation is at hand. Namely, using the
equality $(aq;q)_{\infty}=(aq;q)_n
(aq^{n+1};q)_{\infty}$, we get
   \begin{align*}
\beta^{(a;q)} = \sum_{n=0}^{\infty} \frac{a^{n}
q^{n^2}(aq^{n+1};q)_{\infty}}{(q;q)_n}
\delta_{q^{-n}}, \quad a>0, \, aq<1.
   \end{align*}
Now it is easily seen that for every $(q,a) \in (0,1)
\times (0,\infty)$ such that $aq\Ge 1$, the right-hand
side of the above equality defines the signed measure
(understood as in \cite{Ash}), call it
$\tilde\beta^{(a;q)}$, which is positive if and only
if $aq^k=1$ for some $k\in \nbb$. Moreover, standard
calculations show that if $aq^{k}=1$ for some $k\in
\nbb$, then $\supp{\tilde\beta^{(a;q)}} =
\{q^{-j}\colon j \Ge k\}$, which together with
\eqref{betaq} implies that
$\tilde\beta^{(a;q)}=\beta^{(1/a;q)} \circ
\psi_{a,0}^{-1}$. In a sense, this means that all
singularities appearing in \eqref{betaq} can be
removed.

The H-determinacy of the Hamburger moment problem
associated with the polynomials
$\{V^{(a)}_{n}(x;q)\}_{n=0}^{\infty}$ for $q\in (-1,0)
\cup (0,1)$ was discussed by Ismail in \cite{Ism}.
   \subsection{Index of H-determinacy}
Following \cite{ber-dur}, we define $\ind z \rho \in
\zbb_+\cup\{\infty\}$, the {\em index of
H-determinacy} of an H-determinate measure $\rho$ at a
point $z\in \cbb$, by
   \begin{align*}
\ind z \rho = \sup\{k\in \zbb_+\colon
|t-z|^{2k}\D\rho(t) \text{ is H-determinate}\}.
   \end{align*}
By the index of H-determinacy of an H-determinate
Stieltjes moment sequence $\gammab =
\{\gamma_n\}_{n=1}^{\infty}$ at a point $z\in \cbb$ we
understand the index of H-determinacy of a unique
H-representing measure of $\gammab$ at the point $z$.
Note that if $\rho$ is an H-determinate measure such
that $\ind {z_0} \rho = \infty$ for some $z_0 \in
\cbb$, then $\ind {z} \rho = \infty$ for all $z \in
\cbb$ (cf.\ \cite[Corollary 3.4(1)]{ber-dur}). If this
is the case, then we say that $\rho$ (or $\gammab$)
has {\em infinite index of H-determinacy}. Note also
that the following holds.
   \begin{align} \label{indinf}
   \begin{minipage}{69ex}
{\em If $\rho \in \msc^+$ is an H-determinate measure
such that $\ind {z} \rho = \infty$ for all $z\in
\cbb$, then the measure $t^k \D \rho(t)$ is
H-determinate for all $k \in \zbb_+$.}
   \end{minipage}
   \end{align}
Indeed, this can be deduced from Proposition
\ref{krytnd} and the inequality
   \begin{align*}
\int_{\sigma} t^k \D \rho(t) \Le \int_{\sigma}
|t-\I|^{2l} \D \rho(t), \quad \sigma \in \borel{\rbb},
\, k,l \in \zbb_+, \, k \Le 2l.
   \end{align*}

The following result can be thought of as a complement
to \cite[Theorem 3.6]{ber-dur}.
   \begin{thm} \label{b-d}
Let $\nu$ be an N-extremal measure and
$\varOmega^\prime$ be an infinite subset of
$\varOmega:=\supp{\nu}$. Set $\rho= \sum_{\lambda\in
\varOmega\setminus \varOmega^\prime} \nu(\{\lambda\})
\delta_\lambda$ $($with $\rho=0$ if
$\varOmega^\prime=\varOmega$$)$. Then $\rho$ is an
H-determinate measure such that $\ind z {\rho} =
\infty$ for all $z\in \cbb$.
   \end{thm}
   \begin{proof}
Fix $n\in \nbb$. Take any subset $\varOmega_n$ of
$\varOmega^\prime$ such that $\card{\varOmega_n}=n$.
Set $\rho_n = \sum_{\lambda\in \varOmega\setminus
\varOmega_n} \nu(\{\lambda\}) \delta_\lambda$. Note
that for every Borel function $f\colon \rbb \to
\rbop$,
   \begin{align} \label{domin1}
\int_{\rbb} f \D \rho = \int_{\varOmega \setminus
\varOmega^\prime} f \D \nu \Le \int_{\varOmega
\setminus \varOmega_n} f \D \nu = \int_{\rbb} f \D
\rho_n.
   \end{align}
By \cite[Theorem 3.6]{ber-dur}, $\rho_n$ is
H-determinate and $n \Ge \ind z {\rho_n} \Ge n-1$ for
all $z \in \cbb$. Using Proposition \ref{krytnd} and
applying the inequality \eqref{domin1} first to $f=
\chi_{\sigma}$ and then to
$f(t)=\chi_{\sigma}(t)|t-z|^{2k}$, we deduce that
$\rho$ is H-determinate and $\ind z {\rho} \Ge \ind z
{\rho_n}$ for all $z \in \cbb$. This completes the
proof.
   \end{proof}
   \subsection{The Carleman condition}  \label{Sec2.5}
Suppose $\{\gamma_n\}_{n=0}^\infty$ is a Stieltjes
moment sequence. Clearly, the shifted sequence
$\{\gamma_{n+1}\}_{n=0}^\infty$ is a Stieltjes moment
sequence. It turns out that if
$\{\gamma_{n+1}\}_{n=0}^\infty$ is S-determinate, then
so is $\{\gamma_n\}_{n=0}^\infty$ (see
\cite[Proposition 5.12]{sim}; see also \cite[Lemma
2.4.1]{b-j-j-sA}). The reverse implication is not true
in general (cf.\ \cite{j-j-s0}). However, it is true
if $\{\gamma_n\}_{n=0}^\infty$ satisfies the Carleman
condition (see Proposition \ref{wouk} below). Recall
that a sequence $\{\gamma_n\}_{n=0}^\infty \subseteq
\rbb_+$ satisfies the {\em Carleman condition} if
$\sum_{n=1}^\infty \frac{1}{\gamma_n^{1/2n}} = \infty$
with the convention that $\frac{1}{0}=\infty$.

Below we collect some properties of Stieltjes moment
sequences that satisfy the Carleman condition.
   \begin{pro}\label{wouk}
Let $\gammab=\{\gamma_n\}_{n=0}^\infty$ be a Stieltjes
moment sequence. Then
   \begin{enumerate}
   \item[(i)] if
$\gammab$ satisfies the Carleman condition, then
$\gammab$ is H-determinate,
   \item[(ii)] $\gammab$ satisfies
the Carleman condition if and only if
$\{\gamma_{n+1}\}_{n=0}^\infty$ satisfies the Carleman
condition,
   \item[(iii)] if $\gammab$
satisfies the Carleman condition, then so does
\mbox{$\{\gamma_n + c\}_{n=0}^\infty$} for every $c
\in (0,\infty)$,
   \item[(iv)] $\gammab$
satisfies the Carleman condition if and only if
$\{\gamma_{jp}\}_{j=0}^\infty$ satisfies the Carleman
condition for every $p\in \nbb$ $($equivalently{\em :}
for some $p\in \nbb$$)$.
   \end{enumerate}
   \end{pro}
   \begin{proof}
(i) See \cite[Corollary 4.5]{sim}.

(ii) This can be deduced from the equivalence
(d)$\Leftrightarrow$(e) in \cite[Theorem 19.11]{Rud}
(the equivalence follows from the Carleman inequality,
cf.\ \cite[p.\ 105]{Car}).

(iii) Let $\nu$ be an S-representing measure of
$\{\gamma_n\}_{n=0}^\infty$. Set $a=\nu((1,\infty))$.
Since the case of $a = 0$ is obvious, we can assume
that $a> 0$. Then
   \begin{align*}
\gamma_n + c \Le \gamma_n + \frac{c}{a}
\int_{(1,\infty)} t^n \D \nu(t) \Le \Big(1 +
\frac{c}{a}\Big) \gamma_n, \quad n \in \zbb_+,
   \end{align*}
which implies that $\{\gamma_n + c\}_{n=0}^\infty$
satisfies the Carleman condition.

(iv) This follows from \cite[Section 1]{StSz1}
because, by the Cauchy-Schwarz inequality, $\gamma_n^2
\Le \gamma_k \gamma_l$ for all nonnegative integers
$k,l$ such that $k+l=2n$.
   \end{proof}
Note that in view of Proposition \ref{wouk}, a
Stieltjes moment sequence which satisfies the Carleman
condition has infinite index of H-determinacy because
its index of H-determinacy at $0$ is infinite.
   \section{{\bf Composition operators over one-circuit
directed graphs}}
   \subsection{Criteria for hyponormality and subnormality}
\label{sec3.1} In this paper we study composition
operators in $L^2$-spaces over discrete measure
spaces. By a {\em discrete measure} on a (nonempty)
set $X$ we understand a $\sigma$-finite measure $\mu$
on the $\sigma$-algebra $2^X$ such that $\mu(x) > 0$
for every $x \in X$, with the convention
   \begin{align} \label{dudu}
\mu(x):=\mu(\{x\}), \quad x \in X.
   \end{align}
Note that if $\mu$ is a discrete measure on $X$, then
$X$ is at most countable and $\mu(x) < \infty$ for
every $x\in X$. Moreover, any discrete measure $\mu$
on $X$ is determined by a function $\mu\colon X \to
(0,\infty)$ via $\mu(\{x\})=\mu(x)$. This one-to-one
correspondence between discrete measures and positive
functions is used frequently in the present paper.

Let $\mu$ be a discrete measure on a set $X$ and let
$\phi$ be a self-map of $X$. Then the operator
$C_{\phi}$ in $L^2(\mu)$ given by
   \begin{align*}
\dz{C_{\phi}} = \{f \in L^2(\mu) \colon f \circ \phi
\in L^2(\mu)\} \text{ and } C_{\phi} f = f \circ \phi
\text{ for } f \in \dz{C_{\phi}},
   \end{align*}
is called a {\em composition operator} in $L^2(\mu)$
with a {\em symbol} $\phi$. Since the measure
$\mu\circ \phi^{-1}$ given by $(\mu\circ
\phi^{-1})(\varDelta) = \mu(\phi^{-1}(\varDelta))$ for
$\varDelta \in 2^X$ is absolutely continuous with
respect to $\mu$, we can consider the Radon-Nikodym
derivative $\hfi=\D \mu\circ \phi^{-1}/\D\mu$. Clearly
   \begin{align} \label{hfi}
\hfi(x) = \frac{\mu(\phi^{-1}(\{x\}))}{\mu(x)}, \quad
x \in X.
   \end{align}
Hence, $\hfi(x) > 0$ for every $x \in X$ if and only
if $\phi(X)=X$. It is easily seen that (cf.\
\cite{nor})
   \begin{align} \label{ogrn}
   \begin{minipage}{65ex}
{\em $C_{\phi} \in \ogr{L^2(\mu)}$ if and only if
$\sup_{x\in X} \hfi(x) < \infty$, and if this is the
case, then $\|C_{\phi}^n\|^2 = \sup_{x\in X}
\hfin{n}(x)$ for every $n\in \zbb_+$.}
   \end{minipage}
   \end{align}
Note also that
   \begin{align} \label{chiphi}
\text{$\|C_{\phi}^{n} \chi_{\{u\}}\|^2 = \mu(u)
\hfin{n}(u)$ whenever $u \in X$, $\chi_{\{u\}} \in
\dz{C_{\phi}^n}$ and $n\in \zbb_+$.}
   \end{align}
Applying \cite[Proposition 3.2]{b-j-j-sC} and the
assertions (ii) and (iv) of \cite[Proposition
4.1]{b-j-j-sC}, we get a new criterion for the $n$th
power of $C_{\phi}$ to be densely defined.
   \begin{align} \label{dzn}
   \begin{minipage}{65ex}
{\em If $n\in \nbb$, then $C_{\phi}^n$ is densely
defined if and only if ${\mathsf h}_{\phi^n}(x) <
\infty$ for every $x\in X$.}
   \end{minipage}
   \end{align}
Thus, if $n\in \nbb$ and ${\mathsf h}_{\phi^n}(x) <
\infty$ for all $x\in X$, then ${\mathsf
h}_{\phi^j}(x) < \infty$ for all $j\in \{1, \ldots,
n\}$ and $x\in X$. This hereditary property is no
longer true for a single point $x\in X$ (see
\cite[Example 4.2]{jab}). In fact, given any nonempty
subset $\varXi$ of $\nbb$, we can construct a discrete
measure $\mu$ on a set $X$ and a self-map $\phi$ of
$X$ such that $\varXi=\{n \in \nbb\colon \hfin{n}(x_0)
= \infty\}$ for some~ $x_0 \in X$. This is shown
below.
   \begin{exa}
Fix $k \in \nbb \cup \{\infty\}$. Let $X$,
$\{x_i\}_{i=0}^\infty$,
$\{x_{i,j}\}_{i=1}^{\eta+1}{_{j=1}^{\,l_i}}$ and
$\phi$ be as in Theorem
\ref{golfclub2}\mbox{(ii-b$^*$)} with $\eta=\infty$
and $l_i=k$ for all $i\in \nbb$. Suppose that $\varXi$
is a nonempty subset of $J_{k}$ (see Section \ref{n&t}
for the definition of $J_{k}$). Define the discrete
measure $\mu$ on $X$ by
   \begin{align*}
\mu(\{x\}) =
   \begin{cases}
2^{-i} & \text{if $x \in \{x_{i,j}\colon i\in \nbb, \,
j\in J_{k} \setminus \varXi\}$,}
   \\
1 & \text{otherwise,}
   \end{cases}
\quad x \in X.
   \end{align*}
It is a matter of routine to show that
   \begin{align*}
\hfin{n}(x_0) =
   \begin{cases}
   \infty & \text{if } n \in \varXi,
   \\
   1 & \text{if $n \in J_{k} \setminus \varXi$},
   \\
   0 & \text{if $k < \infty$ and $n > k$.}
   \end{cases}
\quad n \in \nbb.
   \end{align*}
   \end{exa}
Assume now that $C_{\phi}$ is densely defined (as
before, $\mu$ is a discrete measure on a set $X$ and
$\phi$ is a self-map of $X$), or equivalently, by
\eqref{hfi} and \eqref{dzn}, that
$\mu(\phi^{-1}(\{x\})) < \infty$ for all $x\in
\phi(X)$. Then for every function $f\colon X \to
\rbop$, there exists a unique
$\phi^{-1}(2^X)$-measurable function $\efi(f)\colon X
\to \rbop$ such that
   \begin{align*}
\int_{\phi^{-1}(\varDelta)} f \D \mu =
\int_{\phi^{-1}(\varDelta)} \efi(f) \D \mu, \quad
\varDelta \subseteq X,
   \end{align*}
or equivalently, such that for every $x\in \phi(X)$,
   \begin{align} \label{cexp}
\efi(f)(z) = \frac{\sum_{y \in \phi^{-1}(\{x\})}
\mu(y) f(y)}{\mu(\phi^{-1}(\{x\}))}, \quad z \in
\phi^{-1}(\{x\}).
   \end{align}
The above definition is correct because
$X=\bigsqcup_{x \in \phi(X)}\phi^{-1}(\{x\})$. The
function $\efi(f)$ is called the {\em conditional
expectation} of a function $f\colon X \to \rbop$ with
respect to the $\sigma$-algebra $\phi^{-1}(2^X)$ (cf.\
\cite{b-j-j-sC}). Following \cite{b-j-j-sS}, we say
that a family $\{P(x,\cdot)\}_{x\in X}$ of Borel
probability measures on $\rbb_+$ satisfies the {\em
consistency condition} if
   \begin{align*}
\hfi(\phi(z)) \cdot \efi(P(\cdot,\sigma))(z) =
\int_{\sigma} t P(\phi(z), \D t), \quad z \in X, \,
\sigma \in \borel{\rbb_+}.
   \end{align*}
In view of \eqref{hfi} and \eqref{cexp}, we see that
$\{P(x,\cdot)\}_{x\in X}$ satisfies the consistency
condition if and only if
   \begin{align} \label{cc}   \tag{CC}
\frac{1}{\mu(x)}\sum_{y \in \phi^{-1}(\{x\})} \mu(y)
P(y,\sigma) = \int_{\sigma} t P(x, \D t), \quad \sigma
\in \borel{\rbb_+}, \, x \in \phi(X).
   \end{align}

The following characterization of hyponormality of
$C_\phi$ can be deduced from \eqref{hfi}, \eqref{cexp}
and \cite[Corollary 6.7]{ca-hor} (see also \cite[Lemma
2.1]{Bu}).
   \begin{pro}\label{buda}
Let $\mu$ be a discrete measure on a set $X$ and
$\phi$ be a self-map of $X$. Assume that $C_{\phi}$ is
densely defined. Then $C_\phi$ is hyponormal if and
only if for every $x\in X$, $\hfi(x) > 0$ and
   \begin{align}   \label{Bud-ski}
\frac{1}{\mu(x)}\sum_{y \in \phi^{-1}(\{x\})}
\frac{\mu(y)^2}{\mu(\phi^{-1}(\{y\}))} \Le 1.
   \end{align}
   \end{pro}
A criterion for subnormality of unbounded composition
operators given in \cite[Theorems 9 and 17]{b-j-j-sS}
takes in the present situation the following form
(recall that if $C_{\phi}$ is subnormal, then $\hfi(x)
> 0$ for every $x\in X$, or equivalently $C_{\phi}$ is
injective, cf.\ \cite[Section 6]{b-j-j-sC}).
   \begin{thm} \label{glowne}
Let $\mu$ be a discrete measure on a set $X$ and
$\phi$ be a self-map of $X$. Assume that $C_{\phi}$ is
densely defined and $\hfi(x) > 0$ for every $x\in X$.
Suppose that there exists a family
$\{P(x,\cdot)\}_{x\in X}$ of Borel probability
measures on $\rbb_+$ that satisfies \eqref{cc}. Then
$C_{\phi}$ is subnormal and
   \begin{align} \label{twomom}
\hfin{n}(x) = \int_0^\infty t^n P(x,\D t), \quad n \in
\zbb_+, \, x \in X.
   \end{align}
   \end{thm}

Let $\gcal=(V,E)$ be a directed graph (i.e., $V$ is a
nonempty set called the {\em vertex set} of $\gcal$
and $E$ is a subset of $V\times V$ called {\em the
edge set} of $\gcal$). We say that a vertex $u$ is a
{\em parent} of a vertex $v$, and write $\pa{v}=u$, if
$(u,v) \in E$ and $u=w$ whenever $(w,v) \in E$. The
directed graph $\gcal$ is said to be {\em connected}
if for every pair $(u,v)$ of distinct vertices there
exists an {\em undirected path} joining $u$ and $v$,
i.e., a finite sequence $\{u_i\}_{i=1}^k$ of vertices
with $k\ge 2$ such that $u_1 = u$, $u_k = v$ and for
every $i \in J_{k-1}$, either $(u_i,u_{i+1}) \in E$ or
$(u_{i+1},u_i) \in E$. We say that a finite sequence
$\{u_j\}_{j=1}^n$ of distinct vertices is a {\em
circuit} if $n\ge 2$, $(u_j,u_{j+1}) \in E$ for all $j
\in J_{n-1}$ and $(u_n,u_1) \in E$. By a {\em rootless
directed tree} we mean a directed graph $\tcal =
(V,E)$ which is connected, has no circuits, each
vertex of $\tcal$ has a parent and $(u,u) \notin E$
for every $u \in V$. In this case, obviously, the
partial function $\paa$ is a self-map of $V$.

Given a rootless directed tree $\tcal=(V,E)$ and a
family $\lambdab = \{\lambda_v\}_{v\in V}$ of complex
numbers, we define a {\em weighted shift} $\slam$ on
$\tcal$ with weights $\lambdab$ via
   \begin{align*}
\dz{\slam} = \{f \in \ell^2(V)\colon \lambdab \cdot f
\circ \paa \in \ell^2(V)\} \text{ and } \slam f =
\lambdab \cdot f \circ \paa \text{ for } f \in
\dz{\slam},
   \end{align*}
where $(\lambdab \cdot f \circ \paa)(v) = \lambda_v
f(\pa{v})$ for $v\in V$. We refer the reader to
\cite{j-j-s} for more information on directed trees
and weighted shifts on them.

It follows from \cite[Lemma 4.3.1]{j-j-s0} and
\cite[Theorem 3.2.1]{j-j-s} that each weighted shift
$\slam$ on a countable rootless directed tree
$\tcal=(V,E)$ with nonzero weights is unitarily
equivalent to a composition operator $C_{\paa}$ in
$L^2(V,2^V,\mu)$ for some discrete measure $\mu$ on
$V$ (the assumption that $V$ is infinite made in
\cite[Lemma 4.3.1]{j-j-s0} is redundant). As
explicated below, the proof of \cite[Lemma
4.3.1]{j-j-s0} contains more information.
   \begin{lem} \label{jfa-ex}
Let $\tcal=(V, E)$ be a rootless directed tree.
Suppose that $\mu$ is a discrete measure on $V$. Then
the composition operator $C_{\paa}$ in $L^2(\mu)$ is
unitarily equivalent to a weighted shift on $\tcal$
with positive weights.
   \end{lem}
   \begin{proof}
Since $\mu$ is a discrete measure on $V$, we see that
$\card{V} \Le \aleph_0$. Consider the weighted shift
$\slam$ on the directed tree $\tcal$ with weights
$\Big\{\sqrt{\frac{\mu(v)} {\mu(\pa{v})}}\,\,
\Big\}_{v\in V}$. A careful inspection of the proof of
\cite[Lemma 4.3.1]{j-j-s0} reveals that the
composition operator $C_{\paa}$ is unitarily
equivalent to $\slam$ via the unitary isomorphism
$U\colon \ell^2(V)\to L^2(\mu)$ defined in
\cite[(4.3.4)]{j-j-s0}.
   \end{proof}
   \subsection{A class of directed graphs with one circuit}
   \label{sect-gr} In this section we classify
connected directed graphs induced by self-maps whose
vertices, all but one, have valency one and the
valency of the remaining vertex is nonzero (see
Theorem \ref{golfclub} below; see also Figures $2$ and
$3$ which illustrate this theorem).

Let $X$ be a nonempty set and $\phi$ be a self-map of
$X$. Set
   \begin{align} \label{ephidef}
E^{\phi} = \{(x,y) \in X\times X \colon x = \phi(y)\}.
   \end{align}
Then $(X,E^{\phi})$ is a directed graph which we call
the {\em directed graph induced by} $\phi$. Note that
the valency of a vertex $x\in X$ is equal to
$\card{\phi^{-1}(\{x\})}$ and that $\phi(x)$ is the
parent of $x$. We will write $\phi^{-n}(A) =
(\phi^n)^{-1}(A)$ whenever $n \in \zbb_+$ and $A
\subseteq X$.
   \begin{tikzpicture}[scale = .6,
transform shape] \tikzstyle{every node} =
[circle,fill=gray!15] \node (e0)[font=\huge] at (2,-1)
{\phantom{1}$x_0$\phantom{1}}; \node (e-1)[font=\huge]
at (2,3) {$x_{\kappa-1}$}; \node (e-2)[font=\huge] at
(-1,3){$x_{\kappa-2}$}; \node (e1)[font=\huge] at
(-1,-1) {\phantom{1}$x_1$\phantom{1}}; \node
(ek)[font=\huge] at (3.5,1)
{\phantom{L}$x_\kappa$\phantom{l}}; \node
(e11)[font=\huge] at (7,3) {$x_{1,1}$}; \node
(e12)[font=\huge] at (10.5,3) {$x_{1,2}$}; \node
(e13)[font=\huge] at (14,3) {$x_{1,3}$}; \node[fill =
none] (e1n) at (16,3) {}; \node (e21)[font=\huge] at
(7,1) {$x_{2,1}$}; \node (e22)[font=\huge] at (10.5,1)
{$x_{2,2}$}; \node (e23)[font=\huge] at (14,1)
{$x_{2,3}$}; \node[fill = none] (e2n)[font=\huge] at
(16,1) {}; \node[fill = none] (e 5 1)[font=\huge] at
(7,-3) {}; \node[fill = none] (e 5 n)[font=\huge] at
(16,-3) {}; \node (e31)[font=\huge] at (7,-1)
{$x_{3,1}$}; \node (e32)[font=\huge] at (10.5,-1)
{$x_{3,2}$}; \node (e33)[font=\huge] at (14,-1)
{$x_{3,3}$}; \node (e3n)[fill = none] at (16,-1) {};
\draw[<-] (e11)--(ek) node[pos=0.5,above =
0pt,fill=none] {}; \draw[<-] (e21)--(ek) node[pos=0.5,
above = 0pt,sloped, fill=none]{}; \draw[dotted] (e 5
1)--(ek) node [pos=0.5,below=-10pt,sloped,fill=none]
{}; \draw[<-] (e31)--(ek) node[pos=0.5,
below=0pt,fill=none] {}; \draw[<-] (e12)--(e11)
node[pos=0.5,above = 0pt,fill=none] {}; \draw[<-]
(e22)--(e21) node[pos=0.5,above = 0pt,fill=none] {};
\draw[<-] (e32)--(e31)
node[pos=0.5,below=0pt,fill=none] {}; \draw[<-]
(e13)--(e12) node[pos=0.5,above = 0pt,fill=none] {};
\draw[<-] (e23)--(e22)
node[pos=0.5,above=0pt,fill=none] {}; \draw[<-]
(e33)--(e32) node[pos=0.5,below = 0pt,fill=none] {};
\draw[dotted] (e 5 n)--(e 5 1) node[pos=0.5,above =
-10pt,fill=none] {}; \draw[ <-] (ek) to [bend right]
(e-1) node[pos=0.5,above = 0pt,fill=none] {};
\draw[dashed, <-] (e1n)--(e13); \draw[dashed, <-]
(e2n)--(e23); \draw[dashed, <-] (e3n)--(e33);
\draw[dashed, <-] (e-2) to [bend right] (e1); \draw[
<-] (e0) to [bend right] (ek); \draw[ <-] (e-1)--(e-2)
node[pos=0.5,above = 0pt,fill=none] {}; \draw[ <-]
(e1)--(e0) node[pos=0.5,above = 0pt,fill=none] {};
   \end{tikzpicture}
   \begin{center}
   \begin{minipage}{72ex}
{\small {\sf Figure $2$}}. The directed graph
$(X,E^{\phi})$ in the case of \mbox{(ii-a)} (the
self-map \\ \phantom{xxxxxxxi}$\phi$ acts in
accordance with the reverted arrows).
   \end{minipage}
   \end{center}
   \vspace{1ex}
   \begin{tikzpicture}[scale = .6,
transform shape] \tikzstyle{every node} =
[circle,fill=gray!15] \node (e-1)[font=\huge] at (1,1)
{\phantom{l}$x_{1}$\phantom{l}}; \node
(e-2)[font=\huge] at
(-1.5,1){\phantom{l}$x_{2}$\phantom{l}}; \node[fill =
none] (e1)[font=\huge] at (-3.5,1) {}; \node
(ek)[font=\huge] at (3.5,1)
{\phantom{l}$x_0$\phantom{l}}; \node (e11)[font=\huge]
at (7,3) {$x_{1,1}$}; \node (e12)[font=\huge] at
(10.5,3) {$x_{1,2}$}; \node (e13)[font=\huge] at
(14,3) {$x_{1,3}$}; \node[fill = none] (e1n) at (16,3)
{}; \node (e21)[font=\huge] at (7,1) {$x_{2,1}$};
\node (e22)[font=\huge] at (10.5,1) {$x_{2,2}$}; \node
(e23)[font=\huge] at (14,1) {$x_{2,3}$}; \node[fill =
none] (e2n)[font=\huge] at (16,1) {}; \node[fill =
none] (e 5 1)[font=\huge] at (7,-3) {}; \node[fill =
none] (e 5 n)[font=\huge] at (16,-3) {}; \node
(e31)[font=\huge] at (7,-1) {$x_{3,1}$}; \node
(e32)[font=\huge] at (10.5,-1) {$x_{3,2}$}; \node
(e33)[font=\huge] at (14,-1) {$x_{3,3}$}; \node
(e3n)[fill = none] at (16,-1) {}; \draw[<-]
(e11)--(ek) node[pos=0.5,above = 0pt,fill=none] {};
\draw[<-] (e21)--(ek) node[pos=0.5, above =
0pt,sloped, fill=none]{}; \draw[dotted] (e 5 1)--(ek)
node [pos=0.5,below=-10pt,sloped,fill=none] {};
\draw[<-] (e31)--(ek) node[pos=0.5,
below=0pt,fill=none] {}; \draw[<-] (e12)--(e11)
node[pos=0.5,above = 0pt,fill=none] {}; \draw[<-]
(e22)--(e21) node[pos=0.5,above = 0pt,fill=none] {};
\draw[<-] (e32)--(e31)
node[pos=0.5,below=0pt,fill=none] {}; \draw[<-]
(e13)--(e12) node[pos=0.5,above = 0pt,fill=none] {};
\draw[<-] (e23)--(e22)
node[pos=0.5,above=0pt,fill=none] {}; \draw[<-]
(e33)--(e32) node[pos=0.5,below = 0pt,fill=none] {};
\draw[dotted] (e 5 n)--(e 5 1) node[pos=0.5,above =
-10pt,fill=none] {}; \draw[ <-] (ek) to [] (e-1);
\draw[dashed, <-] (e1n)--(e13); \draw[dashed, <-]
(e2n)--(e23); \draw[dashed, <-] (e3n)--(e33);
\draw[dashed, <-] (e-2) to [] (e1); \draw[ <-]
(e-1)--(e-2) node[pos=0.5,above = 0pt,fill=none] {};
   \end{tikzpicture}
   \begin{center}
   \begin{minipage}{72ex}
{\small {\sf Figure $3$}}. The directed graph
$(X,E^{\phi})$ in the case of \mbox{(ii-b)}.
   \end{minipage}
   \end{center}
   \vspace{1ex}
   \begin{thm} \label{golfclub}
Let $X$ and $\phi$ be as above and let $\eta \in \nbb
\cup \{\infty\}$. Then the following two conditions
are equivalent{\em :}
   \begin{enumerate}
   \item[(i)] the directed graph $(X,E^{\phi})$ is
connected and there exists $\omega \in X$ such that
$\card{\phi^{-1}(\{\omega\})} = \eta + 1$ and
$\card{\phi^{-1}(\{x\})} = 1$ for every $x \in X
\setminus \{\omega\}$,
   \item[(ii)] one of the following two conditions
is satisfied{\em :}
   \begin{enumerate}
   \item[(ii-a)] there exist $\kappa \in \zbb_+$ and
two disjoint systems $\{x_i\}_{i=0}^{\kappa}$ and
$\{x_{i,j}\}_{i=1}^{\eta}{_{j=1}^\infty}$ of distinct
points of $X$ such that
   \allowdisplaybreaks
   \begin{align} \label{Xr1}
X & = \{x_0, \ldots, x_{\kappa}\} \cup \{x_{i,j}\colon
i \in J_{\eta}, j \in \nbb\},
   \\  \label{Xr1.5}
\phi(x) & =
   \begin{cases}
x_{i,j-1} & \text{ if } x = x_{i,j} \text{ with } i
\in J_{\eta} \text{ and } \, j \Ge 2,
   \\
x_{\kappa} & \text{ if } x = x_{i,1} \text{ with } i
\in J_{\eta} \text{ or } x = x_0,
   \\
x_{i-1} & \text{ if } x = x_i \text{ with } i \in
J_{\kappa},
   \end{cases}
   \end{align}
   \item[(ii-b)] there exist two disjoint systems
$\{x_i\}_{i=0}^\infty$ and
$\{x_{i,j}\}_{i=1}^{\eta+1}{_{j=1}^\infty}$ of
distinct points of $X$ such that
   \allowdisplaybreaks
   \begin{align} \notag
X & = \{x_i\colon i \in \zbb_+\} \cup \{x_{i,j}\colon
i \in J_{\eta+1}, j \in \nbb\},
   \\ \label{Xr3}
\phi(x) & =
   \begin{cases}
x_{i,j-1} & \text{ if } x = x_{i,j} \text{ with } i
\in J_{\eta+1} \text{ and } \, j \Ge 2,
   \\
x_0 & \text{ if } x = x_{i,1} \text{ with } i \in
J_{\eta+1},
   \\
x_{i+1} & \text{ if } x = x_i \text{ with } i \in
\zbb_+.
   \end{cases}
   \end{align}
   \end{enumerate}
  \end{enumerate}
   \end{thm}
   \begin{proof}
Clearly, we need only to prove the implication
(i)$\Rightarrow$(ii). We do it in five steps. The
proof of Step 1 being simple is omitted.

{\sc Step 1.} If $\omega \in X$ is such that
$\card{\phi^{-1}(x)}\Le 1$ for all $x \in X \setminus
\{\omega\}$, then the restriction of $\phi$ to
$\phi^{-1}(X \setminus \{\omega\})$ is injective.

{\sc Step 2.} If $\omega \in X$ and $\varOmega
\subseteq X$ are such that $\omega \in \varOmega$,
$\phi(\varOmega) \subseteq \varOmega$,
$\card{\phi^{-1}(\{\omega\})} \Ge 2$,
$\card{\phi^{-1}(\{x\})} = 1$ for every $x \in X
\setminus \{\omega\}$ and $\varDelta :=
\phi^{-1}(\{\omega\}) \setminus \varOmega \neq
\emptyset$, then
   \begin{enumerate}
   \item[(s1)] $\card{\phi^{-n}(\{x\})} = 1$ provided
$n \in \zbb_+$ and $x \in X \setminus \varOmega$,
   \item[(s2)] $\phi^{-n}(x) \in X \setminus \varOmega$
provided $n \in \zbb_+$ and $x \in X \setminus
\varOmega$, where $\phi^{-n}(x)$ is a unique element
of $X$ such that $\{\phi^{-n}(x)\} = \phi^{-n}(\{x\})$
(see (s1)),
   \item[(s3)] $x = \phi^n(\phi^{-n}(x))$ and
$\phi^{-(m+n)}(x) = \phi^{-m}(\phi^{-n}(x))$ provided
$m,n \in \zbb_+$ and $x\in X \setminus \varOmega$,
   \item[(s4)]
$\{\phi^{-n}(x)\colon n \in \zbb_+\}$ are distinct
points of $X$ provided $x \in \varDelta$,
   \item[(s5)]
$\{\phi^{-m}(x)\colon m\in \zbb_+\} \cap
\{\phi^{-n}(y)\colon n \in \zbb_+\} = \emptyset$
provided $x,y \in \varDelta$ and~ $x\neq y$.
   \end{enumerate}

For this, take $x\in X \setminus \varOmega$. By
assumption, $\card{\phi^{-1}(\{x\})} = 1$ and thus
there exists a unique $\phi^{-1}(x) \in X$ such that
$\{\phi^{-1}(x)\} = \phi^{-1}(\{x\})$. Since
$\phi(\varOmega) \subseteq \varOmega$ and $x =
\phi(\phi^{-1}(x))$, we deduce that $\phi^{-1}(x) \in
X \setminus \varOmega$. An induction argument yields
(s1), (s2) and (s3). To prove (s4) and (s5), suppose
that $x,y \in \varDelta$ and $\phi^{-i}(x) =
\phi^{-j}(y)$ for some integers $0 \Le i \Le j$. If
$i<j$, then by (s2) and (s3), we have
   \begin{align*}
x = \phi^i(\phi^{-i}(x)) = \phi^i(\phi^{-j}(y)) =
\phi^i(\phi^{-i}(\phi^{-(j-i)}(y))) = \phi^{-(j-i)}(y)
   \end{align*}
and so
   \begin{align*}
\omega = \phi(x) = \phi(\phi^{-(j-i)}(y)) =
\phi(\phi^{-1}(\phi^{-(j-i-1)}(y))) =
\phi^{-(j-i-1)}(y),
   \end{align*}
which yields $y = \phi^{j-i-1}(\omega) \in \varOmega$,
a contradiction. Finally, if $i=j$, then
   \begin{align*}
x= \phi^i(\phi^{-i}(x)) = \phi^j(\phi^{-j}(y)) =y,
   \end{align*}
which completes the proof of (s4) and (s5).

{\sc Step 3.} If $(X,E^{\phi})$ is connected and $Y$
is a nonempty subset of $X$ such that $\phi(Y)
\subseteq Y$ and $\phi(X\setminus Y)\subseteq
X\setminus Y$, then $X=Y$.

Indeed, otherwise there exists $x \in X \setminus Y$.
Take $y \in Y$. Since the graph $(X,E^{\phi})$ is
connected, there exists a finite sequence
$\{u_i\}_{i=1}^k$ of elements of $X$ with $k\ge 2$
such that $u_1 = x$, $u_k = y$ and for every $i \in
J_{k-1}$, either $(u_i,u_{i+1}) \in E^{\phi}$ or
$(u_{i+1},u_i) \in E^{\phi}$. Then there exists $j\in
J_{k-1}$ such that $u_j \in X\setminus Y$ and $u_{j+1}
\in Y$. Thus either $u_j = \phi(u_{j+1}) \in Y$ or
$u_{j+1} = \phi(u_j) \in X\setminus Y$, a
contradiction.

{\sc Step 4.} If $(X,E^{\phi})$ is connected and
$\omega \in X$ is such that
$\card{\phi^{-1}(\{\omega\})} = \eta + 1$,
$\card{\phi^{-1}(\{x\})} = 1$ for every $x \in X
\setminus \{\omega\}$ and $\omega \in
\phi^{-n}(\{\omega\})$ for some $n\in \nbb$, then
\mbox{(ii-a)} holds.

For this, we set
   \begin{align}\label{dzien1}
\kappa &= \min\{n \in \nbb\colon \omega \in
\phi^{-n}(\{\omega\})\}-1.
   \end{align}
First we show that $\{\phi^i(\omega)\}_{i=0}^{\kappa}$
is a sequence of distinct points of $X$. Indeed, if
$\phi^i(\omega) = \phi^j(\omega)$ for some integers $0
\Le i < j \Le \kappa$, then $1 \Le \kappa+1 - (j - i)
< \kappa + 1$ and
   \begin{align*}
\omega \overset{\eqref{dzien1}}=
\phi^{\kappa+1}(\omega) =
\phi^{\kappa+1-j}(\phi^j(\omega)) =
\phi^{\kappa+1-j}(\phi^i(\omega)) = \phi^{\kappa+1 -
(j - i)}(\omega),
   \end{align*}
which contradicts \eqref{dzien1}. Set $\varOmega =
\big\{x_i \colon i \in \{0, \ldots, \kappa\}\big\}$
with $x_i = \phi^{\kappa-i}(\omega)$ for $i \in \{0,
\ldots, \kappa\}$, and $\varDelta =
\phi^{-1}(\{\omega\}) \setminus \varOmega$. Then
clearly $\phi(x_0) = x_{\kappa}=\omega$ and $\phi(x_i)
= x_{i-1}$ for $i \in J_{\kappa}$, and
$\phi(\varOmega) \subset \varOmega$. Since
$\phi|_{\varOmega}$ is injective, we see that
$\phi^{-1}(\{\omega\}) \cap \varOmega = \{x_0\}$. This
and $\card{\phi^{-1}(\{\omega\})} = \eta + 1$ imply
that $\card{\varDelta} = \eta$. Thus, by Step 2, the
conditions (s1)-(s5) hold. Let
$\{x_{i,1}\}_{i=1}^{\eta}$ be a sequence of distinct
points of $\varDelta$. Setting $x_{i,j} =
\phi^{-(j-1)}(x_{i,1})$ for $i\in J_{\eta}$ and $j \in
\nbb$, we verify that $\{x_i\}_{i=0}^{\kappa}$ and
$\{x_{i,j}\}_{i=1}^{\eta}{_{j=1}^\infty}$ are disjoint
systems of distinct points of $X$ which satisfy
\eqref{Xr1.5}. Hence, $\phi(Y) = Y$ and
$\phi^{-1}(\{x\}) \cap Y \neq \emptyset$ for every
$x\in Y$, where
   \begin{align} \label{defx0}
Y = \{x_0, \ldots, x_{\kappa}\} \cup \{x_{i,j}\colon i
\in J_{\eta}, j \in \nbb\}.
   \end{align}
 Since $\card{\phi^{-1}(\{x\})} = 1$ for every $x \in
X \setminus \{\omega\}$ and
$\card{\phi^{-1}(\{\omega\})} = \eta + 1$, we deduce
that $\phi^{-1}(Y) \subseteq Y$, or equivalently that
$\phi(X\setminus Y) \subseteq X\setminus Y$. This
together with Step 3 completes the proof of Step 4.

{\sc Step 5.} If $(X,E^{\phi})$ is connected and
$\omega \in X$ is such that
$\card{\phi^{-1}(\{\omega\})} = \eta + 1$,
$\card{\phi^{-1}(\{x\})} = 1$ for every $x \in X
\setminus \{\omega\}$ and $\omega \notin
\phi^{-n}(\{\omega\})$ for every $n\in \nbb$, then
(ii-b) holds.

   We begin by proving that
   \begin{align} \label{clem}
\text{$\{\phi^i(\omega)\}_{i=0}^\infty$ is a sequence
of distinct points of $X$.}
   \end{align}
For this, suppose that $\phi^i(\omega) =
\phi^j(\omega)$ for some integers $0\Le i< j$. Since,
by assumption $\phi^k(\omega) \neq \omega$ for all
$k\in \nbb$, we can assume that $i\Ge 1$. If $i\Ge 2$,
then $\phi(\phi^{j-1}(\omega)) =
\phi(\phi^{i-1}(\omega))$, which in view of Step 1
implies that $\phi^{j-1}(\omega) =
\phi^{i-1}(\omega)$. An induction argument shows that
$\phi^{j-i+1}(\omega) = \phi(\omega)$. Clearly, the
last equality remains valid if $i=1$. Hence, in both
cases, we see that $\omega, \phi^{j-i}(\omega) \in
\phi^{-1}(\{\phi(\omega)\})$. As $\phi(\omega) \neq
\omega$ and consequently
$\card{\phi^{-1}(\{\phi(\omega)\})}=1$, we deduce that
$\omega=\phi^{j-i}(\omega)$, a contradiction. This
proves \eqref{clem}. Set $\varOmega = \{x_i\colon i\in
\zbb_+\}$ with $x_i = \phi^i(\omega)$, and $\varDelta
= \phi^{-1}(\{\omega\}) \setminus \varOmega$. Clearly
$\omega \in \varOmega$, $\phi(\varOmega) \subseteq
\varOmega$ and $\phi^{-1}(\{\omega\}) \cap \varOmega =
\emptyset$. This yields $\card{\varDelta} = \eta+1$.
Therefore, by Step 2, the conditions (s1)-(s5) hold.
Let $\{x_{i,1}\}_{i=1}^{\eta+1}$ be a sequence of
distinct points of $\varDelta$. Setting $x_{i,j} =
\phi^{-(j-1)}(x_{i,1})$ for $i\in J_{\eta+1}$ and $j
\in \nbb$, we verify that $\{x_i\}_{i=0}^\infty$ and
$\{x_{i,j}\}_{i=1}^{\eta+1}{_{j=1}^\infty}$ are
disjoint systems of distinct points of $X$ which
satisfy \eqref{Xr3}. Now, arguing as in the final
stage of the proof of Step 4 with
   \begin{align} \label{defx2} Y :=
\{x_i\colon i \in \zbb_+\} \cup \{x_{i,j}\colon i \in
J_{\eta+1}, j \in \nbb\},
   \end{align}
we complete the proof of Step 5 and Theorem
\ref{golfclub}.
   \end{proof}

Now we make some comments on Theorem \ref{golfclub}
that support the idea of studying composition
operators with symbols of type \mbox{(ii-a)}. We also
shed more light on the question of simplicity of
directed graphs discussed in Section~ \ref{subsec1}.
   \begin{rem} \label{val1}
1) Suppose that the directed graph $(X,E^{\phi})$ is
connected and $C_{\phi}$ is a composition operator in
$L^2(X,2^X,\mu)$, where $\mu$ is a discrete measure on
$X$. By \eqref{hfi} and \cite[Proposition
6.2]{b-j-j-sC}, $C_{\phi}$ is injective if and only if
   \begin{align} \label{manum}
\text{$\card{\phi^{-1}(\{x\})} \Ge 1$ for all $x \in
X$.}
   \end{align}
To guarantee injectivity of $C_{\phi}$, we assume that
\eqref{manum} holds. We also exclude the (more
complex) case when the directed graph $(X,E^{\phi})$
has more than one vertex of valency greater than $1$.
The case when $(X,E^{\phi})$ has exactly one vertex of
valency greater than $1$ has been described in Theorem
\ref{golfclub}. If $\card{\phi^{-1}(\{x\})} = 1$ for
every $x \in X$ (the {\em flat} case), then, by
\cite[Proposition 2.4]{StB} and Step 3 of the proof of
Theorem \ref{golfclub}, $\phi$ is bijectively
isomorphic to the mapping $i+k \zbb \mapsto (i+1) + k
\zbb$ acting on $\zbb_k:=\zbb/k \zbb$ for some $k\in
\zbb_+$. Hence, the composition operators $C_{\phi}$
is unitarily equivalent to an injective bilateral
weighted shift (the case of $\zbb$) or to a bijective
finite dimensional weighted shift (the case of
$\zbb_{k}$ for $k\in \nbb$).

2) Now we assume that the directed graph
$(X,E^{\phi})$ is not connected. Note that if there
exists $\omega \in X$ such that
$\card{\phi^{-1}(\{\omega\})} = \eta + 1$ for some
$\eta \in \nbb \cup \{\infty\}$ and
$\card{\phi^{-1}(\{x\})} = 1$ for every $x \in X
\setminus \{\omega\}$ (hence \eqref{manum} holds),
then $X\setminus Y \neq \emptyset$, where $Y$ is given
either by \eqref{defx0} or by \eqref{defx2} depending
on whether $\omega \in \phi^{-n}(\{\omega\})$ for some
$n\in \nbb$ or not (see the proofs of Steps 4 and 5 of
Theorem \ref{golfclub}). Indeed, this is a consequence
of the easily verifiable fact that $(Y,E^{\phi|_{Y}})$
is a connected subgraph of $(X,E^{\phi})$. Next
observe that $\phi|_{X\setminus Y}$ is a bijective
self-map of $X\setminus Y$ and thus it is bijectively
isomorphic to a disjoint sum of a number of self-maps
$i+k \zbb \mapsto (i+1) + k \zbb$ of $\zbb_k$, where
$k \in \zbb_+$ (see \cite[Proposition 2.4]{StB}).
Clearly, the directed graph $(Y,E^{\phi|_{Y}})$
satisfies the condition (i) of Theorem \ref{golfclub}.
Hence, the composition operator $C_{\phi}$ is
unitarily equivalent to an orthogonal sum of
composition operators whose symbols are described
above (cf.\ \cite[Appendix C]{b-j-j-sS}). One can draw
a similar conclusion for symbols $\phi$ discussed in
the flat case in 1).

3) The directed graph $(X,E^{\phi})$ appearing in
(ii-b) (see Figure $3$) is isomorphic to the directed
tree $\tcal_{\eta+1,\infty}$ defined in
\cite[(6.2.10)]{j-j-s}. By Lemma \ref{jfa-ex}, the
corresponding composition operator $C_{\phi}$ is
unitarily equivalent to a weighted shift on
$\tcal_{\eta+1,\infty}$ with nonzero weights.
Subnormality of such operators has been studied in
\cite{j-j-s,b-j-j-sB}.
   \end{rem}
   \subsection{Injectivity problem}
If we allow the directed graph $(X, E^{\phi})$ to have
vertices of valency $0$, still preserving the
simplicity requirement\footnote{\;As in Remark
\ref{val1}, we exclude from our considerations the
case when $(X, E^{\phi})$ has more than one vertex of
valency greater than $1$.}, then the question is how
many such vertices can there be. The answer is given
in Proposition \ref{ibj1} and in the proof of Theorem
\ref{golfclub2} (see \eqref{hfizfi}). The question
becomes especially interesting when the composition
operator $C_{\phi}$ generates Stieltjes moment
sequences. In this version, the question is a
particular case of a more general problem, called here
the {\em injectivity problem} (see Problem
\ref{open2}). In the present section we investigate
the injectivity problem in the context of directed
graphs $(X,E^{\phi})$ having at most one vertex of
valency greater than $1$.

The following assumption will be used many times in
this section.
   \begin{align} \label{sass}
   \begin{minipage}{65ex}
Let $X$ be a nonempty set, $\phi$ be a self-map of
$X$, $E^{\phi}$ be as in \eqref{ephidef} and
$C_{\phi}$ be a composition operator in
$L^2(X,2^X,\mu)$ with symbol $\phi$, where $\mu$ is a
discrete measure on $X$.
   \end{minipage}
   \end{align}
If \eqref{sass} holds, then we set
   \begin{align*}
Z_{\phi}=\{x\in X\colon \card{\phi^{-1}(\{x\})} = 0\}.
   \end{align*}
Recall that the case of $Z_{\phi}=\emptyset$ has been
discussed in Remark \ref{val1}.

We begin by considering the situation where the
valency of each vertex of the directed graph
$(X,E^{\phi})$ does not exceed $1$.
   \begin{pro}\label{ibj1}
Assume that \eqref{sass} holds. If the directed graph
$(X,E^{\phi})$ is connected and
$\card{\phi^{-1}(\{x\})} \Le 1$ for every $x\in X$,
then
   \begin{enumerate}
   \item[(i)] $\card{Z_{\phi}} \Le 1$,
   \item[(ii)] $C_{\phi}$ is injective whenever
$C_{\phi}$ generates Stieltjes moment sequences.
   \end{enumerate}
   \end{pro}
   \begin{proof}
(i) For if not, there are two distinct vertices $u, v
\in Z_{\phi}$. By the connectivity of $(X,E^{\phi})$,
there exists an undirected path $\{u_i\}_{i=1}^k
\subseteq X$ joining $u$ and $v$ of smallest possible
length $k \Ge 2$ (with $u_1=u$ and $u_k=v$). It is
easily seen that the sequence $\{u_i\}_{i=1}^k$ is
injective. Hence, since $u,v\in Z_{\phi}$, we see that
$k\Ge 3$, $u_2 = \phi(u)$ and $u_{k-1} = \phi(v)$. By
induction, there exists $j \in \{1, \ldots, k-2\}$
such that $u_{j+1}=\phi(u_j)$ and
$u_{j+1}=\phi(u_{j+2})$. As a consequence,
$u_j,u_{j+2} \in \phi^{-1}(\{u_{j+1}\})$, which
contradicts the inequality
$\card{\phi^{-1}(\{u_{j+1}\})}\Le 1$. This proves (i).

(ii) Suppose, on the contrary, that $C_{\phi}$
is not injective, or equivalently that $Z_{\phi} \neq
\emptyset$. By (i), $Z_{\phi}=\{\omega\}$ for some
$\omega \in X$, and thus, by \cite[Proposition
2.4]{StB} and Steps 1 and 3 of the proof of Theorem
\ref{golfclub}, $\phi$ is bijectively isomorphic to
the mapping $i \mapsto i+1$ acting on $\zbb_+$. Hence,
without loss of generality, we can assume that
$X=\zbb_+$ and $\phi(i)=i+1$ for all $i\in \zbb_+$.
Then $\chi_{\{1\}} \in \dzn{C_{\phi}}$, $C_{\phi}
\chi_{\{1\}} \neq 0$ and $C_{\phi}^2 \chi_{\{1\}} =
0$. This contradicts \cite[Lemma 1.1(ii)]{sto-t}
because $C_{\phi}$ generates Stieltjes moment
sequences.
   \end{proof}
It remains to consider the case when the directed
graph $(X,E^{\phi})$ has exactly one vertex of valency
greater than $1$. The following theorem which
describes such graphs (see Figures $4$ and $5$) will
be deduced from Theorem \ref{golfclub}.

   \vspace{2ex}
   \begin{tikzpicture}[scale = .6,
transform shape] \tikzstyle{every node} =
[circle,fill=gray!15] \node (e0)[font=\huge] at (2,-1)
{\phantom{1}$x_0$\phantom{1}}; \node (e-1)[font=\huge]
at (2,3) {$x_{\kappa-1}$}; \node (e-2)[font=\huge] at
(-1,3){$x_{\kappa-2}$}; \node (e1)[font=\huge] at
(-1,-1) {\phantom{1}$x_1$\phantom{1}}; \node
(ek)[font=\huge] at (3.5,1)
{\phantom{L}$x_\kappa$\phantom{l}}; \node
(e11)[font=\huge] at (7,3) {$x_{1,1}$}; \node
(e12)[font=\huge] at (10.5,3) {$x_{1,2}$}; \node
(e21)[font=\huge] at (7,1) {$x_{2,1}$}; \node[fill =
none] (e 5 1)[font=\huge] at (7,-3) {}; \node[fill =
none] (e 5 n)[font=\huge] at (16,-3) {}; \node
(e31)[font=\huge] at (7,-1) {$x_{3,1}$}; \node
(e32)[font=\huge] at (10.5,-1) {$x_{3,2}$}; \node
(e33)[font=\huge] at (14,-1) {$x_{3,3}$}; \node
(e3n)[fill = none] at (16,-1) {}; \draw[<-]
(e11)--(ek) node[pos=0.5,above = 0pt,fill=none] {};
\draw[<-] (e21)--(ek) node[pos=0.5, above =
0pt,sloped, fill=none]{}; \draw[dotted] (e 5 1)--(ek)
node [pos=0.5,below=-10pt,sloped,fill=none] {};
\draw[<-] (e31)--(ek) node[pos=0.5,
below=0pt,fill=none] {}; \draw[<-] (e12)--(e11)
node[pos=0.5,above = 0pt,fill=none] {}; \draw[<-]
(e32)--(e31) node[pos=0.5,below=0pt,fill=none] {};
\draw[<-] (e33)--(e32) node[pos=0.5,below =
0pt,fill=none] {}; \draw[dotted] (e 5 n)--(e 5 1)
node[pos=0.5,above = -10pt,fill=none] {};

\draw[ <-] (ek) to [bend right] (e-1)
node[pos=0.5,above = 0pt,fill=none] {}; \draw[dashed,
<-] (e3n)--(e33); \draw[dashed, <-] (e-2) to [bend
right] (e1); \draw[ <-] (e0) to [bend right] (ek);
\draw[ <-] (e-1)--(e-2) node[pos=0.5,above =
0pt,fill=none] {}; \draw[ <-] (e1)--(e0)
node[pos=0.5,above = 0pt,fill=none] {};
   \end{tikzpicture}
   \begin{center}
   \begin{minipage}{72ex}
{\small {\sf Figure $4$}}. The directed graph
$(X,E^{\phi})$ in the case of \mbox{(ii-a$^*$)} with
$l_1=2$, \\ \phantom{xxxxxxxi}$l_2=1$, $l_3=\infty$,
$\ldots$.
   \end{minipage}
   \end{center}

   \vspace{2ex}
   \begin{tikzpicture}[scale = .6,
transform shape] \tikzstyle{every node} =
[circle,fill=gray!15] \node (e-1)[font=\huge] at (1,1)
{\phantom{l}$x_{1}$\phantom{l}}; \node
(e-2)[font=\huge] at
(-1.5,1){\phantom{l}$x_{2}$\phantom{l}}; \node[fill =
none] (e1)[font=\huge] at (-3.5,1) {}; \node
(ek)[font=\huge] at (3.5,1)
{\phantom{l}$x_0$\phantom{l}}; \node (e11)[font=\huge]
at (7,3) {$x_{1,1}$}; \node (e12)[font=\huge] at
(10.5,3) {$x_{1,2}$}; \node (e21)[font=\huge] at (7,1)
{$x_{2,1}$}; \node (e22)[font=\huge] at (10.5,1)
{$x_{2,2}$}; \node (e23)[font=\huge] at (14,1)
{$x_{2,3}$}; \node[fill = none] (e2n)[font=\huge] at
(16,1) {}; \node[fill = none] (e 5 1)[font=\huge] at
(7,-3) {}; \node[fill = none] (e 5 n)[font=\huge] at
(16,-3) {}; \node (e31)[font=\huge] at (7,-1)
{$x_{3,1}$}; \draw[<-] (e11)--(ek) node[pos=0.5,above
= 0pt,fill=none] {}; \draw[<-] (e21)--(ek)
node[pos=0.5, above = 0pt,sloped, fill=none]{};
\draw[dotted] (e 5 1)--(ek) node
[pos=0.5,below=-10pt,sloped,fill=none] {}; \draw[<-]
(e31)--(ek) node[pos=0.5, below=0pt,fill=none] {};
\draw[<-] (e12)--(e11) node[pos=0.5,above =
0pt,fill=none] {}; \draw[<-] (e22)--(e21)
node[pos=0.5,above = 0pt,fill=none] {}; \draw[<-]
(e23)--(e22) node[pos=0.5,above=0pt,fill=none] {};
\draw[dotted] (e 5 n)--(e 5 1) node[pos=0.5,above =
-10pt,fill=none] {};

\draw[ <-] (ek) to [] (e-1); \draw[dashed, <-]
(e2n)--(e23); \draw[dashed, <-] (e-2) to [] (e1);
\draw[ <-] (e-1)--(e-2) node[pos=0.5,above =
0pt,fill=none] {};
   \end{tikzpicture}
   \begin{center}
   \begin{minipage}{72ex}
{\small {\sf Figure $5$}}. The directed graph
$(X,E^{\phi})$ in the case of \mbox{(ii-b$^*$)} with
$l_1=2$, \\ \phantom{xxxxxxxi} $l_2=\infty$, $l_3=1$,
$\ldots$.
   \end{minipage}
   \end{center}

   \vspace{1ex}
   \begin{thm}\label{golfclub2}
Assume that \eqref{sass} holds and $\eta \in \nbb \cup
\{\infty\}$. Then the following two conditions are
equivalent{\em :}
   \begin{enumerate}
   \item[(i)] the directed graph $(X,E^{\phi})$ is
connected and there exists $\omega \in X$ such that
$\card{\phi^{-1}(\{\omega\})} = \eta + 1$ and
$\card{\phi^{-1}(\{x\})} \Le 1$ for every $x \in X
\setminus \{\omega\}$,
   \item[(ii)] one of the following two conditions
is satisfied{\em :}
   \begin{enumerate}
   \item[(ii-a$^*$)] there exist $\kappa \in \zbb_+$,
a sequence $\{l_i\}_{i=1}^{\eta} \subseteq \nbb \cup
\{\infty\}$ and two disjoint systems
$\{x_i\}_{i=0}^{\kappa}$ and
$\{x_{i,j}\}_{i=1}^{\eta}{_{j=1}^{l_i}}$ of distinct
points of $X$ such that
   \allowdisplaybreaks
   \begin{align*}
X & = \{x_0, \ldots, x_{\kappa}\} \cup
\bigcup_{i=1}^{\eta}\Big\{x_{i,j}\colon j \in
J_{l_i}\Big\},
   \\
\phi(x) & =
   \begin{cases}
x_{i,j-1} & \text{ if } x = x_{i,j} \text{ with } i
\in J_{\eta} \text{ and } \, j \in J_{l_i} \setminus
\{1\},
   \\
x_{\kappa} & \text{ if } x = x_{i,1} \text{ with } i
\in J_{\eta} \text{ or } x = x_0,
   \\
x_{i-1} & \text{ if } x = x_i \text{ with } i \in
J_{\kappa},
   \end{cases}
   \end{align*}
   \item[(ii-b$^*$)] there exist a sequence
$\{l_i\}_{i=1}^{\eta+1} \subseteq \nbb \cup
\{\infty\}$ and two disjoint systems
$\{x_i\}_{i=0}^\infty$ and
$\{x_{i,j}\}_{i=1}^{\eta+1}{_{j=1}^{\,l_i}}$ of
distinct points of $X$ such that
   \allowdisplaybreaks
   \begin{align*}
X & = \{x_i\colon i \in \zbb_+\} \cup
\bigcup_{i=1}^{\eta+1}\Big\{x_{i,j}\colon j \in
J_{l_i}\Big\},
   \\
\phi(x) & =
   \begin{cases}
x_{i,j-1} & \text{ if } x = x_{i,j} \text{ with } i
\in J_{\eta+1} \text{ and } \, j \in J_{l_i} \setminus
\{1\},
   \\
x_0 & \text{ if } x = x_{i,1} \text{ with } i \in
J_{\eta+1},
   \\
x_{i+1} & \text{ if } x = x_i \text{ with } i \in
\zbb_+.
   \end{cases}
   \end{align*}
   \end{enumerate}
   \end{enumerate}
   \end{thm}
   \begin{proof}
(ii)$\Rightarrow$(i) Obvious.

(i)$\Rightarrow$(ii) In view of Theorem
\ref{golfclub}, we may assume that $Z_{\phi} \neq
\emptyset$. Let $\{Y_z\}_{z\in Z_{\phi}}$ be a family
of pairwise disjoint countably infinite sets such that
$X \cap \bigsqcup_{z\in Z_{\phi}} Y_z = \emptyset$.
Set $\widehat X = X \sqcup \bigsqcup_{z\in Z_{\phi}}
Y_z$. For every $z\in Z_{\phi}$, let
$\{y_{z,i}\}_{i=1}^{\infty}$ be a sequence of distinct
points of $Y_z$ such that $Y_z=\{y_{z,i}\colon i \in
\nbb\}$. Define the self-map $\hat \phi$ of $\widehat
X$ by
   \begin{align*}
\hat \phi(x) =
   \begin{cases}
   \phi(x) & \text{if } x\in X,
\\
   z & \text{if } x=y_{z,1} \text{ for some } z\in
   Z_{\phi},
\\
   y_{z,i-1} & \text{if } x=y_{z,i} \text{ for some }
   z\in Z_{\phi} \text{ and } i \Ge 2.
   \end{cases}
   \end{align*}
It is a matter of routine to verify that $\phi
\subseteq \hat\phi$ (i.e., $\hat\phi$ extends $\phi$),
the directed graph $(\widehat X, E^{\hat \phi})$ is
connected, $\hat \phi^{-1}(\{\omega\}) =
\phi^{-1}(\{\omega\})$ and $\card{\hat
\phi^{-1}(\{x\})} = 1$ for every $x \in \widehat X
\setminus \{\omega\}$. Hence, by Theorem
\ref{golfclub}, $(\widehat X, \hat \phi)$ takes the
form \mbox{(ii-a)} or \mbox{(ii-b)} (with $(\widehat
X, \hat \phi)$ in place of $(X, \phi)$). Set $\tilde
\eta=\eta$ if $(\widehat X, \hat \phi)$ takes the form
\mbox{(ii-a)} and $\tilde \eta=\eta+1$ otherwise.
Since $\hat \phi^{-1}(\{\omega\}) =
\phi^{-1}(\{\omega\})$, we deduce that $x_{i,1} \in X$
for every $i \in J_{\tilde \eta}$. This, the explicit
description of $(\widehat X,\hat\phi)$ and an
induction argument combined with $\phi \subseteq
\hat\phi$ imply that there exists a sequence
$\{l_i\}_{i=1}^{\tilde\eta} \subseteq \nbb \cup
\{\infty\}$ such that (ii) holds. In particular, we
have
   \begin{align} \label{hfizfi}
Z_{\phi} = \{x_{i,l_i}\colon i\in J_{\tilde\eta}, \,
l_i <\infty\}.
   \end{align}
This completes the proof.
   \end{proof}

Now we take a closer look at the directed graphs
described by parts \mbox{(ii-a$^*$)} and
\mbox{(ii-b$^*$)} of Theorem \ref{golfclub2} which
admit composition operators generating Stieltjes
moment sequences.
   \begin{pro}\label{golfclub3}
Assume that \eqref{sass} holds. If $(X,E^{\phi})$ is
as in Theorem {\em \ref{golfclub2}}\mbox{\em
(ii-a$^*$)} with $\eta \in \nbb \cup \{\infty\}$ and
$C_{\phi}$ generates Stieltjes moment sequences, then
   \begin{enumerate}
   \item[(i)] $l_i \in \{1\} \cup \{\infty\}$ for every
$i\in J_{\eta}$,
   \item[(ii)] $\card{Z_{\phi}} \Le \eta-1$,
   \item[(iii)] $C_{\phi}$ is injective whenever
the Stieltjes moment sequence
$\{\hfin{n+1}(x_{\kappa})\}_{n=0}^{\infty}$ is
S-determinate.
   \end{enumerate}
   \end{pro}
   \begin{proof}
(i) Indeed, otherwise $l_i \in \nbb_2$ for some $i\in
J_{\eta}$, which implies that
$\chi_{\{x_{i,l_{i}-1}\}}\in \dzn{C_{\phi}}$,
$C_{\phi} \chi_{\{x_{i,l_{i}-1}\}} \neq 0$ and
$C_{\phi}^2 \chi_{\{x_{i,l_{i}-1}\}}=0$. This
contradicts \cite[Lemma 1.1(ii)]{sto-t} because
$C_{\phi}$ generates Stieltjes moment sequences.

(ii) Suppose, on the contrary, that $\card{Z_{\phi}} >
\eta-1$. Then, by \eqref{hfizfi},
$\card{Z_{\phi}}=\eta\in \nbb$ and consequently
$C_{\phi} \in \ogr{L^2(\mu)}$ and $C_{\phi}$ is not
injective. In view of Lambert's theorem (we use
\cite[Theorem 7]{StSz2}), $C_{\phi}$ is subnormal and,
as such, is injective (see \cite[Theorem 9d]{Har-Wh}).
This gives a contradiction and proves (ii).

(iii) By \cite[Theorem 10.4]{b-j-j-sC}, the sequence
$\{\hfin{n}(x)\}_{n=0}^{\infty}$ is a Stieltjes moment
sequence for every $x\in X$. Suppose, on the contrary,
that $C_{\phi}$ is not injective. Using (i), we see
that there exists $i\in J_{\eta}$ such that $l_i=1$.
We easily verify that $\delta_0$ is an S-representing
measure of the Stieltjes moment sequence
$\{\hfin{n}(x_{i,l_i})\}_{n=0}^{\infty}$ (cf.\
\eqref{hfi}). Applying \cite[Lemma 38]{b-j-j-sS} to
$x=x_{\kappa}$, we are led to a contradiction.
   \end{proof}
The following result is an analog of Proposition
\ref{golfclub3}.
   \begin{pro}\label{golfclub4}
Assume that \eqref{sass} holds. If $(X,E^{\phi})$ is
as in Theorem {\em \ref{golfclub2}}\mbox{\em
(ii-b$^*$)} with $\eta \in \nbb \cup \{\infty\}$ and
$C_{\phi}$ generates Stieltjes moment sequences, then
   \begin{enumerate}
   \item[(i)] $l_i \in \{1\} \cup \{\infty\}$ for every
$i\in J_{\eta+1}$,
   \item[(ii)] $\card{Z_{\phi}} \Le \eta$,
   \item[(iii)] $C_{\phi}$ is injective whenever
the Stieltjes moment sequence
$\{\hfin{n+1}(x_0)\}_{n=0}^{\infty}$ is S-determinate.
   \end{enumerate}
   \end{pro}
   \begin{proof}
Arguing as in the proof of Proposition \ref{golfclub3}
(but now applying \cite[Lemma 38]{b-j-j-sS} to
$x=x_0$), we get (i) and (iii).

(ii) Suppose, on the contrary, that $\card{Z_{\phi}} >
\eta$. Then, by \eqref{hfizfi}, $\eta \in \nbb$ and
$\card{Z_{\phi}}=\eta+1$. It follows from (i) that
$l_i=1$ for every $i\in J_{\eta+1}$. As a consequence,
$\chi_{\{x_0\}} \in \dzn{C_{\phi}}$, $C_{\phi}
\chi_{\{x_0\}} \neq 0$ and $C_{\phi}^2 \chi_{\{x_0\}}
= 0$, which contradicts \cite[Lemma 1.1(ii)]{sto-t}
because $C_{\phi}$ generates Stieltjes moment
sequences.
   \end{proof}
Similar reasoning as in the proof of part (iii) of
Proposition \ref{golfclub3} gives a criterion for
injectivity of $C_{\phi}$ in a more general situation.
   \begin{pro}\label{sdeterm}
Suppose that \eqref{sass} holds. If $C_{\phi}$
generates Stieltjes moment sequences and the Stieltjes
moment sequence $\{\hfin{n+1}(x)\}_{n=0}^{\infty}$ is
S-determinate for every $x\in X$, then $C_{\phi}$ is
injective.
   \end{pro}
The above considerations leads us to the following
injectivity problem (see \cite{b-j-j-sC} for the
necessary definitions).
   \begin{opq} \label{open2}
Suppose that $(X, \ascr,\mu)$ is a $\sigma$-finite
measure space, $\phi$ is a nonsingular self-map of $X$
and $C_{\phi}$ is a composition operator in $L^2(\mu)$
with symbol $\phi$ generating Stieltjes moment
sequences. Is it true that $C_{\phi}$ is injective?
   \end{opq}
Problem \ref{open2} seems to be hard to solve. To shed
more light on this we make the following remark.
   \begin{rem}
First, we note that Problem \ref{open2} has an
affirmative solution for bounded composition
operators. Indeed, by Lambert's theorem (we use
\cite[Theorem 7]{StSz2} again), a bounded composition
operator in an $L^2$-space generating Stieltjes moment
sequences is subnormal and consequently, by
\cite[Theorem 9d]{Har-Wh}, it is injective. If the
composition operator in question is over a rootless
directed tree and it has sufficiently many
quasi-analytic vectors, then the property of
generating Stieltjes moment sequences is equivalent to
subnormality (use Lemma \ref{jfa-ex}, its proof and
\cite[Theorem 5.3.1]{b-j-j-sA}). Hence, in view of
\cite[Corollary 6.3]{b-j-j-sC}, in this particular
case, Problem \ref{open2} has an affirmative solution
as well (this can be also deduced from Proposition
\ref{sdeterm} by applying \eqref{chiphi}
and Proposition \ref{wouk}). Propositions \ref{ibj1},
\ref{golfclub3}, \ref{golfclub4} and \ref{sdeterm}
provide yet another examples for which Problem
\ref{open2} has an affirmative solution.
   \end{rem}
   \subsection{The Radon-Nikodym derivatives} \label{Sec3RN}
The following assumption will be used frequently
through this paper.
   \begin{align} \label{sa}
   \begin{minipage}{65ex} Let $\eta \in \nbb
\cup \{\infty\}$ and $\kappa \in \zbb_+$, and let
$X=X_{\eta,\kappa}$ be a set satisfying \eqref{Xr1},
where $\{x_i\}_{i=0}^{\kappa}$ and
$\{x_{i,j}\}_{i=1}^{\eta}{_{j=1}^\infty}$ are two
disjoint systems of distinct points of $X$,
$\phi=\phi_{\eta,\kappa}$ be a self-map of $X$
satisfying \eqref{Xr1.5} and $\mu$ be a discrete
measure on $X$. We adhere to the convention that
$x_{-1} = x_{\kappa}$ and $x_{i,0} = x_{\kappa}$ for
$i \in J_{\eta}$.
   \end{minipage}
   \end{align}
It is easily seen that
   \begin{align} \label{injcos}
\text{\em if \eqref{sa} holds, then the composition
operator $C_{\phi}$ is injective.}
   \end{align}
We begin by deriving a formula for iterated inverses
of $\phi$ at the point $x_{\kappa}$.
   \begin{lem} \label{gc1}
Suppose \eqref{sa} holds. If $n=j(\kappa +1)+ r$ for
some $j \in \zbb_+$ and $r \in \{0, \ldots, \kappa\}$,
then
   \begin{align} \label{inver}
\phi^{-n}(\{x_{\kappa}\}) = \{x_{r-1}\} \cup
\Big\{x_{i,l(\kappa+1) + r}\colon i \in J_{\eta}, \, l
\in \{0, \ldots, j\}\Big\}.
   \end{align}
   \end{lem}
   \begin{proof}
We proceed by induction on $j$. The case of $j=0$ is
easily verified. This and \eqref{Xr1.5} imply that
   \begin{align} \notag
\phi^{-(\kappa+1)}(\{x_{\kappa}\}) =
\phi^{-1}(\phi^{-\kappa}(\{x_{\kappa}\})) & =
\phi^{-1}(\{x_{\kappa-1}\} \cup \{x_{i,\kappa}\colon i
\in J_{\eta}\})
   \\ \label{ciara}
& = \{x_{\kappa}\} \cup \{x_{i,\kappa + 1}\colon i \in
J_{\eta}\}.
   \end{align}
Suppose that \eqref{inver} holds for a fixed $j\in
\zbb_+$. Then, by \eqref{ciara} and induction
hypothesis, we have
   \begin{multline*}
\phi^{-((j+1)(\kappa+1) + r)}(\{x_{\kappa}\}) =
\phi^{-(j(\kappa+1) + r)}(\{x_{\kappa}\} \cup
\{x_{i,\kappa + 1}\colon i \in J_{\eta}\})
   \\
= \{x_{r-1}\} \cup \Big\{x_{i,l(\kappa+1) + r}\colon i
\in J_{\eta}, \, l \in \{0, \ldots, j\}\Big\} \cup
\Big\{x_{i,(j+1)(\kappa+1)+r}\colon i \in
J_{\eta}\Big\},
   \end{multline*}
which completes the proof.
   \end{proof}
Applying \eqref{hfi} to $\phi$ and $\phi^n$ and using
Lemma \ref{gc1}, we can easily calculate the
Radon-Nikodym derivatives $\{\hfi(x)\colon x\in X\}$
and $\{\hfin{n}(x_{\kappa})\colon n\in \zbb_+\}$.
   \begin{pro} \label{hfi2}
Suppose \eqref{sa} holds. Then for every $x\in X$,
   \begin{align} \label{ogrh}
\hfi(x) =
   \begin{cases}
\frac{\mu(x_{j+1})}{\mu(x_j)} & \text{if } x=x_j
\text{ with } j \in \{0, \ldots, \kappa-1\},
\\[1ex]
\frac{\mu(x_0) + \sum_{i=1}^{\eta}
\mu(x_{i,1})}{\mu(x_{\kappa})} & \text{if }
x=x_{\kappa},
\\[1ex]
\frac{\mu(x_{i,j+1})}{\mu(x_{i,j})} & \text{if }
x=x_{i,j} \text{ with } i \in J_{\eta} \text{ and } j
\in \nbb.
   \end{cases}
   \end{align}
If $n=j(\kappa +1)+ r$ for some $j \in \zbb_+$ and $r
\in \{0, \ldots, \kappa\}$, then
   \begin{align} \label{hfi-c}
\hfin{n}(x_{\kappa}) =
   \begin{cases}
1 & \text{ if } n=0,
   \\[1ex]
1 + \sum_{i=1}^{\eta} \sum_{l=1}^j
\frac{\mu(x_{i,l(\kappa+1)})}{\mu(x_{\kappa})} &
\text{ if } j \Ge 1 \text{ and } r=0,
   \\[1ex]
\frac{\mu(x_{r-1})}{\mu(x_{\kappa})} +
\sum_{i=1}^{\eta} \sum_{l=0}^j
\frac{\mu(x_{i,l(\kappa+1)+r})}{\mu(x_{\kappa})} &
\text{ if } r \in J_{\kappa}.
   \end{cases}
   \end{align}
   \end{pro}
Now we calculate the Radon-Nikodym derivatives
$\hfin{n}$, $n\Ge 0$, at the vertices lying on the
circuit.
   \begin{lem}
Suppose \eqref{sa} holds. Then \allowdisplaybreaks
   \begin{align}
\hfin{n+1}(x_{r-1}) & = \frac{\mu(x_r)}{\mu(x_{r-1})}
\hfin{n}(x_r), && r \in J_{\kappa}, \, n \in \zbb_+,
\label{hfir-4}
   \\
\hfin{n+r}(x_0) & = \frac{\mu(x_r)}{\mu(x_0)}
\hfin{n}(x_r), && r \in \{0, \ldots, \kappa\}, \, n
\in \zbb_+. \label{hfir-2}
   \end{align}
   \end{lem}
   \begin{proof}
If $r \in J_{\kappa}$ and $n \in \zbb_+$, then by
\eqref{hfi} and \eqref{Xr1.5} we have
   \allowdisplaybreaks
   \begin{align*}
\hfin{n+1}(x_{r-1}) &=
   \frac{\mu(\phi^{-n}(\phi^{-1}(\{x_{r-1}\})))}
   {\mu(x_{r-1})}
   \\
&= \frac{\mu(\phi^{-n}(\{x_{r}\}))}{\mu(x_{r})}
\frac{\mu(x_{r})}{\mu(x_{r-1})}
   \\
& = \frac{\mu(x_{r})}{\mu(x_{r-1})} \hfin{n}(x_{r}),
   \end{align*}
which gives \eqref{hfir-4}. Applying induction on $r$
and \eqref{hfir-4}, we obtain \eqref{hfir-2}.
   \end{proof}
   The subsequent lemma plays an essential role in the
present paper.
   \begin{lem} \label{fund}
If \eqref{sa} holds, then
   \begin{align} \label{dif1}
\hfin{n + \kappa+1}(x_0) = \hfin{n}(x_0) +
\sum_{i=1}^{\eta} \frac{\mu(x_{i,1})}{\mu(x_0)}
\hfin{n}(x_{i,1}), \quad n \in \zbb_+.
   \end{align}
   \end{lem}
   \begin{proof} Observe that by
\eqref{hfi} and \eqref{Xr1.5} we have
\allowdisplaybreaks
 \begin{align*} \hfin{n +
\kappa+1}(x_0) & =
\frac{\mu(\phi^{-n}(\phi^{-1}(\{x_\kappa\})))}{\mu(x_0)}
   \\
&= \frac{\mu(\phi^{-n}(\{x_0\} \sqcup \{x_{i,1}\colon
i \in J_{\eta}\}))}{\mu(x_0)}
   \\
& = \hfin{n}(x_0) + \sum_{i=1}^{\eta}
\frac{\mu(\phi^{-n}(\{x_{i,1}\}))}{\mu(x_0)}
   \\
& = \hfin{n}(x_0) + \sum_{i=1}^{\eta}
\frac{\mu(x_{i,1})}{\mu(x_0)} \hfin{n}(x_{i,1}), \quad
n \in \zbb_+,
   \end{align*}
which completes the proof of \eqref{dif1}.
   \end{proof}
The question of density of domains of powers of
$C_{\phi}$ can be answered in terms of the
Radon-Nikodym derivatives $\hfin{n}$, $n\Ge 0$,
calculated at $x_0$.
   \begin{pro} \label{hfi4}
Suppose \eqref{sa} holds and $n\in \nbb$. Then the
following conditions are equivalent{\em :}
   \begin{enumerate}
   \item[(i)] $\dz{C_{\phi}^n}$ is dense in $L^2(\mu)$,
   \item[(ii)] $\hfin{n+r}(x_0) < \infty$ for every
$r \in \{0, \ldots, \kappa\}$.
   \end{enumerate}
Moreover, if $r\in \{0,\ldots,\kappa\}$, then the
following conditions are equivalent{\em :}
   \begin{enumerate}
   \item[(iii)] $\dzn{C_{\phi}}$ is dense in $L^2(\mu)$,
   \item[(iv)] $\dz{C_{\phi}^j}$ is dense in $L^2(\mu)$
for all $j\in \nbb$,
   \item[(v)] $\hfin{j}(x) < \infty$ for all $j \in \nbb$
and $x\in X$,
   \item[(vi)] $\hfin{j}(x_r) < \infty$ for all $j \in
\nbb$.
   \end{enumerate}
   \end{pro}
   \begin{proof}
(i)$\Rightarrow$(ii) By \eqref{dzn},
$\hfin{l}(x_{\kappa}) < \infty$ for all $l \in \{0,
\ldots, n\}$, and thus, by \eqref{hfir-2} with
$r=\kappa$, $\hfin{l+\kappa}(x_0) < \infty$ for all $l
\in \{0, \ldots, n\}$. Since, by \eqref{hfi},
$\hfin{l}(x_0) < \infty$ for every $l \in \{0, \ldots,
\kappa-1\}$, we see that $\hfin{l}(x_0) < \infty$ for
every $l \in \{0, \ldots, n+\kappa\}$.

(ii)$\Rightarrow$(i) Applying \eqref{hfir-2}, we
deduce that $\hfin{n}(x_r) < \infty$ for every $r\in
\{0, \ldots, \kappa\}$. It follows from \eqref{hfi}
that $\hfin{n}(x_{i,j}) < \infty$ for all $i\in
J_{\eta}$ and $j\in \nbb$. This, \eqref{dzn} and
\eqref{Xr1} yield (i).

Now we prove the ``moreover'' part. By \eqref{dzn} and
\cite[Theorem 4.7]{b-j-j-sC}, it suffices to prove
that (vi) implies (iv). It follows from \eqref{hfir-2}
that $\hfin{j}(x_0) < \infty$ for all integers $j\Ge
\kappa$. Applying \eqref{hfi}, we deduce that
$\hfin{j}(x_0) < \infty$ for all $j \in \zbb_+$.
Hence, by the implication (ii)$\Rightarrow$(i),
$\dz{C_{\phi}^j}$ is dense in $L^2(\mu)$ for all $j
\in \nbb$.
   \end{proof}
   \begin{cor}  \label{porow}
Suppose \eqref{sa} holds and $n\in \nbb$. Then the
following conditions are equivalent{\em :}
   \begin{enumerate}
   \item[(i)] $\overline{\dz{C_{\phi}^n}} = L^2(\mu)$ and
$\overline{\dz{C_{\phi}^{n+1}}} \varsubsetneq
L^2(\mu)$,
   \item[(ii)] $\hfin{n+r}(x_0) < \infty$ for every
$r \in \{0, \ldots, \kappa\}$ and
$\hfin{n+\kappa+1}(x_0) = \infty$.
   \end{enumerate}
   \end{cor}
   \section{{\bf Subnormality of $C_{\phi_{\eta,\kappa}}$
via the consistency condition \eqref{cc}}}
   \subsection{Characterizations of \eqref{cc}}
This section deals with the consistency condition
which, according to Theorem \ref{glowne},
automatically implies subnormality of
$C_{\phi_{\eta,\kappa}}$ (because, under our standing
assumption \eqref{sa}, ${\mathsf
h}_{\phi_{\eta,\kappa}}(x) > 0$ for all $x\in X$).
   \begin{thm} \label{cc1}
Suppose \eqref{sa} holds, $C_{\phi}$ is densely
defined and $\{P(x,\cdot)\}_{x\in X}$ is a family of
Borel probability measures on $\rbb_+$. Then
$\{P(x,\cdot)\}_{x\in X}$ satisfies \eqref{cc} if and
only if the following three conditions are
satisfied{\em :}
   \begin{enumerate}
   \item[(i)] $P(x_r,\sigma) = \frac{\mu(x_0)}{\mu(x_r)}
\int_{\sigma} t^r P(x_0, \D t)$ for all $r \in
\{0,\ldots, \kappa\}$ and $\sigma \in \borel{\rbb_+}$,
   \item[(ii)] $\sum_{i=1}^{\eta} \frac{\mu(x_{i,1})}{\mu(x_0)}
P(x_{i,1}, \sigma) = \int_{\sigma} (t^{\kappa+1}-1)
P(x_0,\D t)$ for all $\sigma \in \borel{\rbb_+}$,
   \item[(iii)] $P(x_{i,j},\sigma) = \frac{\mu(x_{i,1})}
{\mu(x_{i,j})} \int_{\sigma} t^{j-1} P(x_{i,1}, \D t)$
for all $i \in J_{\eta}$, $j \in \nbb_2$ and $\sigma
\in \borel{\rbb_+}$.
   \end{enumerate}
Moreover, if $\{P(x,\cdot)\}_{x\in X}$ satisfies
\eqref{cc}, then
   \begin{enumerate}
   \item[(iv)] $P(x_r,[0,1)) = P(x_{i,j},[0,1]) = 0$
for all $r \in \{0,\ldots, \kappa\}$, $i \in J_{\eta}$
and $j \in \nbb$,
   \item[(v)] $P(x_0,\sigma \cap (1,\infty)) =
\sum_{i=1}^{\eta} \frac{\mu(x_{i,1})} {\mu(x_0)}
\int_{\sigma} \frac{1}{t^{\kappa+1}-1} P(x_{i,1},\D
t)$ for all $\sigma \in \borel{\rbb_+}$,
   \item[(vi)] $\sum_{i=1}^{\eta} \frac{\mu(x_{i,1})}
{\mu(x_0)} \int_0^\infty \frac{1}{t^{\kappa+1}-1}
P(x_{i,1},\D t) \Le 1$,
   \item[(vii)] $P(x_0,\{1\}) =\vartheta:= 1-
\sum_{i=1}^{\eta} \frac{\mu(x_{i,1})} {\mu(x_0)}
\int_0^\infty \frac{1}{t^{\kappa+1}-1} P(x_{i,1},\D
t)$,
   \item[(viii)] $P(x_0,\sigma) =
\sum_{i=1}^{\eta} \frac{\mu(x_{i,1})} {\mu(x_0)}
\int_{\sigma} \frac{1}{t^{\kappa+1}-1} P(x_{i,1},\D t)
+ \vartheta \delta_1(\sigma)$ for all $\sigma \in
\borel{\rbb_+}$,
   \item[(ix)] $\sum_{i=1}^{\eta} \frac{\mu(x_{i,1})}
{\mu(x_0)} \int_0^\infty
\frac{t^{\kappa+1}}{t^{\kappa+1}-1} P(x_{i,1},\D t)
<\infty$.
   \end{enumerate}
   \end{thm}
   \begin{proof} Since, by  \eqref{dzn},  $\hfi(x_{\kappa}) < \infty$,
we infer from Proposition \ref{hfi2} that
   \begin{align} \label{sumi1}
\sum_{i=1}^{\eta} \mu(x_{i,1}) < \infty.
   \end{align}

Assume now that $\{P(x,\cdot)\}_{x\in X}$ satisfies
\eqref{cc}. Substituting $x=x_r$ with $r\in \{0,
\ldots, \kappa-1\}$ into \eqref{cc}, we get
   \begin{align*}
P(x_{r+1},\sigma) = \frac{\mu(x_r)}{\mu(x_{r+1})}
\int_{\sigma} t P(x_r, \D t), \quad \sigma \in
\borel{\rbb_+}, \, r\in \{0, \ldots, \kappa-1\}.
   \end{align*}
An induction argument shows that (i) holds.
Substituting $x=x_\kappa$ into \eqref{cc} and using
(i), we obtain
   \begin{align*}
\mu(x_0) P(x_0, \sigma) + \sum_{i=1}^{\eta}
\mu(x_{i,1}) P(x_{i,1}, \sigma) & = \mu(x_{\kappa})
\int_{\sigma} t P(x_{\kappa}, \D t)
   \\
& = \mu(x_0) \int_{\sigma} t^{\kappa+1} P(x_0, \D t),
\quad \sigma \in \borel{\rbb_+}.
   \end{align*}
This implies (ii). Substituting $x=x_{i,j}$ into
\eqref{cc} yields
   \begin{align*}
P(x_{i,j+1},\sigma) =
\frac{\mu(x_{i,j})}{\mu(x_{i,j+1})} \int_{\sigma} t
P(x_{i,j}, \D t), \quad \sigma \in \borel{\rbb_+}, \,
i \in J_{\eta}, \, j \in \nbb.
   \end{align*}
An induction argument leads to (iii).

Similar reasoning shows that the conditions (i)-(iii)
imply that $\{P(x,\cdot)\}_{x\in X}$ satisfies
\eqref{cc}.

To prove the ``moreover'' part, we assume that
$\{P(x,\cdot)\}_{x\in X}$ satisfies the condition
\eqref{cc}. Since $P(x, \rbb_+)=1$ for all $x \in X$,
we deduce from (ii) and \eqref{sumi1} that
$\int_0^\infty |t^{\kappa+1}-1| \, P(x_0,\D t) <
\infty$. Hence, by (ii) again, $P(x_{i,1},[0,1]) = 0$
for all $i\in J_{\eta}$ and $P(x_0,[0,1)) = 0$.
Applying (i) and (iii) gives (iv). It follows from
(ii) that
   \begin{align}  \label{istn}
\sum_{i=1}^{\eta} \frac{\mu(x_{i,1})}{\mu(x_0)}
P(x_{i,1}, \sigma) = \int_{\sigma \cap (1,\infty)}
(t^{\kappa+1}-1) P(x_0,\D t), \quad \sigma \in
\borel{\rbb_+}.
   \end{align}
Using (iv) and \eqref{istn} and integrating the
function $t \mapsto
\frac{\chi_{\sigma}(t)}{t^{\kappa+1}-1}$ with respect
to the measure $\sum_{i=1}^{\eta}
\frac{\mu(x_{i,1})}{\mu(x_0)} P(x_{i,1}, \cdot)$, we
obtain (v). Since $P(x_0,\rbb_+)=1$, the conditions
(vi) and (vii) follow from (v). The equality (viii) is
a direct consequence of (iv), (v) and (vii). Finally,
integrating the function $t \mapsto t^{\kappa+1}$ with
respect to the (positive) measure $\sigma \mapsto
\sum_{i=1}^{\eta} \frac{\mu(x_{i,1})} {\mu(x_0)}
\int_{\sigma} \frac{1}{t^{\kappa+1}-1} P(x_{i,1},\D
t)$, we deduce from (v) and \eqref{twomom} that
   \begin{align*}
\sum_{i=1}^{\eta} \frac{\mu(x_{i,1})} {\mu(x_0)}
\int_0^\infty \frac{t^{\kappa+1}}{t^{\kappa+1}-1}
P(x_{i,1},\D t) \Le \hfin{\kappa+1}(x_0).
   \end{align*}
Applying Proposition \ref{hfi4} with $n=1$ and
$r=\kappa$, we get (ix).
   \end{proof}
The following proposition provides new criteria for
$C_{\phi_{\eta,\kappa}}$ to have densely defined $n$th
power (cf.\ Proposition \ref{hfi4}).
   \begin{pro}\label{wk}
Suppose \eqref{sa} holds, the composition operator
$C_{\phi}$ is densely defined and
$\{P(x,\cdot)\}_{x\in X}$ is a family of Borel
probability measures on $\rbb_+$ that satisfies
\eqref{cc}. Then for every $n\in \nbb$, the following
conditions are equivalent{\em :}
   \begin{enumerate}
   \item[(i)] $C_{\phi}^n$ is densely defined,
   \item[(ii)] $\hfin{n+\kappa}(x_0) < \infty$,
   \item[(iii)] $\sum_{i=1}^{\eta} \mu(x_{i,1})
\int_0^\infty \frac{t^{n+\kappa}}{t^{\kappa+1}-1}
P(x_{i,1},\D t) <\infty$.
   \end{enumerate}
   \end{pro}
   \begin{proof}
By \eqref{twomom} and Theorem \ref{cc1}(iv), the
sequence $\{\hfin{j}(x_0)\}_{n=0}^\infty$ is
monotonically increasing. Hence, by Proposition
\ref{hfi4}, the conditions (i) and (ii) are
equivalent. Integrating the function $t \mapsto
t^{n+\kappa}$ with respect to the measure
$P(x_0,\cdot)$ and using \eqref{twomom} and Theorem
\ref{cc1}(viii), we deduce that the conditions (ii)
and (iii) are equivalent.
   \end{proof}
As a consequence of Theorem \ref{cc1}(iv) and
Proposition \ref{wk}, we have the following corollary
(cf.\ Corollary \ref{porow}).
   \begin{cor} \label{re-m}
Suppose \eqref{sa} holds, $C_{\phi}$ is densely
defined and $\{P(x,\cdot)\}_{x\in X}$ is a family of
Borel probability measures on $\rbb_+$ that satisfies
\eqref{cc}. Let $n\in \nbb$. Then
$\overline{\dz{C_{\phi}^{n+1}}} \varsubsetneq
\overline{\dz{C_{\phi}^{n}}} = L^2(\mu)$ if and only
if the following two conditions hold:
   \begin{enumerate}
   \item[(i)] $\sum_{i=1}^{\eta} \mu(x_{i,1})
\int_0^\infty \frac{t^{n+\kappa}}{t^{\kappa+1}-1}
P(x_{i,1},\D t) <\infty$,
   \item[(ii)]  $\sum_{i=1}^{\eta} \mu(x_{i,1})
\int_0^\infty \frac{t^{n+\kappa+1}}{t^{\kappa+1}-1}
P(x_{i,1},\D t) = \infty$.
   \end{enumerate}
   \end{cor}
   \subsection{Modelling subnormality via \eqref{cc}}
In Procedure \ref{proceed} below, we propose a method
of constructing all possible subnormal composition
operators $C_{\phi_{\eta,\kappa}}$ in
$L^2(X_{\eta,\kappa}, \mu)$ that satisfy \eqref{cc} in
the meaning that they admit families of probability
measures satisfying \eqref{cc}. The starting point of
our procedure is a family $\{P(x_{i,1},\cdot)\}_{i\in
J_{\eta}}$ of Borel probability measures on $\rbb_+$
that satisfies the conditions \eqref{i10}-\eqref{i12}
below. Let us point out that if a densely defined
$C_{\phi_{\eta,\kappa}}$ admits a family
$\{P(x,\cdot)\}_{x\in X_{\eta,\kappa}}$ of Borel
probability measures on $\rbb_+$ that satisfies
\eqref{cc}, then, by Theorem \ref{cc1}, the measures
$P(x_{i,1}, \cdot)$, $i\in J_{\eta}$, satisfy the
conditions \eqref{i10}-\eqref{i12}.
   \begin{proc} \label{proceed}
Let $\eta \in \nbb \cup \{\infty\}$, $\kappa \in
\zbb_+$, $\phi$ be a self-map of a set $X$,
$\{x_i\}_{i=0}^{\kappa}$ and
$\{x_{i,j}\}_{i=1}^{\eta}{_{j=1}^\infty}$ be two
disjoint systems of distinct points of $X$ that
satisfy \eqref{Xr1} and \eqref{Xr1.5}. Let
$\{P(x_{i,1},\cdot)\}_{i\in J_{\eta}}$ be a family of
Borel probability measures on $\rbb_+$ that satisfies
the following three conditions: \allowdisplaybreaks
   \begin{gather}  \label{i10}
P(x_{i,1}, [0,1]) = 0, \quad i \in J_{\eta},
   \\ \label{i11}
\int_0^\infty t^j P(x_{i,1}, \D t) < \infty, \quad j
\in \nbb, \,i \in J_{\eta},
   \\ \label{i12}
\int_0^\infty \frac{t^{\kappa+1}}{t^{\kappa+1}-1}
P(x_{i,1}, \D t) < \infty, \quad i \in J_{\eta}.
   \end{gather}
Let $\{\mu(x_{i,1})\}_{i \in J_{\eta}}$ be a family of
positive real numbers such that
   \begin{align} \label{i13}
\sum_{i=1}^{\eta} \mu(x_{i,1}) \int_0^\infty
\frac{t^{\kappa+1}}{t^{\kappa+1}-1} P(x_{i,1}, \D t) <
\infty.
   \end{align}
It follows from \eqref{i10} and \eqref{i13} that
   \begin{align} \label{i14}
0 \Le \sum_{i=1}^{\eta} \mu(x_{i,1}) \int_0^\infty
\frac{t^r}{t^{\kappa+1}-1} P(x_{i,1}, \D t) < \infty,
\quad r \in \{0,\ldots, \kappa+1\}.
   \end{align}
Using \eqref{i10} and \eqref{i14}, we get
   \begin{align}    \notag
\sum_{i=1}^{\eta} \mu(x_{i,1}) & = \sum_{i=1}^{\eta}
\mu(x_{i,1}) \int_0^\infty
\frac{t^{\kappa+1}}{t^{\kappa+1}-1} P(x_{i,1},\D t)
   \\  \label{i15}
& \hspace{6ex}- \sum_{i=1}^{\eta} \mu(x_{i,1})
\int_0^\infty \frac{1}{t^{\kappa+1}-1} P(x_{i,1},\D t)
< \infty.
   \end{align}
Now, by \eqref{i14}, we can take $\mu(x_0) \in
(0,\infty)$ such that
   \begin{align} \label{tet-a}
0 \Le \varTheta:=\sum_{i=1}^{\eta}
\frac{\mu(x_{i,1})}{\mu(x_0)} \int_0^\infty
\frac{1}{t^{\kappa+1}-1} P(x_{i,1},\D t) \Le 1.
   \end{align}
Then we define the Borel measure $P(x_0, \cdot)$ on
$\rbb_+$ by
   \begin{align} \label{p-0}
P(x_0, \sigma) = \sum_{i=1}^{\eta}
\frac{\mu(x_{i,1})}{\mu(x_0)} \int_{\sigma}
\frac{1}{t^{\kappa+1}-1} P(x_{i,1},\D t) +
(1-\varTheta) \delta_1(\sigma), \quad \sigma \in
\borel{\rbb_+}.
   \end{align}
By \eqref{i10} and \eqref{tet-a}, $P(x_0, \cdot)$ is a
probability measure such that $P(x_0, [0,1))=0$.
Moreover, $P(x_0, \cdot)$ satisfies the condition (ii)
of Theorem \ref{cc1}. Since $P(x_{i,1}, \cdot)$, $i
\in J_{\eta}$, are probability measures, we infer from
\eqref{i10} and \eqref{i11} that $0 < \int_0^\infty
t^j P(x_{i,1}, \D t) < \infty$ for all $j \in \nbb$
and $i \in J_{\eta}$. This enables us to define the
family $\{\mu(x_{i,j})\}_{i=1}^{\eta}{}_{j=2}^\infty$
of positive real numbers by
   \begin{align} \label{miii}
\mu(x_{i,j}) = \mu(x_{i,1}) \int_0^\infty t^{j-1}
P(x_{i,1}, \D t), \quad i \in J_{\eta}, \, j \in
\nbb_2,
   \end{align}
and the family
$\{P(x_{i,j},\cdot)\}_{i=1}^{\eta}{}_{j=2}^\infty$ of
Borel measures on $\rbb_+$ by
   \begin{align*}
P(x_{i,j},\sigma) = \frac{\mu(x_{i,1})}{\mu(x_{i,j})}
\int_{\sigma} t^{j-1} P(x_{i,1},\D t), \quad i \in
J_{\eta}, \, j \in \nbb_2, \, \sigma \in
\borel{\rbb_+}.
   \end{align*}
In view of \eqref{miii}, the family
$\{P(x_{i,j},\cdot)\}_{i=1}^{\eta}{}_{j=2}^\infty$
consists of probability measures. According to
\eqref{i10}, \eqref{i14} and \eqref{p-0}, $0 <
\int_0^\infty t^{r} P(x_0,\D t) < \infty$ for every $r
\in \{0, \ldots, \kappa\}$. Hence, we can define
positive real numbers $\mu(x_r)$, $r \in J_{\kappa}$,
via
   \begin{align*}
\mu(x_r) = \mu(x_0) \int_0^\infty t^{r} P(x_0,\D t),
\quad r \in J_{\kappa}.
   \end{align*}
As a consequence, the measures $P(x_r, \cdot)$, $r \in
J_{\kappa}$, defined by
   \begin{align*}
P(x_r, \sigma) = \frac{\mu(x_0)}{\mu(x_r)}
\int_{\sigma} t^r P(x_0, \D t), \quad r \in
J_{\kappa}, \, \sigma \in \borel{\rbb_+},
   \end{align*}
are Borel probability measures on $\rbb_+$. Let $\mu$
be the discrete measure on $X$ such that
$\mu(\{x_r\})=\mu(x_r)$ and $\mu(\{x_{i,j}\}) =
\mu(x_{i,j})$ for all $r \in \{0,\ldots,\kappa\}$,
$i\in J_{\eta}$ and $j\in \nbb$, and let $C_\phi$ be
the corresponding composition operator in $L^2(\mu)$
with $\phi=\phi_{\eta,\kappa}$. By \eqref{dzn},
\eqref{ogrh} and \eqref{i15}, $C_{\phi}$ is densely
defined. Applying Theorem \ref{cc1}, we see that the
family $\{P(x,\cdot)\}_{x\in X}$ satisfies \eqref{cc}.
Hence, by Theorem \ref{glowne}, $C_{\phi}$ is
subnormal.

Our procedure enables us to model all subnormal
composition operators $C_{\phi_{\eta,\kappa}}$ that
satisfy \eqref{cc} and have densely defined $n$th
power ($n$ is a fixed positive integer). It suffices
to replace \eqref{i12} by the condition
   \begin{gather} \label{i12n-1}
\int_0^\infty \frac{t^{\kappa+n}}{t^{\kappa+1}-1}
P(x_{i,1},\D t) < \infty, \quad i \in J_{\eta},
   \end{gather}
(leaving the assumptions \eqref{i10} and \eqref{i11}
unchanged) and to choose a family $\{\mu(x_{i,1})\}_{i
\in J_{\eta}} \subseteq (0,\infty)$ that satisfies, in
place of \eqref{i13}, the following inequality
   \begin{gather} \label{i12n}
\sum_{i=1}^{\eta} \mu(x_{i,1}) \int_0^\infty
\frac{t^{\kappa+n}}{t^{\kappa+1}-1} P(x_{i,1},\D t) <
\infty.
   \end{gather}
Indeed, by \eqref{i10}, the conditions \eqref{i12n-1}
and \eqref{i12n} imply \eqref{i12} and \eqref{i13},
respectively. On the other hand, under the assumptions
\eqref{i10}-\eqref{i12}, $C_{\phi}^n$ is densely
defined if and only if \eqref{i12n} holds (cf.\
Proposition \ref{wk}). As a consequence (see also
Corollary \ref{re-m}), $\overline{\dz{C_{\phi}^{n+1}}}
\varsubsetneq \overline{\dz{C_{\phi}^{n}}} = L^2(\mu)$
if and only if both \eqref{i12n} and \eqref{i12nN}
hold, where
   \begin{align} \label{i12nN}
\sum_{i=1}^{\eta} \mu(x_{i,1}) \int_0^\infty
\frac{t^{\kappa+n+1}}{t^{\kappa+1}-1} P(x_{i,1},\D t)
= \infty.
   \end{align}
   \end{proc}
Using Procedure \ref{proceed}, we will show that for
every $n\in \nbb$, there exists a subnormal
composition operator $C_{\phi}$ such that $C_{\phi}^n$
is densely defined, while $C_{\phi}^{n+1}$ is not.
Examples of this kind have been given in
\cite{b-d-j-s} by using weighted shifts on directed
trees (see also a recent paper \cite{b-j-j-sSq} for
more subtle examples).
   \begin{exa}
Fix $n \in \nbb$. Consider a sequence $\{P(x_{i,1},
\cdot)\}_{i=1}^\infty$ of Borel probability measures
on $\rbb_+$ given by $P(x_{i,1},
\sigma)=\delta_{i+1}(\sigma)$ for all $\sigma \in
\borel{\rbb_+}$ and $i \in \nbb$. Set $\mu(x_{i,1}) =
\frac{1}{(i+1)^{n+1}}$ for $i \in \nbb$. It is now a
routine matter to verify that the conditions
\eqref{i10}, \eqref{i11}, \eqref{i12n} and
\eqref{i12nN} hold for $\eta=\infty$ and for arbitrary
$\kappa \in \zbb_+$. Hence, applying Procedure
\ref{proceed}, we get a composition operator
$C_{\phi}$ with the required properties, i.e.,
$\overline{\dz{C_{\phi}^{n+1}}} \varsubsetneq
\overline{\dz{C_{\phi}^{n}}} = L^2(\mu)$. Note that by
\eqref{hfi-c} and Proposition \ref{hfi4}, $C_{\phi}^j$
is densely defined for every $j \in \nbb$ whenever
$\eta < \infty$.
   \end{exa}
   \subsection{Criteria for subnormality related to
$x_0$} \label{sect4.3}
   In this section we give criteria for subnormality
of composition operators $C_{\phi}$ in $L^2(X, \mu)$
with $X=X_{\eta,\kappa}$ and $\phi=\phi_{\eta,\kappa}$
written in terms of the Radon-Nikodym derivatives
$\{\hfin{n}\}_{n=0}^\infty$ calculated at the points
$x_0$ and $x_{i,1}$, $i\in J_\eta$ (cf.\ Theorem
\ref{suff}). Surprisingly, in the case of $\eta=1$ the
subnormality of $C_{\phi}$ can be inferred from the
behaviour of $\{\hfin{n}\}_{n=0}^\infty$ only at the
point $x_0$ (cf.\ Proposition \ref{eta1}). We begin by
stating two necessary lemmata.
   \begin{lem}\label{det} Let
$\{\gamma_n\}_{n=0}^\infty$ be a Stieltjes moment
sequence and let $p \in \nbb$. Then
$\{\gamma_{jp}\}_{j=0}^\infty$ is a Stieltjes moment
sequence and the following assertions hold{\em :}
   \begin{enumerate}
   \item[(i)] if
$\{\gamma_{jp}\}_{j=0}^\infty$ is S-determinate, then
so is $\{\gamma_n\}_{n=0}^\infty$,
   \item[(ii)] if $\{\gamma_{jp}\}_{j=0}^\infty$ is
S-determinate and $\{\gamma_{(j+1)p} -
\gamma_{jp}\}_{j=0}^\infty$ is a Stieltjes moment
sequence, then $\{\gamma_n\}_{n=0}^\infty$ is
S-determinate and its unique S-representing measure
vanishes on $[0,1)$,
   \item[(iii)] if $\{\gamma_{n + 1} -
\gamma_{n}\}_{n=0}^\infty$ is a Stieltjes moment
sequence, then $\{\gamma_{n}\}_{n=0}^\infty$ satisfies
the Carleman condition if and only if $\{\gamma_{n+1}
- \gamma_{n}\}_{n=0}^\infty$ satisfies the Carleman
condition.
   \end{enumerate}
   \end{lem}
   \begin{proof}
(i) Let $\rho$ be an S-representing measure of
$\{\gamma_n\}_{n=0}^\infty$, $W\colon \rbb_+ \to
\rbb_+$ be a function given by $W(t) = t^{p}$ for $t
\in \rbb_+$, and $\rho \circ W^{-1}$ be a Borel
measure on $\rbb_+$ given by $\rho \circ
W^{-1}(\sigma) = \rho(W^{-1}(\sigma))$ for $\sigma \in
\borel{\rbb_+}$. Using the measure transport theorem,
we see that $\{\gamma_{jp}\}_{j=0}^\infty$ is a
Stieltjes moment sequence with the S-representing
measure $\rho \circ W^{-1}$. If $\rho^{\prime}$ is
another S-representing measure of
$\{\gamma_n\}_{n=0}^\infty$, then the measure
$\rho^{\prime} \circ W^{-1}$, being an S-representing
measure of $\{\gamma_{jp}\}_{j=0}^\infty$, coincides
with $\rho \circ W^{-1}$, and consequently $\rho =
\rho^{\prime}$.

(ii) In view of (i), $\{\gamma_n\}_{n=0}^\infty$ is
S-determinate. Denote by $\rho$ its unique
S-represent\-ing measure. Let $\nu$ be an
S-representing measure of $\{\gamma_{(j+1)p} -
\gamma_{jp}\}_{j=0}^\infty$. Then
   \begin{align*}
\sum_{i,j=0}^n \big(\gamma_{(i+j+1)p} -
\gamma_{(i+j)p}\big) \lambda_i \bar \lambda_j =
\int_0^\infty \Big|\sum_{j=0}^n \lambda_j t^{j}\Big|^2
\D \nu(t) \Ge 0
   \end{align*}
for all finite sequences $\{\lambda_j\}_{j=0}^n
\subseteq \cbb$. By Lemma \ref{kontr} and the
S-determinacy of $\{\gamma_{jp}\}_{j=0}^\infty$, we
deduce that $\rho\circ W^{-1}([0,1))=0$. Hence
$\rho([0,1))=0$.

(iii) Set $\varDelta_n = \gamma_{n+1} - \gamma_{n}$
for $n \in \zbb_+$. Since $\varDelta_n \Le
\gamma_{n+1}$ for all $n \in \zbb_+$, we infer from
Proposition \ref{wouk}(ii) that if
$\{\gamma_{n}\}_{n=0}^\infty$ satisfies the Carleman
condition, then so does
$\{\varDelta_n\}_{n=0}^\infty$. To prove the converse
implication assume that $\{\varDelta_n\}_{n=0}^\infty$
satisfies the Carleman condition. Note that
   \begin{align} \label{didif}
\gamma_{n} = \sum_{l=0}^{n-1} \varDelta_l + \gamma_0,
\quad n \in \nbb.
   \end{align}
Let $\tau$ be an S-representing measure of
$\{\varDelta_n\}_{n=0}^\infty$. Set
$\varDelta_n^{\prime} = \int_{(1,\infty)} t^n \D
\tau(t)$ for $n \in \zbb_+$. If $\tau((1,\infty)) =
0$, then, by \eqref{didif},
   \begin{align} \label{baegam}
\gamma_n \Le n \big(\tau([0,1]) + \gamma_0\big), \quad
n \in \nbb.
   \end{align}
If $\tau((1,\infty)) > 0$, then using \eqref{didif}
and the fact that the sequence
$\{\varDelta_n^{\prime}\}_{n=0}^\infty$ is
monotonically increasing, we obtain
   \begin{align} \label{baegam2}
\gamma_n \Le n \big(\tau([0,1]) + \gamma_0 +
\varDelta_n^{\prime}\big) \Le n
\bigg(\frac{\tau([0,1]) + \gamma_0}{\tau((1,\infty))}
+ 1\bigg) \varDelta_n, \quad n \in \nbb.
   \end{align}
Combining \eqref{baegam} and \eqref{baegam2} completes
the proof.
   \end{proof}
The next lemma is a direct consequence of Lemma
\ref{fund}.
   \begin{lem}\label{suf1}
Assume that \eqref{sa} holds, $\hfin{n}(x_0) < \infty$
for every $n\in \nbb$ and
$\{\hfin{n}(x_{i,1})\}_{n=0}^\infty$ is a Stieltjes
moment sequence with an S-representing measure
$P(x_{i,1}, \cdot)$ for every $i\in J_{\eta}$. Then
$\{\hfin{n + \kappa + 1}(x_0) -
\hfin{n}(x_0)\}_{n=0}^\infty$ is a Stieltjes moment
sequence with an S-representing measure $\nu$ given by
   \begin{align*}
\nu(\sigma) = \sum_{i=1}^{\eta}
\frac{\mu(x_{i,1})}{\mu(x_0)} P(x_{i,1}, \sigma),
\quad \sigma \in \borel{\rbb_+}.
   \end{align*}
   \end{lem}
   The above enables us to prove the aforementioned
criteria for subnormality.
   \begin{thm}\label{suff}
Assume that \eqref{sa} holds and
$\{\hfin{n}(x)\}_{n=0}^\infty$ is a Stieltjes moment
sequence for every $x \in \{x_0\} \cup \{x_{i,1}
\colon i\in J_{\eta}\}$. If one of the following four
conditions is satisfied\/\footnote{\;Recall that by
Lemma \ref{suf1}, $\{\hfin{n + \kappa + 1}(x_0) -
\hfin{n}(x_0)\}_{n=0}^\infty$ is a Stieltjes moment
sequence.}{\em :}
   \begin{enumerate}
   \item[(i)] $\{\hfin{n + \kappa + 1}(x_0) - \hfin{n}(x_0)
\}_{n=0}^\infty$ is S-determinate and there exists an
S-repre\-senting measure $P(x_0, \cdot)$ of
$\{\hfin{n}(x_0)\}_{n=0}^\infty$ such that $P(x_0,
[0,1))=0$,
   \item[(ii)] $\{\hfin{n + \kappa + 1}(x_0) -
\hfin{n}(x_0) \}_{n=0}^\infty$ and
$\{\hfin{j(\kappa+1)}(x_0)\}_{j=0}^\infty$ are
S-determinate,
   \item[(iii)]
$\{\hfin{j(\kappa+1)}(x_0)\}_{j=0}^\infty$ satisfies
the Carleman condition,
   \item[(iv)]
$\{\hfin{(j+1)(\kappa+1)}(x_0) -
\hfin{j(\kappa+1)}(x_0)\}_{j=0}^\infty$ satisfies the
Carleman condition,
   \end{enumerate}
then $C_{\phi}$ is subnormal.
   \end{thm}
   \begin{proof}
By Proposition \ref{hfi4}, $C_{\phi}$ is densely
defined. It follows from Lemmata \ref{det} and
\ref{suf1} that the conditions (iii) and (iv) are
equivalent. If (iv) holds, then by Proposition
\ref{wouk}(i) and Lemma \ref{det}(i), $\{\hfin{n +
\kappa + 1}(x_0) - \hfin{n}(x_0) \}_{n=0}^\infty$ is
S-determinate, and thus, because of
(iv)$\Rightarrow$(iii), the condition (ii) holds.
Applying Lemma \ref{det}(ii), we see that (ii) implies
(i). All this means that it suffices to prove that (i)
implies the subnormality of $C_{\phi}$.

To this end, assume that (i) holds. Let $P(x_{i,1},
\cdot)$ be an S-representing measure of
$\{\hfin{n}(x_{i,1})\}_{n=0}^\infty$ for $i\in
J_{\eta}$. Note that
   \begin{align*}
\hfin{n + \kappa+1}(x_0) - \hfin{n}(x_0) =
\int_0^\infty t^n (t^{\kappa+1} -1) P(x_0,\D t), \quad
n \in \zbb_+.
   \end{align*}
As $P(x_0,[0,1))=0$, the set-function $\sigma \mapsto
\int_{\sigma} (t^{\kappa+1} -1) P(x_0,\D t)$ is a
(positive) measure. Applying Lemma \ref{suf1} and
using the S-determinacy assumption, we deduce that the
condition (ii) of Theorem \ref{cc1} holds. Now we
define the measures $\{P(x_r, \cdot)\colon r \in
J_{\kappa}\}$ and $\{P(x_{i,j}, \cdot)\colon i \in
J_{\eta}, \, j\in \nbb_2\}$ by the conditions (i) and
(iii) of Theorem \ref{cc1}, respectively. Using
\eqref{hfi} and the fact that $P(x, \cdot)$ is an
S-representing measure of
$\{\hfin{n}(x)\}_{n=0}^\infty$ for every $x \in
\{x_0\} \cup \{x_{i,1} \colon i\in J_{\eta}\}$, we
verify that $\{P(x, \cdot)\}_{x\in X}$ is a family of
Borel probability measures on $\rbb_+$ that satisfies
\eqref{cc}. Applying Theorem \ref{glowne}, we conclude
that $C_{\phi}$ is subnormal.
   \end{proof}
The situation changes drastically if $\eta$ equals
$1$.
   \begin{pro}\label{eta1}
Suppose \eqref{sa} holds and $\eta=1$. Then
$\overline{\dzn{C_{\phi}}}=L^2(\mu)$,
$\{\hfin{n}(x)\}_{n=0}^{\infty}\subseteq (0,\infty)$
for every $x\in X$ and the following conditions are
equivalent{\em :}
   \begin{enumerate}
   \item[(i)] there exists a family
$\{P(x, \cdot)\}_{x\in X}$ of Borel probability
measures on $\rbb_+$ that satisfies \eqref{cc},
   \item[(ii)] $\{\hfin{n}(x_0)\}_{n=0}^\infty$
is a Stieltjes moment sequence which has an
S-representing measure $\rho$ vanishing on $[0,1)$,
   \item[(iii)] $0 \Le
\sum_{i,j=0}^n \hfin{i+j}(x_0) \lambda_i \bar
\lambda_j \Le \sum_{i,j=0}^n \hfin{i+j+1}(x_0)
\lambda_i \bar \lambda_j$ for all finite sequences
$\{\lambda_i\}_{i=0}^n$ of complex numbers.
   \end{enumerate}
Moreover, if any of the above conditions holds, then
$C_{\phi}$ is subnormal. \end{pro}
   \begin{proof}
It follows from \eqref{hfi}, \eqref{hfi-c} and
Proposition \ref{hfi4} that
$\overline{\dzn{C_{\phi}}}=L^2(\mu)$ and
$\{\hfin{n}(x)\}_{n=0}^{\infty}\subseteq (0,\infty)$
for every $x\in X$.

(i)$\Rightarrow$(ii) Apply Theorem \ref{glowne} and
Theorem \ref{cc1}(iv).

(ii)$\Rightarrow$(i) Set $P(x_0, \cdot) =
\rho(\cdot)$. By Lemma \ref{fund},
$\{\hfin{n}(x_{1,1})\}_{n=0}^\infty$ is a Stieltjes
moment sequence with an S-representing probability
measure $P(x_{1,1}, \cdot)$ given by
   \begin{align} \label{pp11}
P(x_{1,1}, \sigma) = \frac{\mu(x_0)}{\mu(x_{1,1})}
\int_{\sigma} (t^{\kappa+1}-1) P(x_0, \D t), \quad
\sigma \in \borel{\rbb_+}.
   \end{align}
Clearly, the condition (ii) of Theorem \ref{cc1}
holds. Next, we define $\{P(x_r,
\cdot)\}_{r=1}^{\kappa}$, the Borel measures on
$\rbb_+$, using the condition (i) of Theorem
\ref{cc1}. Since $P(x_0, \cdot)$ is an S-representing
measure of $\{\hfin{n}(x_0)\}_{n=0}^\infty$, we deduce
from \eqref{hfi} that the so-defined measures are
probabilistic. Finally, we define $\{P(x_{1,j},
\cdot)\}_{j=2}^\infty$, the Borel measures on
$\rbb_+$, using the condition (iii) of Theorem
\ref{cc1}. Noting that \allowdisplaybreaks
   \begin{align*} \int_0^\infty t^{j-1} P(x_{1,1},
\D t) &\overset{\eqref{pp11}}=
\frac{\mu(x_0)}{\mu(x_{1,1})} \int_0^\infty t^{j-1}
(t^{\kappa+1}-1) P(x_0, \D t)
   \\
& \hspace{1.7ex}= \frac{\mu(x_0)}{\mu(x_{1,1})}
(\hfin{j-1 + (\kappa+1)}(x_0) - \hfin{j-1}(x_0))
   \\
& \overset{\eqref{dif1}} = \hfin{j-1}(x_{1,1})
   \\
& \overset{\eqref{hfi}}=
\frac{\mu(x_{1,j})}{\mu(x_{1,1})}, \quad j \in \nbb_2,
   \end{align*}
we see that the measures $\{P(x_{1,j},
\cdot)\}_{j=2}^\infty$ are probabilistic. Now,
applying Theorem \ref{cc1}, we conclude that $\{P(x,
\cdot)\}_{x\in X}$ satisfies \eqref{cc}.

(ii)$\Leftrightarrow$(iii) Apply Lemma \ref{kontr}.

The ``moreover'' part is a direct consequence of
Theorem \ref{glowne}.
   \end{proof}
   Regarding the implication (ii)$\Rightarrow$(i) of
Proposition \ref{eta1}, we note that the assumption
that $\{\hfin{n}(x_0)\}_{n=0}^\infty$ is a Stieltjes
moment sequence is not sufficient for $C_{\phi}$ to be
subnormal, even if $C_{\phi}\in\ogr{L^2(\mu)}$.
   \begin{exa} \label{final}
First, we show that if
   \begin{align} \label{dzieci500+}
\text{$\{\gamma_n\}_{n=0}^\infty\subseteq (0,\infty)$,
$\gamma_0=1$ and $\gamma_{n+\kappa+1} - \gamma_{n} >
0$ for every $n\in \zbb_+$,}
   \end{align}
then there exists a discrete measure $\mu$ on
$X=X_{1,\kappa}$ such that $\hfin{n}(x_0)=\gamma_n$
for every $n \in \zbb_+$ with $\phi=\phi_{1,\kappa}$.
For this, take any $\mu(x_0) \in(0,\infty)$ and set
$\mu(x_r)=\mu(x_0) \gamma_r$ for every $r \in
J_{\kappa}$. Next we put
   \begin{align} \label{roznic}
\mu(x_{1,n+1}) = \mu(x_0) (\gamma_{n+\kappa+1} -
\gamma_{n}), \quad n\in\zbb_+.
   \end{align}
Since $\gamma_{n+\kappa+1} - \gamma_{n} > 0$ for every
$n\in \zbb_+$, we see that $\mu(x_{1,j}) \in
(0,\infty)$ for every $j\in \nbb$. Clearly, by
\eqref{hfi}, $\hfin{r}(x_0) = \gamma_r$ for every $r
\in \{0, \ldots, \kappa\}$. Using induction, Lemma
\ref{fund}, \eqref{hfi} and \eqref{roznic}, we verify
that $\hfin{n}(x_0)=\gamma_n$ for every $n \in
\zbb_+$.

It is worth mentioning that if \eqref{sa} holds and
$\eta=1$, then the sequence
$\{\gamma_n\}_{n=0}^\infty$ defined by
$\gamma_n=\hfin{n}(x_0)$ for every $n \in \zbb_+$
satisfies \eqref{dzieci500+}. Indeed, this is a direct
consequence of Proposition \ref{eta1}, Lemma
\ref{fund} and \eqref{hfi}.

Now, we can apply the above procedure as follows. Take
$a\in (0,1)$ and set $\gamma_n = \frac{1}{2} (a^n +
2^n)$ for every $n \in \zbb_+$. Clearly,
$\{\gamma_n\}_{n=0}^\infty$ is an S-determinate
Stieltjes moment sequence with the S-representing
measure $\frac{1}{2} (\delta_{a} + \delta_2)$. Since
   \begin{align*}
a^{n+\kappa+1} + 2^{n+\kappa+1} > 2^{n+\kappa+1} \Ge 1
+ 2^n \Ge a^n + 2^n, \quad n\in \zbb_+,
   \end{align*}
we see that $\gamma_{n+\kappa+1} - \gamma_{n} > 0$ for
every $n\in \zbb_+$. Applying our procedure, we get a
discrete measure $\mu$ on $X$ such that
$\hfin{n}(x_0)=\gamma_n$ for every $n \in \zbb_+$.
Note that $C_{\phi}\in \ogr{L^2(\mu)}$. Indeed, by
\eqref{hfi} and \eqref{roznic}, we have
   \begin{align*}
\lim_{n \to \infty} \hfi(x_{1,n}) = \lim_{n \to
\infty} \frac{\gamma_{n+\kappa+1}-\gamma_{n}}
{\gamma_{n+\kappa}-\gamma_{n-1}} =
\frac{2^{\kappa+2}-2}{2^{\kappa+1}-1} < \infty.
   \end{align*}
This, combined with \eqref{hfi} and \eqref{ogrn},
yields $C_{\phi}\in \ogr{L^2(\mu)}$. It is worth
mentioning that $C_{\phi}$ is not subnormal (hence, by
Proposition \ref{wn0gr} below,
$\{\hfin{n}(x_{1,1})\}_{n=0}^\infty$ is not a
Stieltjes moment sequence). Indeed, otherwise, by
\cite[Theorem 13]{b-j-j-sS}, $C_{\phi}$ admits a
family $\{P(x,\cdot)\}_{x\in X}$ of Borel probability
measures on $\rbb_+$ that satisfies \eqref{cc}. This,
together with Proposition \ref{eta1}, contradicts the
S-determinacy of $\{\gamma_n\}_{n=0}^\infty$.
   \end{exa}
The question of characterizing subnormality of bounded
composition operators of the form
$C_{\phi_{\eta,\kappa}}$ has a simple solution.
   \begin{pro} \label{wn0gr}
Suppose \eqref{sa} holds and $C_\phi \in
\ogr{L^2(\mu)}$. Then $C_{\phi}$ is subnormal if and
only if $\{\hfin{n}(x)\}_{n=0}^\infty$ is a Stieltjes
moment sequence for every $x \in \{x_0\} \cup
\{x_{i,1} \colon i\in J_{\eta}\}$.
   \end{pro}
   \begin{proof}
Since, by \eqref{ogrn}, $\hfin{n}(x_0) \Le
\|C_{\phi}\|^{2n}$ for all $n \in \zbb_+$, the
sufficiency follows from Theorem \ref{suff}(iii). The
sufficiency can also be deduced from Lambert's theorem
(cf.\ \cite{lam1}) by using \eqref{hfi} and
\eqref{hfir-2}. The necessity is a direct consequence
of the corresponding part of Lambert's theorem.
   \end{proof}
Now we discuss the question of ``optimality'' of the
assumptions of Proposition \ref{wn0gr}. As shown in
Example \ref{final}, for $\eta=1$ and for every
$\kappa \in \zbb_+$, there exists a non-subnormal
$C_{\phi} \in \ogr{L^2(\mu)}$ such that
$\{\hfin{n}(x_0)\}_{n=0}^\infty$ is a Stieltjes moment
sequence. In Example \ref{prz2} below, we will show
that for $\eta=1$ and $\kappa=0$ the assumption that
$\{\hfin{n}(x_{1,1})\}_{n=0}^\infty$ is a Stieltjes
moment sequence is not sufficient for $C_{\phi}$ to be
subnormal, even if $C_{\phi} \in \ogr{L^2(\mu)}$.
   \begin{exa} \label{prz2}
Assume that $\eta=1$ and $\kappa=0$. Take any positive
real numbers $\mu(x_0)$ and $\mu(x_{1,1})$. Let
$\{\gamma_n\}_{n=0}^\infty$ be a Stieltjes moment
sequence with an S-representing measure $\nu$ such
that $\gamma_0=1$, $\nu((1,\infty))=0$ and
$\nu(\{1\})> 0$. Then
   \begin{align} \label{Levy2}
0 < \gamma_n \Le 1, \quad n \in \zbb_+,
   \end{align}
which implies that $\{\gamma_n\}_{n=0}^\infty$ is
S-determinate (see e.g., Proposition \ref{wouk}(i)).
Set $\mu(x_{1,j})=\mu(x_{1,1}) \gamma_{j-1}$ for every
$j\in \nbb_2$. Clearly, $\hfin{n}(x_{1,1}) = \gamma_n$
for all $n\in \zbb_+$, which means that
$\{\hfin{n}(x_{1,1})\}_{n=0}^\infty$ is a Stieltjes
moment sequence. Since $\{\gamma_{n}\}_{n=0}^\infty$
is monotonically decreasing, we see that
$\hfi(x_{1,j}) = \frac{\gamma_{j}}{\gamma_{j-1}} \Le
1$ for every $j \in \nbb$. This, together with
\eqref{ogrn}, implies that $C_{\phi} \in
\ogr{L^2(\mu)}$. Note that
   \begin{align} \notag
\hfin{n}(x_0) & = \hfin{0}(x_0) + \sum_{l=0}^{n-1}
(\hfin{l+1}(x_0) - \hfin{l}(x_0))
   \\ \label{Levy}
& \hspace{-1.8ex}\overset{\eqref{dif1}} = 1 +
\frac{\mu(x_{1,1})}{\mu(x_0)} \sum_{l=0}^{n-1}
\gamma_l, \quad n \in \nbb.
   \end{align}
This and \eqref{Levy2} imply that
   \begin{align} \label{limsup}
\lim_{n\to\infty} \hfin{n}(x_0)^{1/n} = 1.
   \end{align}
Now we show that $\{\hfin{n}(x_0)\}_{n=0}^\infty$ is
not a Stieltjes moment sequence. Indeed, otherwise by
\eqref{limsup} and \cite[Exercise 4(e), p.\ 71]{Rud},
we see that $\rho((1,\infty)) = 0$, where $\rho$ is an
S-representing measure of
$\{\hfin{n}(x_0)\}_{n=0}^\infty$. Hence
   \begin{align*}
1 = \rho(\rbb_+) \Ge \hfin{n}(x_0)
\overset{\eqref{Levy}} \Ge 1 + \frac{\mu(x_{1,1})
\nu(\{1\})}{\mu(x_0)} \, n, \quad n \in \nbb,
   \end{align*}
which contradicts $\nu(\{1\})> 0$. By Proposition
\ref{wn0gr}, $C_{\phi}$ is not subnormal.
   \end{exa}
   \subsection{Extending to families satisfying
\eqref{cc}} \label{s4.4}
    We begin this section by providing necessary and
sufficient conditions for the extendibility of a given
family of Borel probability measures on $\rbb_+$
indexed by $\{x_{i,1}\}_{i\in J_{\eta}}$ to a family
of Borel probability measures on $\rbb_+$ satisfying
\eqref{cc}.
   \begin{thm} \label{main}
Suppose \eqref{sa} holds and $C_{\phi}$ is densely
defined. Then the following assertions hold.
   \begin{enumerate}
   \item[(i)] If $\{P(x,\cdot)\}_{x\in X}$ is a
family of Borel probability measures on $\rbb_+$ that
satisfies \eqref{cc}, then
   \begin{enumerate}
   \item[(i-a)]
$\{\hfin{n}(x_{i,1})\}_{n=0}^\infty$ is a Stieltjes
moment sequence for every $i\in J_{\eta}$,
   \item[(i-b)]  for every $i\in J_{\eta}$, the measure $P(x_{i,1}, \cdot)$
is an S-representing measure of
$\{\hfin{n}(x_{i,1})\}_{n=0}^\infty$ vanishing on
$[0,1]$,
   \item[(i-c)] $\sum_{i=1}^{\eta} \frac{\mu(x_{i,1})}
{\mu(x_r)} \int_0^\infty \frac{t^r-1}{t^{\kappa+1}-1}
P(x_{i,1},\D t) + \frac{\mu(x_0)} {\mu(x_r)}= 1$ for
every $r \in J_{\kappa}$,
   \item[(i-d)] $\sum_{i=1}^{\eta} \frac{\mu(x_{i,1})}
{\mu(x_0)} \int_0^\infty \frac{1}{t^{\kappa+1}-1}
P(x_{i,1},\D t) \Le 1$.
   \end{enumerate}
   \item[(ii)] If {\em (i-a)} holds and
$\{P(x_{i,1},\cdot)\}_{i\in J_{\eta}}$ is a family of
Borel probability measures on $\rbb_+$ satisfying
\mbox{{\em (i-b)}}, \mbox{{\em (i-c)}} and \mbox{{\em
(i-d)}}, then there exists a family $\{P(x,
\cdot)\}_{x \in X \setminus \{x_{i,1}\colon i \in
J_{\eta}\}}$ of Borel probability measures on $\rbb_+$
such that $\{P(x,\cdot)\}_{x\in X}$ satisfies
\eqref{cc}.
   \end{enumerate}
   \end{thm}
   \begin{proof}
(i) The conditions \mbox{(i-a)} and \mbox{(i-b)}
follow from \eqref{hfi}, Theorem \ref{glowne} and
Theorem \ref{cc1}(iv). The condition \mbox{(i-d)} is a
direct consequence of Theorem \ref{cc1}(vi). Using the
conditions (i), (vii) and (viii) of Theorem \ref{cc1}
and integrating the function $t \mapsto t^r$ with
respect to the measure $P(x_0, \cdot)$, we obtain
   \allowdisplaybreaks
   \begin{align*}
\frac{\mu(x_r)}{\mu(x_0)} & = \int_0^\infty t^r P(x_0,
\D t)
   \\
   & = \sum_{i=1}^{\eta} \frac{\mu(x_{i,1})}{\mu(x_0)}
\int_0^\infty \frac{t^r}{t^{\kappa+1}-1} P(x_{i,1}, \D
t)
   \\
& \hspace{6.5ex} + 1 - \sum_{i=1}^{\eta}
\frac{\mu(x_{i,1})}{\mu(x_0)} \int_0^\infty
\frac{1}{t^{\kappa+1}-1} P(x_{i,1}, \D t)
   \\
& = 1 + \sum_{i=1}^{\eta} \frac{\mu(x_{i,1})}
{\mu(x_0)} \int_0^\infty \frac{t^r-1}{t^{\kappa+1}-1}
P(x_{i,1},\D t), \quad r \in J_{\kappa},
   \end{align*}
which implies \mbox{(i-c)}.

(ii) For $i\in J_{\eta}$ and $j\in \nbb_2$, we define
the measure $P(x_{i,j}, \cdot)$ by the condition (iii)
of Theorem \ref{cc1}. Note that $P(x_{i,j}, \rbb_+)=1$
because
   \begin{align*}
P(x_{i,j}, \rbb_+) = \frac{\mu(x_{i,1})}{\mu(x_{i,j})}
\int_0^\infty t^{j-1} P(x_{i,1}, \D t)
\overset{\rm{(i-b)}}=
\frac{\mu(x_{i,1})}{\mu(x_{i,j})} \hfin{j-1}(x_{i,1})
\overset{\eqref{hfi}}=1.
   \end{align*}
Next, we define the measures $P(x_0,\cdot)$, \ldots,
$P(x_{\kappa},\cdot)$ by
   \begin{align} \label{xxx}
P(x_r,\sigma) = \sum_{i=1}^{\eta}
\frac{\mu(x_{i,1})}{\mu(x_r)} \int_{\sigma}
\frac{t^r}{t^{\kappa+1}-1} P(x_{i,1}, \D t) +
\vartheta \frac{\mu(x_0)}{\mu(x_r)} \delta_1(\sigma)
   \end{align}
for $\sigma \in \borel{\rbb_+}$ and $r \in \{0,
\ldots,\kappa\}$ with $\vartheta=1- \sum_{i=1}^{\eta}
\frac{\mu(x_{i,1})} {\mu(x_0)} \int_0^\infty
\frac{1}{t^{\kappa+1}-1} P(x_{i,1},\D t)$. By
\mbox{(i-b)}, \mbox{(i-c)} and \mbox{(i-d)}, the
quantity $\vartheta$ is well-defined, $\vartheta \in
[0,1]$ and $P(x_r,\cdot)$ is a well-defined finite
measure. That it is probabilistic follows from the
equalities
   \begin{align*}
P(x_r,\rbb_+) & \overset{\eqref{xxx}} =
\sum_{i=1}^{\eta} \frac{\mu(x_{i,1})}{\mu(x_r)}
\int_0^\infty \frac{t^r-1}{t^{\kappa+1}-1} P(x_{i,1},
\D t) + \frac{\mu(x_0)}{\mu(x_r)}
   \\
& \hspace{.5ex}\overset{\rm{(i-c)}} = 1, \quad r \in
\{0,\ldots,\kappa\}.
   \end{align*}
Now, integrating the function $t \mapsto t^r$ with
respect to $P(x_0,\cdot)$, we see that the condition
(i) of Theorem \ref{cc1} is satisfied. Finally,
integrating the function $t \mapsto t^{\kappa+1} -1$
(which is positive on $(1,\infty)$) with respect to
$P(x_0,\cdot)$, we deduce that the condition (ii) of
Theorem \ref{cc1} is satisfied. This combined with
Theorem \ref{cc1} completes the proof.
   \end{proof}
Below we introduce the condition
\mbox{(i-d$^{\prime}$)}, a weaker version of the
condition \mbox{(i-d)} of Theorem \ref{main} that will
lead us to constructing exotic examples (cf.\ Section
\ref{sec5}). For more information concerning the
conditions (e1) and (e2) below, the reader is referred
to Lemma \ref{snieg} and Remark \ref{snieg2}.
   \begin{thm} \label{6.2}
Suppose \eqref{sa} holds, the condition {\em (i-a)} of
Theorem {\em \ref{main}} is satisfied and
$\{P(x_{i,1}, \cdot)\}_{i\in J_{\eta}}$ is a family of
Borel probability measures on $\rbb_+$ satisfying the
conditions \mbox{{\em (i-b)}} and \mbox{{\em (i-c)}}
of Theorem {\em \ref{main}} and the condition
below{\em :}
   \begin{enumerate}
   \item[(i-d$^{\prime}$)] $\sum_{i=1}^{\eta}
\mu(x_{i,1}) \int_0^\infty \frac{1}{t^{\kappa+1}-1}
P(x_{i,1},\D t) < \infty$.
   \end{enumerate}
Consider the following three conditions{\em :}
   \begin{enumerate}
   \item[(e1)]
$\{\hfin{n}(x_{\kappa}) + c\}_{n=0}^\infty$ is an
S-determinate Stieltjes moment sequence for every $c
\in (0,\infty)$,
   \item[(e2)]
$\{\hfin{n+1}(x_{\kappa}) + c\}_{n=0}^\infty$ is an
S-determinate Stieltjes moment sequence for every $c
\in (0,\infty)$,
   \item[(e3)] $\{\hfin{n}(x_{\kappa})\}_{n=0}^\infty$
is a Stieltjes moment sequence that satisfies the
Carleman condition.
   \end{enumerate}
Then {\em (e3)}$\Rightarrow${\em
(e2)}$\Rightarrow${\em (e1)} and, if {\em (e1)} holds,
then the composition operator $C_{\phi}$ is densely
defined and there exists a family $\{P(x, \cdot)\}_{x
\in X \setminus \{x_{i,1}\colon i \in J_{\eta}\}}$ of
Borel probability measures on $\rbb_+$ such that
$\{P(x,\cdot)\}_{x\in X}$ satisfies \eqref{cc}.
Moreover, if {\em (e3)} holds, then
$\{\hfin{n}(x_{r})\}_{n=0}^\infty$ satisfies the
Carleman condition for every $r \in \{0,\ldots,
\kappa\}$.
   \end{thm}
   \begin{proof}
It follows from \mbox{(i-b)} and
\mbox{(i-d$^{\prime}$)} that the quantity $\xi$
defined below
   \begin{align}  \label{claim}
\xi := \frac{\mu(x_0)} {\mu(x_{\kappa})} -
\sum_{i=1}^{\eta} \frac{\mu(x_{i,1})}
{\mu(x_{\kappa})} \int_0^\infty
\frac{1}{t^{\kappa+1}-1} P(x_{i,1},\D t)
   \end{align}
is a real number. By \mbox{(i-b)}, \mbox{(i-c)} and
\mbox{(i-d$^{\prime}$)}, we have
   \begin{align} \label{zlam}
\xi = \frac{\mu(x_r)} {\mu(x_{\kappa})} -
\sum_{i=1}^{\eta} \frac{\mu(x_{i,1})}
{\mu(x_{\kappa})} \int_0^\infty
\frac{t^r}{t^{\kappa+1}-1} P(x_{i,1},\D t), \quad r
\in \{0, \ldots, \kappa\}.
   \end{align}
In particular, the above series are convergent in
$\rbb_+$.

First, we show that\footnote{\;It follows from the
proof of \eqref{xi} that $\int_0^\infty
\frac{t^n}{t^{\kappa+1}-1} P(x_{i,1},\D t) < \infty$
for all $n\in \zbb_+$ and $i\in J_{\eta}$, and in the
case of $\eta=\infty$, $\sum_{i=1}^{\eta}
\frac{\mu(x_{i,1})} {\mu(x_{\kappa})} \int_0^\infty
\frac{t^n}{t^{\kappa+1}-1} P(x_{i,1},\D t)< \infty$
for any integer $n$ such that $0 \Le n \Le \kappa$.
However, this series may be divergent to infinity for
some integers $n\Ge \kappa+1$.}
   \begin{align} \label{xi}
\hfin{n+1}(x_{\kappa}) = \xi + \sum_{i=1}^{\eta}
\frac{\mu(x_{i,1})} {\mu(x_{\kappa})} \int_0^\infty
t^n \frac{t^{\kappa+1}}{t^{\kappa+1}-1} P(x_{i,1},\D
t), \quad n \in \zbb_+.
   \end{align}
To this end, we fix $n\in \zbb_+$. Then $n=j(\kappa +
1) + r$ for some $j\in \zbb_+$ and $r \in
\{0,\ldots,\kappa\}$. We begin by considering the case
of $r < \kappa$. Since then $r+1$ is the remainder of
the division of $n+1$ by $\kappa + 1$, we infer from
\mbox{(i-b)} that \allowdisplaybreaks
   \begin{align*}
\hfin{n+1}(x_\kappa) & \overset{\eqref{hfi-c}}=
\frac{\mu(x_{r})}{\mu(x_{\kappa})} + \sum_{i=1}^{\eta}
\sum_{l=0}^j
\frac{\mu(x_{i,l(\kappa+1)+r+1})}{\mu(x_{\kappa})}
   \\
& \overset{\eqref{hfi}}=
\frac{\mu(x_{r})}{\mu(x_{\kappa})} + \sum_{i=1}^{\eta}
 \sum_{l=0}^j \frac{\mu(x_{i,1})}{\mu(x_{\kappa})}
\hfin{l(\kappa+1)+r}(x_{i,1})
   \\
& \hspace{1.8ex}= \frac{\mu(x_{r})}{\mu(x_{\kappa})} +
\sum_{i=1}^{\eta} \frac{\mu(x_{i,1})}{\mu(x_{\kappa})}
\int\limits_{(1,\infty)} \sum_{l=0}^j (t^{\kappa+1})^l
t^r P(x_{i,1}, \D t)
   \\
& \hspace{1.8ex}= \frac{\mu(x_{r})}{\mu(x_{\kappa})} +
\sum_{i=1}^{\eta} \frac{\mu(x_{i,1})}{\mu(x_{\kappa})}
\int\limits_{(1,\infty)}
\frac{t^{(j+1)(\kappa+1)}-1}{t^{\kappa+1}-1} \, t^r
P(x_{i,1}, \D t)
   \\
& \overset{\eqref{zlam}}= \xi + \sum_{i=1}^{\eta}
\frac{\mu(x_{i,1})} {\mu(x_{\kappa})} \int_0^\infty
t^n \frac{t^{\kappa+1}}{t^{\kappa+1}-1} P(x_{i,1},\D
t).
   \end{align*}
If $r=\kappa$, then mimicking the above proof we get
\allowdisplaybreaks
   \begin{align*}
\hfin{n+1}(x_\kappa) & \overset{\eqref{hfi-c}}= 1 +
\sum_{i=1}^{\eta} \sum_{l=1}^{j+1}
\frac{\mu(x_{i,l(\kappa+1)})}{\mu(x_{\kappa})}
   \\
& \hspace{1.7ex}= 1 + \sum_{i=1}^{\eta}
 \sum_{l=1}^{j+1} \frac{\mu(x_{i,1})}{\mu(x_{\kappa})}
\hfin{l(\kappa+1)-1}(x_{i,1})
   \\
& \overset{\eqref{zlam}}= \xi + \sum_{i=1}^{\eta}
\frac{\mu(x_{i,1})} {\mu(x_{\kappa})} \int_0^\infty
t^n \frac{t^{\kappa+1}}{t^{\kappa+1}-1} P(x_{i,1},\D
t),
   \end{align*}
which proves \eqref{xi}.

It follows from \eqref{xi} that
   \begin{align} \label{xi-1}
\hfin{n}(x_{\kappa}) = \xi + \sum_{i=1}^{\eta}
\frac{\mu(x_{i,1})} {\mu(x_{\kappa})} \int_0^\infty
t^n \frac{t^{\kappa}}{t^{\kappa+1}-1} P(x_{i,1},\D t),
\quad n \in \zbb_+.
   \end{align}
(The case of $n=0$ can be deduced from \eqref{zlam}
with $r=\kappa$.)

(e3)$\Rightarrow$(e2) Apply the assertions (i)-(iii)
of Proposition \ref{wouk}.

(e2)$\Rightarrow$(e1) Since the class of Stieltjes
moment sequences is closed under the operation of
taking pointwise limits (which follows from the
Stieltjes theorem, cf.\ \cite[Theorem 6.2.5]{b-c-r})
and $\hfin{n+1}(x_{\kappa}) = \lim_{c \to 0+}
(\hfin{n+1}(x_{\kappa})+c)$ for all $n\in \zbb_+$, we
see that $\{\hfin{n+1}(x_{\kappa})\}_{n=0}^\infty$ is
a Stieltjes moment sequence. Let $\rho$ be an
S-representing measure of
$\{\hfin{n+1}(x_{\kappa})\}_{n=0}^\infty$. Then, by
\eqref{xi}, we have
   \begin{align} \label{xi+}
\int_0^\infty t^n \D\rho(t) = \hfin{n+1}(x_{\kappa}) =
\xi + \int_0^\infty t^n \D \nu(t), \quad n \in \zbb_+,
   \end{align}
where $\nu$ is the Borel measure on $\rbb_+$ given by
   \begin{align} \label{xi++}
\nu(\sigma) = \sum_{i=1}^{\eta} \frac{\mu(x_{i,1})}
{\mu(x_{\kappa})} \int_{\sigma}
\frac{t^{\kappa+1}}{t^{\kappa+1}-1} P(x_{i,1},\D t),
\quad \sigma \in \borel{\rbb_+}.
   \end{align}
Now we prove that the condition \mbox{(i-d)} of
Theorem \ref{main} is satisfied. Indeed, otherwise
$\xi \in (-\infty,0)$. By (e2), the Stieltjes moment
sequence $\{\hfin{n+1}(x_{\kappa}) +
|\xi|\}_{n=0}^\infty$ is S-determinate. This together
with \eqref{xi+} implies that
   \begin{align} \label{leq}
\rho(\sigma) + |\xi| \delta_1 (\sigma) = \nu(\sigma),
\quad \sigma \in \borel{\rbb_+}.
   \end{align}
Using \eqref{xi++} and (i-b), and substituting $\sigma
= \{1\}$ into \eqref{leq}, we deduce that $\xi=0$, a
contradiction. This shows that the condition
\mbox{(i-d)} of Theorem \ref{main} is satisfied.

Since $\hfin{n}(x_{\kappa}) < \infty$ for all $n \in
\zbb_+$, we infer from Proposition \ref{hfi4} that
$\dzn{C_{\phi}}$ is dense in $L^2(\mu)$. Hence, by
Theorem \ref{main}(ii), there exists $\{P(x,
\cdot)\}_{x \in X \setminus \{x_{i,1}\colon i \in
J_{\eta}\}}$, a family of Borel probability measures
on $\rbb_+$, such that $\{P(x,\cdot)\}_{x\in X}$
satisfies \eqref{cc}. By Theorem \ref{glowne},
$\hfin{n}(x_{\kappa}) = \int_0^\infty t^n
P(x_{\kappa}, \D t)$ for all $n \in \zbb_+$. This
means that for every $c \in (0, \infty)$,
$\{\hfin{n}(x_{\kappa}) + c\}_{n=0}^\infty$ is a
Stieltjes moment sequence and, by (e2),
$\{\hfin{n+1}(x_{\kappa}) + c\}_{n=0}^\infty$ is an
S-determinate Stieltjes moment sequence, which implies
that $\{\hfin{n}(x_{\kappa}) + c\}_{n=0}^\infty$ is
S-determinate for every $c \in (0, \infty)$ (see
\cite[Proposition 5.12]{sim}; see also \cite[Lemma
2.4.1]{b-j-j-sA}).

Assume now that (e1) holds. Passing to the limit, as
in the proof of (e2)$\Rightarrow$(e1), we see that
$\{\hfin{n}(x_{\kappa})\}_{n=0}^\infty$ is a Stieltjes
moment sequence. Let $\rho^{\prime}$ be an
S-represent\-ing measure of
$\{\hfin{n}(x_{\kappa})\}_{n=0}^\infty$. Then, by
\eqref{xi-1}, we have
   \begin{align*}
\int_0^\infty t^n \D\rho^{\prime}(t) =
\hfin{n}(x_{\kappa}) = \xi + \int_0^\infty t^n \D
\nu^{\prime}(t), \quad n \in \zbb_+,
   \end{align*}
where $\nu^{\prime}$ is the Borel measure on $\rbb_+$
given by
   \begin{align*}
\nu^{\prime}(\sigma) = \sum_{i=1}^{\eta}
\frac{\mu(x_{i,1})} {\mu(x_{\kappa})} \int_{\sigma}
\frac{t^{\kappa}}{t^{\kappa+1}-1} P(x_{i,1},\D t),
\quad \sigma \in \borel{\rbb_+}.
   \end{align*}
Arguing as in the paragraph containing \eqref{xi++}
(using (e1) in place of (e2)), we deduce that the
condition \mbox{(i-d)} of Theorem \ref{main} is
satisfied. According to Proposition \ref{hfi4},
$C_{\phi}$ is densely defined. Applying Theorem
\ref{main}(ii), we get the required family of Borel
probability measures on $\rbb_+$ that satisfies
\eqref{cc}.

Now we prove the ``moreover'' part. Assume (e3) holds.
First note that, by what has been proved above, the
assumptions of Theorem \ref{glowne} are satisfied.
Hence, by Proposition \ref{hfi4},
$\{\hfin{n}(x)\}_{n=0}^\infty$ is a Stieltjes moment
sequence for every $x\in X$. Since
$\{\hfin{n}(x_{\kappa})\}_{n=0}^\infty$ satisfies the
Carleman condition, we deduce from \eqref{hfir-4} and
Proposition \ref{wouk}(ii) that
$\{\hfin{n}(x_{\kappa-1})\}_{n=0}^\infty$ satisfies
the Carleman condition as well. An induction argument
completes the proof.
   \end{proof}
The following is related to the ``moreover'' part of
the conclusion of Theorem \ref{6.2}.
   \begin{pro}
Suppose \eqref{sa} holds,
$\{\hfin{n}(x_{r})\}_{n=0}^\infty$ is a Stieltjes
moment sequence for every $r \in \{0, \ldots,
\kappa\}$, and $\{\hfin{n}(x_{\kappa})\}_{n=0}^\infty$
is S-determinate. Then for every $r \in \{0, \ldots,
\kappa\}$, $\{\hfin{n}(x_r)\}_{n=0}^\infty$ is
S-determinate.
   \end{pro}
   \begin{proof}
It follows from \eqref{hfir-4} applied to $r=\kappa$
that the Stieltjes moment sequence
$\{\hfin{n+1}(x_{\kappa - 1})\}_{n=0}^\infty$ is
S-determinate. This combined with \cite[Proposition
5.12]{sim} (see also \cite[Lemma 2.1.1]{j-j-s0})
implies that $\{\hfin{n}(x_{\kappa -
1})\}_{n=0}^\infty$ is S-determinate as well. Thus, an
induction argument completes the proof.
   \end{proof}
   \section{{\bf Examples of exotic non-hyponormal operators}}
   \label{sec5}
   \subsection{Outline}
   In the last part of the paper we construct
non-hyponormal injective composition operators
generating Stieltjes moment sequences. The
construction relies on the key observation that there
is a gap between the conditions \mbox{(i-d)} and
\mbox{(i-d$^\prime$)} of Theorems \ref{main} and
\ref{6.2}. The indices of H-determinacy of the
sequences $\{\hfin{n}(x)\}_{n=0}^\infty$, $x \in X$,
are discussed as well.

We begin by introducing for any $\eta \in \nbb \cup
\{\infty\}$ and $\kappa \in \zbb_+$ a class of
composition operators over the directed graph
$\gcal_{\eta,\kappa} := (X_{\eta,\kappa},
E^{\phi_{\eta,\kappa}})$ which admit families
$\{P(x_{i,1},\cdot)\}_{i\in J_{\eta}}$ of Borel
probability measures on $\rbb_+$ satisfying the
conditions \mbox{(i-a)}, \mbox{(i-b)}, \mbox{(i-c)}
and \mbox{(i-d$^\prime$)} of Theorems \ref{main} and
\ref{6.2}, but not \mbox{(i-d)} of Theorem \ref{main}
(see Procedure \ref{proceed2} and Lemma
\ref{pierwsze}). The construction of these operators
essentially depends on a choice of specific N-extremal
measures $\nu$ and $\tau$, and a partition
$\{\varDelta_i\}_{i=1}^\eta$ of $\supp{\tau}$. The
fundamental properties of the so-constructed
operators, including the characterization of their
hyponormality for $\kappa=0$, are proven in Lemma
\ref{pierwsze}. Section \ref{s5.3} shows that the gap
between the conditions \mbox{(i-d)} and
\mbox{(i-d$^\prime$)} does exist. Theorem
\ref{czwarte} is the culminating result of the present
paper. It shows that there exists a non-hyponormal
composition operator generating Stieltjes moment
sequences over the locally finite directed graph
$\gcal_{2,0}$. Its proof heavily depends on the
existence of N-extremal probability measures $\zeta$
and $\rho$ satisfying a restrictive condition which is
not easy to deal with. Fortunately, there are
N-extremal probability measures coming from shifted
Al-Salam-Carlitz \mbox{$q$-polynomials} or from a
quartic birth and death process that satisfy this
condition. It is worth pointing out that, without the
use of these special N-extremal measures, Step 1 of
the proof of Theorem \ref{czwarte} combined with
\eqref{gs1+} implies that for every sufficiently large
integer $\eta$, there exists a non-hyponormal
composition operator generating Stieltjes moment
sequences over the locally finite directed graph
$\gcal_{\eta,0}$.
   \subsection{General scheme}
In this section we introduce the aforementioned class
of composition operators. We do it according to the
following procedure.
   \begin{proc} \label{proceed2}
Fix $\eta \in \nbb \cup \{\infty\}$ and $\kappa \in
\zbb_+$. Suppose that \allowdisplaybreaks
   \begin{gather} \label{gs1}
\text{$\nu$ and $\tau$ are N-extremal measures of the
same Stieltjes moment sequence,}
   \\[.5ex] \label{gs2}
1=\inf\supp{\nu} < \inf\supp{\tau},
   \\ \label{gs3}
\nu(\rbb_+) = 1+ \nu(\{1\}),
   \\ \label{gs4}
1 + \int_0^\infty \frac{1}{t^{\kappa}} \D \tau(t) >
\tau(\rbb_+),
   \end{gather}
and
   \begin{align} \label{gs6}
\text{$\{\varDelta_i\}_{i=1}^\eta$ is a partition of
$\supp{\tau}$.}
   \end{align}
(A partition of a nonempty set is always assumed to
consist of nonempty sets.) Since, by Lemma \ref{proN},
$\card{\supp{\tau}} = \aleph_0$, such partition always
exists.

It follows from Lemma \ref{proN}, \eqref{gs1} and
\eqref{gs2} that $\int_{\varDelta_i}
\frac{t^{\kappa+1}-1}{t^{\kappa}} \D \tau(t) \in
(0,\infty)$ for every $i\in J_{\eta}$ (see Section
\ref{n&t} for the definition of $J_{\eta} $). Hence,
   \begin{align} \label{gs10}
c_i:= \Big(\int_{\varDelta_i}
\frac{t^{\kappa+1}-1}{t^{\kappa}} \D \tau(t)\Big)^{-1}
\in (0,\infty), \quad i\in J_{\eta}.
   \end{align}
Set
   \begin{align} \label{gs7}
P(x_{i,1}, \sigma) = c_i \int_{\varDelta_i \cap
\sigma} \frac{t^{\kappa+1}-1}{t^{\kappa}} \D \tau(t),
\quad \sigma \in \borel{\rbb_+}, \, i \in J_{\eta}.
   \end{align}
Clearly, $P(x_{i,1}, \cdot)$ is a Borel probability
measure on $\rbb_+$ such that $\int_0^\infty t^n
P(x_{i,1}, \D t) \in (0,\infty)$ for all $n \in
\zbb_+$ and $i\in J_{\eta}$ (use \eqref{gs1} and
\eqref{gs2}). Take any $\mu(x_{\kappa}) \in
(0,\infty)$ and define the family
$\{\mu(x_{i,1})\}_{i\in J_{\eta}}$ of positive real
numbers by
   \begin{align} \label{gs8}
\mu(x_{i,1}) = \frac{1}{c_i} \, \mu(x_{\kappa}), \quad
i \in J_{\eta}.
   \end{align}
Next, we define the family $\{\mu(x_{i,j})\}_{(i,j)\in
J_\eta \times \nbb_2}$ of positive real numbers and
the family $\{P(x_{i,j}, \cdot)\}_{(i,j)\in J_\eta
\times \nbb_2}$ of Borel probability measures on
$\rbb_+$ by
   \allowdisplaybreaks
   \begin{align} \label{muxij}
\mu(x_{i,j}) & = \mu(x_{i,1}) \int_0^\infty t^{j-1}
P(x_{i,1}, \D t), \quad i \in J_{\eta}, \, j\in
\nbb_2,
   \\ \label{muxij1}
P(x_{i,j},\sigma) & =
\frac{\mu(x_{i,1})}{\mu(x_{i,j})} \int_{\sigma}
t^{j-1} P(x_{i,1}, \D t), \quad \sigma \in
\borel{\rbb_+}, \, i \in J_{\eta}, \, j\in \nbb_2.
   \end{align}
Let $P(x_{\kappa}, \cdot)$ be the Borel measure on
$\rbb_+$ given by
   \begin{align}  \label{gs11}
P(x_{\kappa}, \sigma)=\nu(\sigma) - \nu(\{1\})
\delta_1(\sigma), \quad \sigma \in \borel{\rbb_+}.
   \end{align}
By \eqref{gs1} and \eqref{gs3}, $P(x_{\kappa}, \cdot)$
is a probability measure. It is also clear that $0 <
\int_0^\infty t^n \D P(x_{\kappa}, \D t) < \infty$ for
every $n\in \zbb_+$. In view of \eqref{gs1} and
\eqref{gs2}, we have
   \begin{align} \label{xxx1}
0 < \int_0^\infty \frac{1}{t^n} \D \tau(t)<\infty,
\quad n \in \zbb_+,
   \end{align}
and
   \begin{align}  \notag
\int_0^\infty \frac{1}{t^{\kappa-r}} \D \tau(t) -
\nu(\{1\}) & \Ge \int_0^\infty \frac{1}{t^{\kappa}} \D
\tau(t) - \nu(\{1\})
   \\  \notag
& \hspace{-1.7ex}\overset{\eqref{gs4}} > \tau(\rbb_+)
-1 - \nu(\{1\})
   \\  \notag
& \hspace{-1.7ex} \overset{\eqref{gs1}}= \nu(\rbb_+)
-1 - \nu(\{1\})
   \\ \label{xxx2}
& \hspace{-1.7ex} \overset{\eqref{gs3}} = 0, \quad r
\in \{0,\ldots, \kappa\}.
   \end{align}
If $\kappa \Ge 1$, then we set
   \begin{align} \label{dd1}
\mu(x_r) = \mu(x_{\kappa}) \bigg(\int_0^\infty
\frac{1}{t^{\kappa-r}} \D \tau(t) - \nu(\{1\})\bigg),
\quad r \in \{0, \ldots, \kappa-1\}.
   \end{align}
It follows from \eqref{xxx1} and \eqref{xxx2} that
$\mu(x_r) \in (0,\infty)$ for every $r\in \{0, \ldots,
\kappa-1\}$.

Finally, let $\mu$ be the (unique) discrete measure on
$X=X_{\eta,\kappa}$ such that $\mu(\{x\})=\mu(x)$ for
every $x \in X$ (we follow the convention
\eqref{dudu}), and let $C_{\phi}$ be the corresponding
composition operator in $L^2(\mu)$ with the symbol
$\phi=\phi_{\eta,\kappa}$. Since $\phi(X)=X$, we infer
from \eqref{hfi} that $\hfi(x) > 0$ for every $x\in
X$.
   \end{proc}
   \subsection{Three key lemmata}
We begin by listing the most fundamental properties of
the composition operator $C_\phi$ constructed in
Procedure \ref{proceed2}.
   \begin{lem} \label{pierwsze}
Let $\kappa$, $\eta$, $\nu$, $\tau$,
$\{\varDelta_i\}_{i=1}^\eta$, $X$, $\mu$,
$P(x_{\kappa}, \cdot)$, $\{P(x_{i,j},\cdot)\}_{i\in
J_{\eta}, j\in \nbb}$ and $C_{\phi}$ be as in
Procedure {\em \ref{proceed2}}. Then
$\overline{\dzn{C_{\phi}}} = L^2(\mu)$ and the
following holds{\em :}
   \begin{enumerate}
   \item[(i)]
$\{\hfin{n}(x_{\kappa})\}_{n=0}^\infty$ is an
H-determinate Stieltjes moment sequence with the
S-representing measure $P(x_{\kappa}, \cdot)$ whose
index of H-determinacy at $0$ is $0$,
   \item[(ii)]  for all $i\in J_{\eta}$
and $j\in \nbb$, $\{\hfin{n}(x_{i,j})\}_{n=0}^\infty$
is a Stieltjes moment sequence with the S-representing
measure $P(x_{i,j}, \cdot)$,
   \item[(iii)] the
condition \mbox{{\em (i-a)}} of Theorem {\em
\ref{main}} holds and the family
$\{P(x_{i,1},\cdot)\}_{i\in J_{\eta}}$ satisfies the
conditions \mbox{{\em (i-b)}}, \mbox{{\em (i-c)}} and
\mbox{{\em (i-d$^{\prime}$)}} of Theorems {\em
\ref{main}} and {\em \ref{6.2}},
   \item[(iv)]
$\{P(x_{i,1}, \cdot)\}_{i\in J_{\eta}}$ does not
satisfy the condition {\em \mbox{(i-d)}} of Theorem
{\em \ref{main}},
   \item[(v)] if $\kappa=0$, then $C_{\phi}$ is hyponormal
if and only if
   \begin{align} \label{Bud-ski-1}
\sum_{i=1}^{\eta} \frac{(\int_{\varDelta_i}(t-1) \D
\tau(t))^2}{\int_{\varDelta_i}t (t-1) \D \tau(t)} \Le
\frac{\int_0^\infty (t-1) \D
\tau(t)}{1+\int_0^\infty(t-1) \D \tau(t)},
   \end{align}
   \item[(vi)] if $\kappa=0$, then $C_{\phi}$ generates
Stieltjes moment sequences,
   \item[(vii)] if $\kappa=0$ and $\eta=1$, then there
exists a family $\{P^\prime(x,\cdot)\}_{x\in X}$ of
Borel probability measures on $\rbb_+$ that satisfies
\eqref{cc} and thus $C_{\phi}$ is subnormal.
   \end{enumerate}
   \end{lem}
   \begin{proof}
(ii) It follows from \eqref{hfi} and \eqref{muxij}
that
   \begin{align} \label{gs9}
\hfin{n}(x_{i,1}) = \int_0^\infty t^n P(x_{i,1}, \D
t), \quad n \in \zbb_+, \, i \in J_{\eta}.
   \end{align}
Moreover, we have
   \begin{align*}
\int_0^\infty t^n P(x_{i,j}, \D t) &
\overset{\eqref{muxij1}}=
\frac{\mu(x_{i,1})}{\mu(x_{i,j})} \int_0^\infty
t^{j+n-1} P(x_{i,1}, \D t)
   \\
& \hspace{.5ex}\overset{\eqref{muxij}} =
\frac{\mu(x_{i,1})}{\mu(x_{i,j})}
\frac{\mu(x_{i,j+n})}{\mu(x_{i,1})}
   \\
&\hspace{.5ex}\overset{\eqref{hfi}}=
\hfin{n}(x_{i,j}), \quad n \in \zbb_+, \, i \in
J_{\eta}, \, j \in \nbb_2.
   \end{align*}
Altogether this implies that (ii) holds. In
particular, the condition \mbox{(i-a)} of Theorem~
\ref{main} holds. By \eqref{gs2}, \eqref{gs6},
\eqref{gs7} and \eqref{gs9}, the family
$\{P(x_{i,1},\cdot)\}_{i\in J_{\eta}}$ satisfies the
condition \mbox{(i-b)} of Theorem \ref{main}.

(iv) The condition \mbox{(i-b)} of Theorem \ref{main},
\eqref{gs6}, \eqref{gs7} and \eqref{gs8} yield
   \begin{align} \notag
0 < \sum_{i=1}^{\eta}
\frac{\mu(x_{i,1})}{\mu(x_{\kappa})} \int_0^\infty
\frac{1}{t^{\kappa+1}-1} P(x_{i,1}, \D t) & =
\sum_{i=1}^{\eta} \int_{\varDelta_i}
\frac{1}{t^{\kappa}} \D \tau(t)
   \\  \label{wim}
& = \int_0^\infty \frac{1}{t^{\kappa}} \D \tau(t)
\overset{\eqref{xxx1}}< \infty.
   \end{align}
Hence, the family $\{P(x_{i,1}, \cdot)\}_{i\in
J_{\eta}}$ satisfies the condition
\mbox{(i-d$^{\prime}$)} of Theorem \ref{6.2} and the
quantity $\xi$ defined by \eqref{claim} is a real
number. Since the closed support of an N-extremal
measure has no accumulation point in $\rbb$ (cf.\
Lemma \ref{proN}), we infer from \eqref{gs1} and
\eqref{gs2} that
   \begin{align} \label{ndod}
\nu(\{1\}) > 0.
   \end{align}
Thus, noting that
   \begin{align*}
\xi &\overset{\eqref{claim}} = \frac{\mu(x_0)}
{\mu(x_{\kappa})} - \sum_{i=1}^{\eta}
\frac{\mu(x_{i,1})} {\mu(x_{\kappa})} \int_0^\infty
\frac{1}{t^{\kappa+1}-1} P(x_{i,1},\D t)
   \\
& \overset{\eqref{wim}} = \frac{\mu(x_0)}
{\mu(x_{\kappa})} - \int_0^\infty \frac{1}{t^{\kappa}}
\D \tau(t)
   \\
   & \hspace{1.4ex}\overset{(*)} =
   \begin{cases}
- \nu(\{1\}) \text{ if } \kappa \Ge 1,
   \\[1ex]
1 - \nu(\rbb_+) \overset{\eqref{gs3}} = - \nu(\{1\})
\text{ if } \kappa=0,
   \end{cases}
   \end{align*}
where ($*$) refers to \eqref{dd1} if $\kappa \Ge 1$
and to \eqref{gs1} if $\kappa=0$, we obtain
   \begin{align} \label{mcp}
\xi= - \nu(\{1\}) < 0.
   \end{align}
It follows from \eqref{mcp} that the assertion (iv)
holds.

(iii) In view of what has been done, it remains to
show that the condition \mbox{(i-c)} of Theorem
\ref{main} holds for $\kappa \in \nbb$. Suppose
$\kappa \in \nbb$ and $r\in J_{\kappa}$. Using
\eqref{gs6}, \eqref{gs7} and \eqref{gs8}, we get
   \begin{align*}
\frac{\mu(x_r)}{\mu(x_{\kappa})} - \sum_{i=1}^{\eta}
\frac{\mu(x_{i,1})} {\mu(x_{\kappa})} & \int_0^\infty
\frac{t^r}{t^{\kappa+1}-1} P(x_{i,1},\D t)
   \\
& \hspace{.2ex}= \frac{\mu(x_r)}{\mu(x_{\kappa})} -
\int_0^\infty \frac{1}{t^{\kappa-r}} \D \tau(t)
   \\
& \overset{(\dag)} =
   \begin{cases}
- \nu(\{1\}) \text{ if } r \in J_{\kappa-1},
   \\[1ex]
1 - \nu(\rbb_+) \overset{ \eqref{gs3} }= - \nu(\{1\})
\text{ if } r=\kappa,
   \end{cases}
   \end{align*}
where ($\dag$) refers to \eqref{gs1} and \eqref{dd1}.
This and \eqref{mcp} yield
   \begin{align} \label{qqq}
\frac{\mu(x_r)}{\mu(x_{\kappa})} - \sum_{i=1}^{\eta}
\frac{\mu(x_{i,1})} {\mu(x_{\kappa})} \int_0^\infty
\frac{t^r}{t^{\kappa+1}-1} P(x_{i,1},\D t) = \xi,
\quad r \in J_{\kappa}.
   \end{align}
It follows from \eqref{claim} and \eqref{qqq} that the
family $\{P(x_{i,1}, \cdot)\}_{i\in J_{\eta}}$
satisfies the condition \mbox{(i-c)} of Theorem
\ref{main}. Therefore (iii) holds.

(i) Arguing as in the proof of Theorem \ref{6.2}, we
verify that \eqref{xi-1} is satisfied. Hence, applying
\eqref{gs6}, \eqref{gs7} and \eqref{gs8} we get
\allowdisplaybreaks
   \begin{align*}
\hfin{n}(x_\kappa) & \overset{\eqref{xi-1}}= \xi +
\sum_{i=1}^{\eta} \frac{\mu(x_{i,1})}
{\mu(x_{\kappa})} \int_0^\infty t^n
\frac{t^{\kappa}}{t^{\kappa+1}-1} P(x_{i,1},\D t)
   \\
& \hspace{1.7ex}= \xi + \int_0^\infty t^n \D \tau(t)
   \\
& \overset{\eqref{gs1}}= \xi + \int_0^\infty t^n \D
\nu(t)
   \\
&\hspace{-3.8ex}\overset{\eqref{gs11}\&\eqref{mcp}}=
\int_0^\infty t^n \D P(x_\kappa, \D t), \quad n \in
\zbb_+.
   \end{align*}
This means that
$\{\hfin{n}(x_{\kappa})\}_{n=0}^\infty$ is a Stieltjes
moment sequence with the S-represent\-ing measure
$P(x_{\kappa}, \cdot)$. Employing \eqref{gs1},
\eqref{ndod}, \eqref{gs11} and \cite[Theorem
3.6]{ber-dur}, we deduce that
$\{\hfin{n}(x_{\kappa})\}_{n=0}^\infty$ is
H-determinate, $\ind{z}{P(x_{\kappa}, \cdot)}=0$ if
$z\in \cbb\setminus \supp{P(x_{\kappa}, \cdot)}$ and
$\ind{z}{P(x_{\kappa}, \cdot)}=1$ if $z\in
\supp{P(x_{\kappa}, \cdot)}$. Hence, by \eqref{gs2},
$\ind {\hspace{.2ex} 0} {P(x_{\kappa},\cdot)} = 0$.

It follows from (i) and Proposition \ref{hfi4} that
$\dzn{C_{\phi}}$ is dense in $L^2(\mu)$. In
particular, $C_{\phi}$ is densely defined.

Assume now that $\kappa=0$.

(v) It follows from the Cauchy-Schwarz inequality that
   \begin{align*}
\mu(x_{i,j+1})^2 \overset{\eqref{muxij}} \Le
\mu(x_{i,j}) \mu(x_{i,j+2}), \quad i\in J_\eta, \,
j\in \nbb.
   \end{align*}
Hence, the inequality \eqref{Bud-ski} holds for every
$x\in X \setminus \{x_0\}$. Note that
   \begin{align*}
\frac{1}{\mu(x_0)} \sum_{y \in \phi^{-1}(\{x_0\})} &
\frac{\mu(y)^2}{\mu(\phi^{-1}(\{y\}))} = \frac{1}{1 +
\sum_{i=1}^\eta \frac{\mu(x_{i,1})}{\mu(x_0)}} +
\sum_{i=1}^\eta \frac{\mu(x_{i,1})}{\mu(x_0)}
\frac{\mu(x_{i,1})}{\mu(x_{i,2})}
   \\
& \overset{\eqref{gs8}\&\eqref{muxij}}= \frac{1}{1 +
\sum_{i=1}^\eta \frac{1}{c_i}} + \sum_{i=1}^\eta
\frac{1}{c_i\int_0^\infty t P(x_{i,1},\D t)}
   \\
& \overset{\eqref{gs6}\&\eqref{gs10}}= \frac{1}{1 +
\int_0^\infty (t-1) \D \tau(t)} + \sum_{i=1}^\eta
\frac{\int_{\varDelta_i} (t-1) \D
\tau(t)}{\int_0^\infty t P(x_{i,1},\D t)}
   \\
& \overset{\eqref{gs10}\&\eqref{gs7}}= \frac{1}{1 +
\int_0^\infty (t-1) \D \tau(t)} + \sum_{i=1}^\eta
\frac{(\int_{\varDelta_i} (t-1) \D
\tau(t))^2}{\int_{\varDelta_i} t(t-1) \D \tau(t)}.
   \end{align*}
Therefore, the inequality \eqref{Bud-ski} holds for
$x=x_0$ if and only if \eqref{Bud-ski-1} is satisfied.
This combined with Proposition \ref{buda} yields (v).

(vi) This follows from (i), (ii) and \cite[Theorem
10.4]{b-j-j-sC}.

(vii) Apply (i), \eqref{gs2}, \eqref{gs11},
Proposition \ref{eta1} and Theorem \ref{glowne}.
   \end{proof}
   The next lemma, which is of technical nature, will
play a key role in constructing examples of exotic
composition operators in Sections \ref{s5.3} and
\ref{s5.4}.
   \begin{lem} \label{epslem}
Let $\beta \in \msc^+$ be such that $\beta(\rbb_+)
> 1$, $0 < \inf\supp{\beta}$ and $\supp
{\beta}=\{\theta_1, \theta_2, \ldots\}$, where
$\{\theta_i\}_{i=1}^\infty$ is an injective sequence.
Set $\theta_i^{(a)} = \psi_{a^{-1},a}(\theta_i)$ and
$\beta^{(a)} = \beta \circ \psi_{a^{-1},a}^{-1}$ for
$i \in \nbb$ and $a \in (0,\infty)$ $($see
\eqref{defpsi} and \eqref{dnf} for the necessary
definitions$)$. Then the following holds{\em :}
   \begin{enumerate}
   \item[(i)] $\beta^{(a)} \in \msc^+$ and $\supp{\beta^{(a)}}
= \{\theta_1^{(a)}, \theta_2^{(a)}, \ldots\} \subseteq
(1, \infty)$ for all $a \in (0,\infty)$,
   \item[(ii)] there exists $m \in \nbb$ such that
$\beta(\{\theta_1,\ldots,\theta_{j}\}) > 1$ for every
integer $j \Ge m$,
   \item[(iii)] if $m \in \nbb$ is such that $\beta(\{\theta_1,
\ldots, \theta_{m}\}) > 1$, then there exists $a_1 \in
(0,\infty)$ such that for every $a \in (0, a_1)$,
   \begin{align} \label{tetrablok}
\sum_{i=1}^{m} \frac{\theta_i^{(a)}-1}{\theta_i^{(a)}}
\; \beta^{(a)}\big(\{\theta_i^{(a)}\}\big)
> \frac{\int_0^\infty (t-1) \D \beta^{(a)}(t)}
{1+\int_0^\infty (t-1) \D \beta^{(a)}(t)},
   \end{align}
   \item[(iv)] there exists $a_2 \in (0,\infty)$
such that for every $a \in (0, a_2)$,
   \begin{align} \label{tatroj}
\sum_{i=1}^\infty
\frac{\theta_i^{(a)}-1}{\theta_i^{(a)}} \;
\beta^{(a)}\big(\{\theta_i^{(a)}\}\big)
> \frac{\int_0^\infty (t-1) \D \beta^{(a)}(t)}
{1+\int_0^\infty (t-1) \D \beta^{(a)}(t)},
   \end{align}
   \item[(v)] if $\kappa\in \nbb$, then there exists
$a_3 \in (0,\infty)$ such that for every $a \in
(a_3,\infty)$,
   \begin{align} \label{11X14}
1 + \int_0^\infty \frac{1}{t^{\kappa}} \D \beta^{(a)}
(t) > \beta^{(a)}(\rbb_+).
   \end{align}
   \end{enumerate}
   \end{lem}
   \begin{proof}
(i) Apply Lemma \ref{transH}.

(ii) Use the fact that $\lim_{j \to \infty}
\beta(\{\theta_1, \ldots, \theta_{j}\}) =
\beta(\rbb_+) > 1$.

(iii) It follows from our assumptions that
   \begin{align*}
\lim_{a \to 0+} \sum_{i=1}^{m}
\frac{\theta_i}{a+\theta_i} \, \beta(\{\theta_i\}) =
\beta(\{\theta_1, \ldots, \theta_{m}\}) > 1
   \end{align*}
and $\lim_{a \to 0+} \frac{\int_0^\infty t \D
\beta(t)}{a+\int_0^\infty t \D \beta(t)} = 1$. Hence,
there exists $a_1 \in (0,\infty)$ such that
   \begin{align} \label{com1}
\sum_{i=1}^{m} \frac{\theta_i}{a + \theta_i} \,
\beta\big(\{\theta_i\}\big)
> \frac{\int_0^\infty t \D \beta(t)}
{a + \int_0^\infty t \D \beta(t)}, \quad a \in
(0,a_1).
   \end{align}
Using \eqref{defpsi}, we easily verify that
   \begin{align} \label{com2}
\sum_{i=1}^{m} \frac{\theta_i^{(a)}-1}{\theta_i^{(a)}}
\; \beta^{(a)}\big(\{\theta_i^{(a)}\}\big) =
\sum_{i=1}^{m} \frac{\theta_i}{a + \theta_i}
\beta\big(\{\theta_i\}\big), \quad a \in (0,\infty).
   \end{align}
Applying the measure transport theorem, we obtain
   \begin{align} \label{com3}
\frac{\int_0^\infty (t-1) \D \beta^{(a)}(t)}
{1+\int_0^\infty (t-1) \D \beta^{(a)}(t)}
=\frac{\int_0^\infty t \D \beta(t)} {a + \int_0^\infty
t \D \beta(t)}, \quad a \in (0,\infty).
   \end{align}
Combining \eqref{com1}, \eqref{com2} and \eqref{com3}
yields (iii).

(iv) This follows from (i), (ii) and (iii).

(v) Note that
   \begin{align*}
\lim_{a \to \infty} \sum_{i=1}^\infty
\Big(\frac{a}{\theta_i + a}\Big)^{\kappa}
\beta(\{\theta_i\}) = \beta(\rbb_+).
   \end{align*}
Hence, there exists $a_3 \in (0,\infty)$ such that
   \begin{align}    \label{waw}
1 + \sum_{i=1}^\infty \Big(\frac{a}{\theta_i +
a}\Big)^{\kappa} \beta(\{\theta_i\})
> \beta(\rbb_+), \quad a \in (a_3,\infty).
   \end{align}
Using (i) and the measure transport theorem, we get
\allowdisplaybreaks
   \begin{align*}
1 + \int_0^\infty \frac{1}{t^{\kappa}} \D \beta^{(a)}
(t) & = 1 + \int_0^\infty
\frac{1}{\big(\psi_{a^{-1},a}(t)\big)^{\kappa}} \D
\beta (t)
   \\
&= 1 + \sum_{i=1}^\infty \Big(\frac{a}{\theta_i +
a}\Big)^{\kappa} \beta(\{\theta_i\})
   \\
& \hspace{-2.2ex}\overset{\eqref{waw}}
> \beta(\rbb_+) =
\beta^{(a)}(\rbb_+), \quad a \in (a_3,\infty).
   \end{align*}
This completes the proof.
   \end{proof}
The third lemma is related to \mbox{$q$-Pochhammer}
symbol $(z;q)_{\infty}$ (see \eqref{Poch} for its
definition). The function $(0,1) \ni q \mapsto
(q;q)_{\infty} \in (0,1)$ is called the {\em Euler
function}.
   \begin{lem} \label{Euler}
Let $a \in (1,\infty)$. Then there exists $q_0 \in
\big(0,\frac{1}{a}\big)$ such that
   \begin{align*}
(q/a;q)_{\infty} + (aq;q)_{\infty} > 1, \quad q \in
(0,q_0).
   \end{align*}
   \end{lem}
   \begin{proof}
Since the function $n \mapsto \frac{(3n-1)n}{2}$ maps
$\zbb$ injectively into $\zbb_+$, we deduce from
Euler's pentagonal-number theorem (cf.\ \cite[Theorem
14.3]{Apos}) that
   \begin{align} \notag
(q;q)_{\infty} & = \sum_{n=-\infty}^{\infty} (-1)^n
q^{\frac{(3n-1)n}{2}}
   \\ \notag
& = \sum_{k=-\infty}^{\infty} q^{(6k-1)k} -
\sum_{k=-\infty}^{\infty} q^{(3k+1)(2k+1)}
   \\ \notag
& > 1 - \sum_{k=-\infty}^{\infty} q^{(3k+1)(2k+1)}
   \\ \notag
& > 1 - \sum_{k=1}^{\infty} q^k
   \\ \label{eeq}
& = 1 - \frac{q}{1-q}, \quad q \in (0,1).
   \end{align}
It is easily seen that there exists $q_0 \in
(0,\frac{1}{a})$ such that
   \begin{align*}
1 - \frac{q}{1-q} > \frac{1}{2(1-aq)}, \quad q \in
(0,q_0).
   \end{align*}
Hence, by \eqref{eeq}, we have
   \begin{align*}
2(1-aq) (q;q)_{\infty} > 1, \quad q \in (0,q_0).
   \end{align*}
This combined with $\frac{1}{a} < a$ and $aq_0 < 1$
imply that
   \begin{align*}
\prod_{j=1}^{\infty} \Big(1-\frac{1}{a} q^j\Big) +
\prod_{j=1}^{\infty} (1-a q^j) & > 2 (1-a q)
\prod_{j=1}^{\infty} (1-aq q^j)
   \\
&> 2 (1-a q) (q;q)_{\infty} > 1, \quad q \in (0,q_0),
   \end{align*}
which completes the proof.
   \end{proof}
   \subsection{The gap between
the conditions \mbox{(i-d)} and
\mbox{(i-d$^{\prime}$)}} \label{s5.3} In this section,
we show that the assertion (ii) of Theorem \ref{main}
is no longer true if the condition \mbox{(i-d)} of
Theorem \ref{main} is replaced by the condition
\mbox{(i-d$^{\prime}$)} of Theorem \ref{6.2}. As shown
in Theorem \ref{drugie1} below, this can happen even
for subnormal composition operators. In fact, this
phenomenon is independent of whether the operator in
question is subnormal or not (cf.\ Remark
\ref{whynot}).

We begin with the case of $\eta \Ge 2$.
   \begin{thm} \label{drugie}
Let $\kappa \in \zbb_+$ and $\eta \in \nbb_2 \cup
\{\infty\}$. Then there exists a discrete measure
$\mu$ on $X=X_{\eta,\kappa}$ such that the composition
operator $C_{\phi}$ in $L^2(\mu)$ with
$\phi=\phi_{\eta,\kappa}$ has the property that
$\overline{\dzn{C_{\phi}}}=L^2(\mu)$ and the following
conditions hold{\em :}
   \begin{enumerate}
   \item[(i)]
$\{\hfin{n}(x_{\kappa})\}_{n=0}^\infty$ is an
H-determinate Stieltjes moment sequence with index of
H-determinacy at $0$ equal to $0$,
   \item[(ii)]
 for all $i\in J_{\eta}$ and $j\in \nbb$,
$\{\hfin{n}(x_{i,j})\}_{n=0}^\infty$ is an
H-determinate Stieltjes moment sequence with infinite
index of H-determinacy; in particular, the condition
\mbox{{\em (i-a)}} of Theorem {\em \ref{main}} holds,
   \item[(iii)] there exists a unique family
$\{P(x_{i,1},\cdot)\}_{i\in J_{\eta}}$ of Borel
probability measures on $\rbb_+$ that satisfies the
conditions \mbox{{\em (i-b)}}, \mbox{{\em (i-c)}} and
\mbox{{\em (i-d$^{\prime}$)}} of Theorems {\em
\ref{main}} and {\em \ref{6.2}}$;$ this family does
not satisfy the condition {\em \mbox{(i-d)}} of
Theorem {\em \ref{main}},
   \item[(iv)] there is no family
$\{P^\prime(x,\cdot)\}_{x\in X}$ of Borel probability
measures on $\rbb_+$ that satisfies \eqref{cc},
   \item[(v)] if $\kappa=0$, then $C_{\phi}$ generates
Stieltjes moment sequences.
   \end{enumerate}
   \end{thm}
   \begin{proof}
First, we note that if (i) holds, then by Proposition
\ref{hfi4}, $\dzn{C_{\phi}}$ is dense in $L^2(\mu)$.
We begin by proving that the condition (iv) follows
from (i), (ii) and (iii). Suppose, on the contrary,
that there exists a family
$\{P^\prime(x,\cdot)\}_{x\in X}$ of Borel probability
measures on $\rbb_+$ that satisfies \eqref{cc}. It
follows from Theorem \ref{glowne} that
$P^\prime(x_{i,1},\cdot)$ is an S-representing measure
of $\{\hfin{n}(x_{i,1})\}_{n=0}^\infty$ for every
$i\in J_\eta$. Hence, by (ii) and (iii),
$P^\prime(x_{i,1},\cdot)=P(x_{i,1},\cdot)$ for every
$i\in J_\eta$. This together with Theorem
\ref{main}(i) implies that the family
$\{P(x_{i,1},\cdot)\}_{i\in J_\eta}$ satisfies the
condition \mbox{(i-d)} of Theorem \ref{main}, which
contradicts (iii).

Take any S-indeterminate Stieltjes moment sequence
$\gammab = \{\gamma_n\}_{n=0}^{\infty}$ (see e.g., the
classical example due to Stieltjes \cite{Sti}; see
also Section \ref{SecASC}). Let $\alpha$ and $\beta$
be the Krein and the Friedrichs measures of $\gammab$
(see Section \ref{secbc}). Then by Lemma \ref{proN}
and \eqref{KFm}, $0 < \alpha((0,\infty)) <
\alpha(\rbb_+)$. Hence, replacing $(\alpha, \beta)$ by
$(r\alpha, r\beta)$ with $r :=
\alpha((0,\infty))^{-1}$ if necessary, we may assume
that
   \begin{align} \label{gs1+}
   \begin{minipage}{58ex}
$\alpha$ and $\beta$ are N-extremal measures of the
same Stieltjes moment sequence such that
$0=\inf\supp{\alpha} < \inf\supp{\beta}$ and
$\alpha(\rbb_+) = 1+ \alpha(\{0\}) > 1$.
   \end{minipage}
   \end{align}
It follows from Lemma \ref{epslem}(v) that there
exists $a \in (0,\infty)$ such that \eqref{11X14}
holds. This combined with \eqref{gs1+} and Lemma
\ref{transH} shows that the measures
$\nu:=\alpha^{(a)}$ and $\tau:=\beta^{(a)}$ satisfy
the conditions \eqref{gs1}-\eqref{gs4} of Procedure
\ref{proceed2}. Since the measure $\tau$ is N-extremal
and $\eta \Ge 2$, we infer from Lemma \ref{proN} that
there exists a partition
$\{\varDelta_i\}_{i=1}^{\eta}$ of $\supp{\tau}$ such
that
   \begin{align} \label{dell}
\card{\supp{\tau} \setminus \varDelta_i} = \aleph_0,
\quad i\in J_{\eta}.
   \end{align}

Let $X$, $\mu$, $P(x_{\kappa}, \cdot)$,
$\{P(x_{i,j},\cdot)\}_{i\in J_{\eta}, j\in \nbb}$ and
$C_{\phi}$ be as in Procedure \ref{proceed2}. We will
show that the operator $C_{\phi}$ has all the required
properties. Indeed, set $\xi_i = \sum_{\lambda \in
\varDelta_i} \tau(\{\lambda\}) \delta_{\lambda}$ for
every $i \in J_{\eta}$. We deduce from the equality
\eqref{dell} and Theorem \ref{b-d} that each $\xi_i$
is H-determinate and
   \begin{align} \label{ajaja}
\text{$\ind z {\xi_i} = \infty$ for all $z\in \cbb$
and $i \in J_{\eta}$.}
   \end{align}
Fix $(i,j) \in J_{\eta} \times \nbb$ and $z\in \cbb$.
It follows from \eqref{gs2} that $\supp{\xi_i}
\subseteq (1,\infty)$. Hence, by \eqref{gs7} and
\eqref{muxij1}, we see that for every $k\in \zbb_+$,
\allowdisplaybreaks
   \begin{align*}
\int_{\sigma} |t-z|^{2k} P(x_{i,j},\D t) & = c_i
\frac{\mu(x_{i,1})}{\mu(x_{i,j})} \int_{\sigma}
|t-z|^{2k} t^{j-1} \frac{t^{\kappa+1}-1}{t^{\kappa}}
\chi_{\varDelta_i}(t) \D \tau(t)
   \\
& \Le c_i \frac{\mu(x_{i,1})}{\mu(x_{i,j})}
(1+|z|)^{2k} \int_{\sigma} t^{j+\kappa+2k} \D \xi_i(t)
   \\
& \Le c_i \frac{\mu(x_{i,1})}{\mu(x_{i,j})}
(1+|z|)^{2k} \int_{\sigma} t^{2(j+\kappa + k)} \D
\xi_i(t), \quad \sigma \in \borel{\rbb}.
   \end{align*}
This, \eqref{ajaja}, \eqref{indinf} and Proposition
\ref{krytnd} imply that the measure $P(x_{i,j},
\cdot)$ is H-determinate and $\ind z {P(x_{i,j},
\cdot)} = \infty$. Applying Lemma \ref{pierwsze}, we
conclude that $C_{\phi}$ satisfies the conditions
(i)-(v). This completes the proof.
   \end{proof}
The case of $\eta=1$ turns out to be surprisingly
different from the previous one (compare Theorem
\ref{drugie}(iv) with Theorem \ref{drugie1}(iv)).
   \begin{thm} \label{drugie1}
Let $\kappa \in \zbb_+$. Then there exists a discrete
measure $\mu$ on $X=X_{1,\kappa}$ such that the
composition operator $C_{\phi}$ in $L^2(\mu)$ with
$\phi=\phi_{1,\kappa}$ has the property that
$\overline{\dzn{C_{\phi}}}=L^2(\mu)$ and the following
conditions hold{\em :}
   \begin{enumerate}
   \item[(i)]
$\{\hfin{n}(x_{\kappa})\}_{n=0}^\infty$ is an
H-determinate Stieltjes moment sequence with index of
H-determinacy at $0$ equal to $0$,
   \item[(ii)]
for every $j\in \nbb$,
$\{\hfin{n}(x_{1,j})\}_{n=0}^\infty$ is an
S-indeterminate Stieltjes moment sequence$;$ in
particular, the condition \mbox{{\em (i-a)}} of
Theorem {\em \ref{main}} holds,
   \item[(iii)] there exists a Borel probability
measure $P(x_{1,1},\cdot)$ on $\rbb_+$ that satisfies
the conditions \mbox{{\em (i-b)}}, \mbox{{\em (i-c)}}
and \mbox{{\em (i-d$^{\prime}$)}} of Theorems {\em
\ref{main}} and {\em \ref{6.2}}, and does not satisfy
the condition {\em \mbox{(i-d)}} of Theorem {\em
\ref{main}},
   \item[(iv)]
if $\kappa=0$, then there exists a family
$\{P^{\prime}(x,\cdot)\}_{x\in X}$ of Borel
probability measures on $\rbb_+$ that satisfies
\eqref{cc} and, consequently, $C_{\phi}$ is subnormal.
   \end{enumerate}
   \end{thm}
   \begin{proof}
As in the proof of Theorem \ref{drugie}, we see that
there exist measures $\nu$ and $\tau$ satisfying the
conditions \eqref{gs1}-\eqref{gs4} of Procedure
\ref{proceed2}. The only possible partition
$\{\varDelta_i\}_{i=1}^{\eta}$ of $\supp{\tau}$ with
$\eta=1$ is $\varDelta_1=\supp{\tau}$. Let $X$, $\mu$,
$P(x_{\kappa}, \cdot)$,
$\{P(x_{1,j},\cdot)\}_{j=1}^{\infty}$ and $C_{\phi}$
be as in Procedure \ref{proceed2}. In view of Lemma
\ref{pierwsze}, it remains to show that for every
$j\in \nbb$, the Stieltjes moment sequence
$\{\hfin{n}(x_{1,j})\}_{n=0}^\infty$ is
S-indeterminate.

Suppose, on the contrary, that there exists
$k\in \nbb$ such that
$\{\hfin{n}(x_{1,k})\}_{n=0}^\infty$ is S-determinate.
We show that then $\{\hfin{n}(x_{1,1})\}_{n=0}^\infty$
is S-determinate. For this we may assume that $k \Ge
2$. It follows from \eqref{hfi} that
   \begin{align*}
\hfin{n}(x_{1,k}) = \frac{\mu(x_{1,k-1})}
{\mu(x_{1,k})} \hfin{n+1}(x_{1,k-1}), \quad n \in
\zbb_+.
   \end{align*}
Hence, by \cite[Proposition 5.12]{sim},
$\{\hfin{n}(x_{1,k-1})\}_{n=0}^\infty$ is
S-determinate. Applying backward induction, we deduce
that $\{\hfin{n}(x_{1,1})\}_{n=0}^\infty$ is
S-determinate. Since, by Lemma \ref{pierwsze}(ii),
$P(x_{1,1}, \cdot)$ is an S-representing measure of
$\{\hfin{n}(x_{1,1})\}_{n=0}^\infty$ and $P(x_{1,1},
\{0\}) = 0$ (see \eqref{gs2} and \eqref{gs7}), we
infer from \cite[Corollary, p.\ 481]{chi} (see also
\cite[Lemma 2.2.5]{j-j-s0}) that $P(x_{1,1}, \cdot)$
is H-determinate. By \eqref{gs2}, there exists $M \in
(0,\infty)$ such that
   \begin{align*}
\frac{t^{\kappa+1}-1}{t^{\kappa}} \Ge M, \quad t \in
[\inf\supp{\tau}, \infty).
   \end{align*}
This combined with \eqref{gs7} and Proposition
\ref{krytnd} (applied to $\tau$ and $P(x_{1,1},
\cdot)$) implies that $\tau$ is H-determinate, which
contradicts \eqref{gs1}.
   \end{proof}
Regarding Theorem \ref{drugie1}, note that if
$\kappa=0$, then (ii) can also be deduced from (iii)
and (iv) by arguing as in the first paragraph of the
proof of Theorem \ref{drugie}.
   \subsection{Exotic non-hyponormality} \label{s5.4}
   The main purpose of this section is to construct
non-hyponormal composition operators $C_{\phi}$ in
$L^2(X,2^{X},\mu)$ that generate Stieltjes moment
sequences with $X=X_{\eta,0}$ and
$\phi=\phi_{\eta,0}$. Recall that
$C_{\phi_{\eta,\kappa}}$ is always injective (cf.\
\eqref{injcos}). We begin with the case of
$\eta=\infty$.
   \begin{thm} \label{trzecie}
There exists a discrete measure $\mu$ on
$X=X_{\infty,0}$ such that the composition operator
$C_{\phi}$ in $L^2(\mu)$ with $\phi=\phi_{\infty,0}$
has the following properties{\em :}
   \begin{enumerate}
   \item[(i)] $C_{\phi}$ is injective and generates Stieltjes
moment sequences,
   \item[(ii)] $C_{\phi}$ is not hyponormal,
   \item[(iii)]
$\{\hfin{n}(x_{0})\}_{n=0}^\infty$ is an H-determinate
Stieltjes moment sequence with index of H-determinacy
at $0$ equal to $0$,
   \item[(iv)] $\{\hfin{n}(x)\}_{n=0}^\infty$ is an
H-determinate Stieltjes moment sequence with infinite
index of H-determinacy for all $x\in X \setminus
\{x_0\}$.
   \end{enumerate}
   \end{thm}
   \begin{proof}
Arguing as in the proof of Theorem \ref{drugie}, we
see that there exist measures $\alpha$ and $\beta$
satisfying \eqref{gs1+}. Let
$\{\theta_i\}_{i=1}^\infty$ be an injective sequence
such that $\supp{\beta} = \{\theta_1, \theta_2,
\ldots\}$ (cf.\ Lemma \ref{proN}). It follows from
Lemma \ref{epslem}(iv) that there exists $a \in
(0,\infty)$ such that \eqref{tatroj} holds. This and
Lemma \ref{transH} show that the measures
$\nu:=\alpha^{(a)}$ and $\tau:=\beta^{(a)}$ satisfy
the conditions \eqref{gs1}-\eqref{gs4} of Procedure
\ref{proceed2} for $\kappa=0$ and $\eta=\infty$. Note
also that in view of Lemma \ref{epslem}(i),
$\supp{\tau} = \{\theta_1^{(a)}, \theta_2^{(a)},
\ldots\}$. Set $\varDelta_i = \{\theta_i^{(a)}\}$ for
$i \in \nbb$. Clearly, the sequence
$\{\varDelta_i\}_{i=1}^{\infty}$ satisfies the
condition \eqref{gs6} for $\eta=\infty$.

Let $X$, $\mu$, $\{P(x, \cdot)\}_{x\in X}$ and
$C_{\phi}$ be as in Procedure \ref{proceed2} for
$\kappa=0$ and $\eta=\infty$. It follows from
\eqref{gs7}, \eqref{muxij} and \eqref{muxij1} that
$P(x_{i,j},\cdot) = \delta_{\theta_i^{(a)}}(\cdot)$
for all $i,j\in \nbb$. This is easily seen to imply
(iv). Since \eqref{tatroj} yields
   \begin{align*}
\sum_{i=1}^{\infty} \frac{(\int_{\varDelta_i}(t-1) \D
\tau(t))^2}{\int_{\varDelta_i}t (t-1) \D \tau(t)} >
\frac{\int_0^\infty (t-1) \D
\tau(t)}{1+\int_0^\infty(t-1) \D \tau(t)},
   \end{align*}
an application of Lemma \ref{pierwsze} completes the
proof.
   \end{proof}
Now we concern ourselves with the case when $\eta$ is
an arbitrary integer greater than or equal to $2$. The
most striking case is that of $\eta=2$. Below, given
$x\in \rbb$, we write $\lfloor x \rfloor$ for the
largest integer not greater than $x$, and $\lceil x
\rceil$ for the smallest integer not less than $x$.
   \begin{thm} \label{czwarte}
Let $\eta \in \nbb_2$. Then there exists a discrete
measure $\mu$ on $X=X_{\eta,0}$ such that the
composition operator $C_{\phi}$ in $L^2(\mu)$ with
$\phi=\phi_{\eta,0}$ has the following properties{\em
:}
   \begin{enumerate}
   \item[(i)] $C_{\phi}$ is injective and generates Stieltjes
moment sequences,
   \item[(ii)] $C_{\phi}$ is not hyponormal,
   \item[(iii)] $\{\hfin{n}(x)\}_{n=0}^\infty$ is a
Stieltjes moment sequence for every $x\in X$,
   \item[(iv)]
$\{\hfin{n}(x_{0})\}_{n=0}^\infty$ is H-determinate
with index of H-determinacy at $0$ equal to~ $0$,
   \item[(v)]
for all $i\in J_{\eta-1}$ and $j\in \nbb$, the
sequence $\{\hfin{n}(x_{i,j})\}_{n=0}^\infty$ is
H-determinate with infinite index of H-determinacy,
   \item[(vi)] for every $j \in J_{2\eta-4}$,
the sequence $\{\hfin{n}(x_{\eta,j})\}_{n=0}^\infty$
is H-determinate and its unique H-representing measure
$P(x_{\eta,j}, \cdot)$ satisfies the following
condition
   \begin{align} \label{ind0}
\eta-2 - \lceil j/2 \rceil \Le \ind 0 {P(x_{\eta,j},
\cdot)} \Le \eta-2 - \lfloor j/2 \rfloor.
   \end{align}
   \end{enumerate}
   \end{thm}
   \begin{proof}
We split the proof into two steps.

{\sc Step 1.} Let $\eta\in\nbb_2$. Then the conclusion
of Theorem \ref{czwarte} holds if there exist measures
$\alpha$ and $\beta$ satisfying \eqref{gs1+} and an
injective sequence $\{\theta_i\}_{i=1}^{\eta-1}
\subseteq \supp {\beta}$ such that
   \begin{align}  \label{zemeta}
\beta(\{\theta_1, \ldots, \theta_{\eta-1}\}) > 1.
   \end{align}

Indeed, by Lemma \ref{proN}, we can extend the
sequence $\{\theta_i\}_{i=1}^{\eta-1}$ to an injective
sequence $\{\theta_i\}_{i=1}^\infty$ such that $\supp
{\beta}=\{\theta_1, \theta_2, \ldots\}$. It follows
from \eqref{zemeta} and Lemma \ref{epslem}(iii) that
there exists $a \in (0,\infty)$ such that the
inequality \eqref{tetrablok} holds for $m=\eta-1$.
This together with \eqref{gs1+} and Lemma \ref{transH}
imply that the measures $\nu:=\alpha^{(a)}$ and
$\tau:=\beta^{(a)}$ satisfy the conditions
\eqref{gs1}-\eqref{gs4} of Procedure \ref{proceed2}
for $\kappa=0$, and $\supp{\tau} = \{\theta_1^{(a)},
\theta_2^{(a)}, \ldots\}$. Now we define a partition
$\{\varDelta_i\}_{i=1}^{\eta}$ of $\supp{\tau}$~ by
   \begin{align} \label{dopor}
\varDelta_i =
   \begin{cases}
\{\theta_i^{(a)}\} & \text{ if } i \in J_{\eta-1},
   \\[1ex]
   \{\theta_{\eta}^{(a)}, \theta_{\eta+1}^{(a)},
   \dots\} & \text{ if } i = \eta.
   \end{cases}
   \end{align}

Let $X$, $\mu$, $\{P(x, \cdot)\}_{x\in X}$ and
$C_{\phi}$ be as in Procedure \ref{proceed2} for
$\kappa=0$. Since $\eta \in \nbb_2$, it follows from
\eqref{gs2} that \allowdisplaybreaks
   \begin{align*}
\sum_{i=1}^{\eta} \frac{(\int_{\varDelta_i}(t-1) \D
\tau(t))^2}{\int_{\varDelta_i}t (t-1) \D \tau(t)} & >
\sum_{i=1}^{\eta-1} \frac{(\int_{\varDelta_i}(t-1) \D
\tau(t))^2}{\int_{\varDelta_i}t (t-1) \D \tau(t)}
   \\
& = \sum_{i=1}^{\eta-1}
\frac{\theta_i^{(a)}-1}{\theta_i^{(a)}} \;
\tau\big(\{\theta_i^{(a)}\}\big)
   \\
&\hspace{-1.7ex}\overset{\eqref{tetrablok}}>
\frac{\int_0^\infty (t-1) \D
\tau(t)}{1+\int_0^\infty(t-1) \D \tau(t)}.
    \end{align*}
Now applying Lemma \ref{pierwsze} we see that for
every $x \in X$, $\{\hfin{n}(x)\}_{n=0}^\infty$ is a
Stieltjes moment sequence with the S-representing
measure $P(x, \cdot)$ (in particular, (iii) holds),
and $C_{\phi}$ satisfies (i), (ii) and (iv) (note that
if $\eta=1$, then, by Lemma \ref{pierwsze}(vii),
$C_{\phi}$ is hyponormal). Arguing as in the proof of
Theorem \ref{trzecie}, we deduce that (v) holds as
well.

Our next aim is to prove (vi). Assume that $\eta \Ge
3$. Set $\xi_{\eta} = \sum_{\lambda \in
\varDelta_{\eta}} \tau(\{\lambda\}) \delta_{\lambda}$.
Clearly $\xi_{\eta} \in \msc^+$. By \cite[Theorem
3.6]{ber-dur}, $\xi_{\eta}$ is H-determinate and
   \begin{align} \label{ind1}
\ind 0 {\xi_{\eta}}=\eta-2.
   \end{align}
Put $m_j = c_{\eta} \frac{\mu(x_{\eta,1})}
{\mu(x_{\eta,j})}$ for $j\in \nbb$. It follows from
\eqref{gs7} and \eqref{muxij1} that
   \nobreak
   \begin{align} \notag
\int_{\sigma} t^{2k} P(x_{\eta,j},\D t) = m_j
\int_{\sigma} t^{2k+j-1} & (t-1) \D \xi_{\eta} (t),
   \\  \label{indn}
& \sigma \in \borel{\rbb}, k\in \zbb_+, j \in \nbb.
   \end{align}
Since, by \eqref{gs2}, $\supp{\xi_{\eta}} \subseteq
(1,\infty)$, we infer from \eqref{indn} that
\allowdisplaybreaks
   \begin{multline} \label{subst}
\int_{\sigma} t^{2k} P(x_{\eta,j},\D t) \Le m_j
\int_{\sigma} t^{2(k+\lceil j/2 \rceil)} \D \xi_{\eta}
(t) \Le m_j \int_{\sigma} t^{2(\eta-2)} \D \xi_{\eta}
(t),
   \\
\sigma \in \borel{\rbb}, \, k\in \zbb_+,\, k + \lceil
j/2 \rceil \Le \eta -2, \, j \in \nbb .
   \end{multline}
Fix $j \in J_{2\eta-4}$. Substituting $k=0$ into
\eqref{subst} and using \eqref{ind1} and Proposition
\ref{krytnd}, we deduce that $P(x_{\eta,j},\cdot)$ is
H-determinate. Hence, applying \eqref{ind1},
\eqref{subst} with $k=\eta-2 - \lceil j/2 \rceil$, and
Proposition \ref{krytnd}, we see that
   \begin{align} \label{ind3}
\ind 0 {P(x_{\eta,j}, \cdot)} \Ge \eta-2 - \lceil j/2
\rceil.
   \end{align}
Take now $k\in \zbb_+$ such that the measure $t^{2k}
P(x_{\eta,j},\D t)$ is H-determinate. Since
$\supp{\xi_{\eta}} \subseteq (1,\infty)$, there exists
$M \in (0,1)$ such that for every $\sigma \in
\borel{\rbb}$,
   \begin{align*}
\int_{\sigma} t^{2k} P(x_{\eta,j},\D t)
\overset{\eqref{indn}}\Ge M m_j \int_{\sigma} t^{2k+j}
\D \xi_{\eta} (t) \Ge M m_j \int_{\sigma}
t^{2(k+\lfloor j/2 \rfloor)} \D \xi_{\eta}(t).
   \end{align*}
Applying \eqref{ind1} and Proposition \ref{krytnd}
again, we deduce that $k+\lfloor j/2 \rfloor \Le
\eta-2$. This implies that
   \begin{align*}
\ind 0 {P(x_{\eta,j}, \cdot)} \Le \eta-2 - \lfloor j/2
\rfloor.
   \end{align*}
Combining the above inequality with \eqref{ind3}, we
obtain \eqref{ind0}.

{\sc Step 2.} There exist measures $\alpha$ and
$\beta$ satisfying \eqref{gs1+} such that
   \begin{align*}
\beta\big(\{\inf\supp{\beta}\}\big) > 1.
   \end{align*}

Indeed, let $a \in (1,\infty)$. By Lemma \ref{Euler},
there exists $q \in \big(0,\frac{1}{a}\big)$ such that
   \begin{align} \label{qaq}
(q/a;q)_{\infty} + (aq;q)_{\infty} > 1.
   \end{align}
As in \cite[Proposition 4.5.1]{b-v} (see also
\cite[eq.\ (4.4)]{b-v}), we define the Borel measures
$\zeta$ and $\rho$ on $\rbb$ (with different notation)
by \allowdisplaybreaks
   \begin{align} \label{qaq1}
   \begin{aligned}
\zeta & := \beta^{(a;q)} \circ \psi_{1,-1}^{-1}
\overset{\eqref{betaq}}= (aq;q)_{\infty}
\sum_{n=0}^{\infty} \frac{a^{n}
q^{n^2}}{(aq;q)_{n}(q;q)_{n}} \delta_{q^{-n} - 1},
   \\
\rho & := \gamma^{(a;q)} \circ \psi_{1,-1}^{-1}
\overset{\eqref{gamprezy}}= (q/a;q)_{\infty}
\sum_{n=0}^{\infty} \frac{a^{-n}
q^{n^2}}{(q/a;q)_{n}(q;q)_{n}} \,\delta_{aq^{-n} - 1}.
   \end{aligned}
   \end{align}
(Another possible choice of measures $\zeta$ and
$\rho$ is discussed in Remark \ref{piate}.) In view of
the proof of \cite[eq.\ (5.10)]{Ism}, $\zeta$ and
$\rho$ are probability measures. It follows from
\cite[Proposition 4.5.1]{b-v} that $\zeta$ and $\rho$
are N-extremal measures of the same Stieltjes moment
sequence (see also Section \ref{SecASC} for a detailed
discussion of this matter). Note that $\zeta(\{0\}) =
(aq;q)_{\infty} \in (0,1)$. Following the proof of
Theorem \ref{drugie}, we set $\alpha=r \zeta$ and
$\beta = r \rho$ with $r = (1- \zeta(\{0\}))^{-1}$. It
is easily seen that the measures $\alpha$ and $\beta$
satisfy \eqref{gs1+}. Now, combining \eqref{qaq} with
\eqref{qaq1}, we get
   \begin{align*}
\beta(\{\inf\supp{\beta}\}) =
\frac{(q/a;q)_{\infty}}{1-(aq;q)_{\infty}} > 1.
   \end{align*}

To finish the proof of the theorem, take measures
$\alpha$ and $\beta$ as in Step 2. Then $\supp
{\beta}=\{\theta_1, \theta_2, \ldots\}$, where
$\{\theta_i\}_{i=1}^\infty$ is a strictly increasing
sequence (cf.\ Lemma \ref{proN}). This implies that
$\beta(\{\theta_1\}) > 1$, and thus the inequality
\eqref{zemeta} holds for every $\eta \in \nbb_2$.
Applying Step 1 completes the proof.
   \end{proof}
To the best of our knowledge, there are few explicit
examples (or rather classes of examples) of N-extremal
measures. The ones used in the proof of Theorem
\ref{czwarte} are related to the Al-Salam-Carlitz
\mbox{$q$-polynomials} (see Section \ref{SecASC}). In
turn, the measures discussed below come from an
S-indeterminate Stieltjes moment problem associated
with a quartic birth and death process (cf.\
\cite{b-v}).
   \begin{rem} \label{piate}
Define the Borel measures $\zeta$ and $\rho$ on $\rbb$
by \allowdisplaybreaks
   \begin{align} \label{qaq2}
   \begin{aligned}
\zeta & = \frac{\pi}{K_0^2} \delta_{x_0} + \frac{4
\pi}{K_0^2} \sum_{n=1}^{\infty} \frac{2 n \pi}{\sinh(2
n \pi)} \delta_{x_{2n}},
   \\
\rho & = \frac{4 \pi}{K_0^2} \sum_{n=0}^{\infty}
\frac{(2n+1)\pi}{\sinh\big((2n+1)\pi\big)}
\delta_{x_{2n+1}},
   \end{aligned}
   \end{align}
where
   \begin{align*}
K_0 = \frac{\varGamma(\frac{1}{4})^2}{4 \sqrt{\pi}}
\quad \text{and} \quad x_k = \bigg(\frac{k
\pi}{K_0}\bigg)^{\!4\,} \text{ for } k\in \zbb_+.
   \end{align*}
($\varGamma(\cdot)$ stands for the Euler gamma
function.) It was proved in \cite[Proposition
3.4.1]{b-v} that $\zeta$ and $\rho$ are N-extremal
measures of the same Stieltjes moment sequence. A
careful inspection of \cite{b-v} reveals that these
measures are probabilistic. As a consequence,
$\zeta(\{0\}) = \pi K_0^{-2} \in (0,1)$. Note that
$\beta(\{\theta_1\}) > 1$, where, as in the proof of
Step 2 of Theorem \ref{czwarte}, $\beta = r \rho$ with
$r = (1- \zeta(\{0\}))^{-1}$ and
$\theta_1=\inf\supp{\beta}$. Indeed, since $\pi > 3$,
$\varGamma(\frac{1}{4}) < 4$ and $\sinh(\pi) < 12$, we
get
   \begin{align*}
\beta(\{\theta_1\}) = \frac{64
\pi^3}{(\varGamma(\frac{1}{4})^4 - 16 \pi^2)
\sinh(\pi)} > \frac{9}{7} > 1.
   \end{align*}
(In fact, $\beta(\{\theta_1\}) > \frac{23}{2}$.)
Putting all this together, we see that the proof of
Theorem \ref{czwarte} goes through without change if
the measures $\zeta$ and $\rho$ are defined by
\eqref{qaq2} instead of \eqref{qaq1}.
   \end{rem}
Below we make an important remark on the measures
$\zeta$ and $\rho$ and their relatives $\nu$ and
$\tau$ appearing in the proof of Theorem \ref{czwarte}
and in Remark \ref{piate}.
   \begin{rem} \label{fri}
Let $\zeta$ and $\rho$ be measures as in \eqref{qaq1}.
Recall that they are representing measures of the same
Stieltjes moment sequence, say $\gammab$. Since
$\inf\supp{\zeta}=0$, we infer from \eqref{KFm} that
$\zeta$ is the Krein measure of $\gammab$. In turn,
the measure $\rho$ is the Friedrichs measure of
$\gammab$. Indeed, this can be deduced by
combining\footnote{\;In fact \cite[Proposition
3.1]{Ped} is essentially due to Chihara (see
\cite[Theorem 5]{chi}). The reader is also referred to
\cite[Theorem 4.18, Proposition 5.20, and Remark, p.\
127]{sim} for another way of parameterizing N-extremal
measures in which the traditional interval
$[\alpha,0]$ is replaced by $[-1/\alpha, \infty) \cup
\{\infty\}$, where $\alpha$ is as in \cite[Proposition
3.2]{Ped}.} \eqref{KFm} and \cite[Proposition
3.1]{Ped} with \cite[Proposition 3.2]{Ped} and
\cite[Lemma 4.4.2]{b-v} (that this particular $\rho$
is the Friedrichs measure of $\gammab$ follows also
from Theorem \ref{taso}(i) and the discussion in the
paragraph containing \eqref{gamprezy}). Hence, by
Corollary \ref{transH+} (see also Theorem \ref{taso}),
the measures $\beta$ and $\tau$ are the Friedrichs
measures of appropriate S-indeterminate Stieltjes
moment sequences (see Step 1 of the proof of Theorem
\ref{czwarte}); however, $\nu$ is not the Krein
measure (of the S-indeterminate Stieltjes moment
sequence $\{\int_0^{\infty} t^n \D
\nu(t)\}_{n=0}^{\infty}$). It is easily seen that the
above discussion applies to the measures $\zeta$ and
$\rho$ defined by \eqref{qaq2} and the corresponding
measures $\nu$ and $\tau$ (but now we have to use
\cite[Corollary 3.3.3]{b-v} in place of \cite[Lemma
4.4.2]{b-v}).
   \end{rem}
We end this section with yet another remark.
   \begin{rem} \label{whynot}
Since Lemma \ref{pierwsze} was used as one of the main
tools to prove Theorems \ref{trzecie} and
\ref{czwarte}, the following statement can be added to
their conclusions:
   \begin{align*}
   \begin{minipage}{72ex}
{\em there exists a family $\{P(x_{i,1},\cdot)\}_{i\in
J_{\eta}}$ of Borel probability measures on $\rbb_+$
that satisfies the conditions {\em \mbox{(i-b)}}, {\em
\mbox{(i-c)}} and {\em \mbox{(i-d$^{\prime}$)}} of
Theorems {\em \ref{main}} and {\em \ref{6.2}}, and
does not satisfy the condition {\em \mbox{(i-d)}} of
Theorem {\em \ref{main}}.}
   \end{minipage}
   \end{align*}
   \end{rem}
   \subsection{Addendum}
The construction of the composition operator
$C_{\phi}$ that appears in the proofs of theorems of
Section \ref{sec5} depends on the choice of a
partition $\{\varDelta_i\}_{i=1}^\eta$ of
$\supp{\tau}$ (cf.\ Procedure \ref{proceed2}). In
particular, the hyponormality of $C_{\phi}$ requires
finding a partition $\{\varDelta_i\}_{i=1}^\eta$
satisfying the inequality \eqref{Bud-ski-1}. Hence, it
seems to be of some independent interest to calculate
the infimum and the supremum of the quantity appearing
on the left-hand side of \eqref{Bud-ski-1} over all
partitions $\{\varDelta_i\}_{i=1}^\eta$ of
$\supp{\tau}$. The following general proposition sheds
more light on this issue.
   \begin{pro} \label{supinf}
Let $\tau$ be an N-extremal measure such that $1 <
\inf\supp{\tau}$. Then for every $\eta \in \nbb \cup
\{\infty\}$, the following two conditions hold{\em :}
\allowdisplaybreaks
   \begin{align} \label{oczko1}
\inf\bigg\{\varLambda(\{\varDelta_i\}_{i=1}^\eta)
\colon \{\varDelta_i\}_{i=1}^\eta \text{ satisfies }
\eqref{gs6} \bigg\} & = \frac{\big(\int_0^\infty (t-1)
\D \tau(t)\big)^2}{\int_0^\infty t (t-1) \D \tau(t)},
   \\ \label{oczko2}
\sup\bigg\{\varLambda(\{\varDelta_i\}_{i=1}^\eta)
\colon \{\varDelta_i\}_{i=1}^\eta \text{ satisfies }
\eqref{gs6} \bigg\} &\Le \int_0^\infty \frac{t-1}{t}
\D \tau(t),
   \end{align}
where
   \begin{align} \label{oczko2.5}
\varLambda(\{\varDelta_i\}_{i=1}^\eta) =
\sum_{i=1}^{\eta} \frac{(\int_{\varDelta_i}(t-1) \D
\tau(t))^2}{\int_{\varDelta_i}t (t-1) \D \tau(t)},
\quad \{\varDelta_i\}_{i=1}^\eta \text{ satisfies }
\eqref{gs6}.
   \end{align}
If $\eta=\infty$, then the inequality in
\eqref{oczko2} becomes an equality. Moreover, we have
   \begin{align}  \label{oczko3}
\sup\bigg\{\varLambda(\{\varDelta_i\}_{i=1}^{\eta})
\colon \{\varDelta_i\}_{i=1}^{\eta} \text{ satisfies }
\eqref{gs6}, \, \eta \in \nbb\bigg\} = \int_0^\infty
\frac{t-1}{t} \D \tau(t).
   \end{align}
   \end{pro}
   \begin{proof}
We begin with two observations. First, if
$\{\varDelta_i\}_{i=1}^{\eta}$ is a partition of
$\supp{\tau}$, then by Lemma \ref{proN},
$\int_{\varDelta_i}t (t-1) \D \tau(t) \in (0,\infty)$
for every $i\in J_{\eta}$. Second, due to the
Cauchy-Schwarz inequality, the following two
inequalities hold \allowdisplaybreaks
   \begin{align} \label{kwxy}
& \bigg(\sum_{i=1}^{\eta} x_i\bigg)^2 \Le
\bigg(\sum_{i=1}^{\eta} y_i\bigg)
\bigg(\sum_{i=1}^{\eta} \frac{x_i^2}{y_i}\bigg), \quad
\{x_i\}_{i=1}^{\eta} \subseteq \rbb_+, \,
\{y_i\}_{i=1}^{\eta} \subseteq (0,\infty),
   \\ \label{deltain}
& \Big(\int_{\varDelta} (t-1) \D \tau(t)\Big)^2 \Le
\int_{\varDelta} t (t-1) \D \tau(t)\int_{\varDelta}
\frac{t-1}{t} \D \tau(t), \quad \varDelta \in
\borel{\rbb}.
   \end{align}
Therefore, if $\{\varDelta_i\}_{i=1}^\eta$ is a
partition of $\supp{\tau}$, then \allowdisplaybreaks
   \begin{align}  \notag
\frac{\big(\int_0^\infty (t-1) \D
\tau(t)\big)^2}{\int_0^\infty t (t-1) \D \tau(t)} & =
\frac{\big(\sum_{i=1}^{\eta} \int_{\varDelta_i} (t-1)
\D \tau(t)\big)^2}{\sum_{i=1}^{\eta}
\int_{\varDelta_i} t (t-1) \D \tau(t)}
   \\ \notag
& \hspace{-1.8ex}\overset{\eqref{kwxy}} \Le
\sum_{i=1}^{\eta} \frac{(\int_{\varDelta_i}(t-1) \D
\tau(t))^2}{\int_{\varDelta_i}t (t-1) \D \tau(t)}
\overset{\eqref{oczko2.5}} =
\varLambda(\{\varDelta_i\}_{i=1}^{\eta})
   \\  \notag
& \hspace{-1.8ex} \overset{\eqref{deltain}} \Le
\sum_{i=1}^{\eta} \int_{\varDelta_i} \frac{t-1}{t} \D
\tau(t)
   \\  \label{kwaint}
& = \int_0^{\infty} \frac{t-1}{t} \D \tau(t).
   \end{align}
This proves \eqref{oczko2} and the inequality
``$\Ge$'' in \eqref{oczko1}. To show the reverse
inequality, we may assume that $\eta \Ge 2$. Let us
define for every $k\in \nbb$ a partition
$\{\varDelta_{k,i}\}_{i=1}^\eta$ of $\supp{\tau}$ by
   \begin{align*}
\varDelta_{k,i} =
   \begin{cases}
   \{\theta_1, \ldots, \theta_k\} & \text{if } i=1,
\\
   \{\theta_{i+k-1}\} & \text{if } 2 \Le i< \eta,
\\
   \{\theta_{\eta+k-1}, \theta_{\eta+k}, \ldots\} &
   \text{if } \eta < \infty \text{ and } i=\eta,
   \end{cases}
   \end{align*}
where $\{\theta_j\}_{j=1}^\infty$ is a strictly
increasing sequence such that $\supp
{\beta}=\{\theta_1, \theta_2, \ldots\}$. By Lemma
\ref{proN} such a sequence exists and $\lim_{j\to
\infty} \theta_j = \infty$. Then
   \allowdisplaybreaks
   \begin{align*}
\varLambda(\{\varDelta_{k,i}\}_{i=1}^\eta) & =
\frac{(\int_0^{\theta_k} (t-1) \D
\tau(t))^2}{\int_0^{\theta_k} t (t-1) \D \tau(t)} +
\sum_{i=2}^{\eta} \frac{(\int_{\varDelta_{k,i}}(t-1)
\D \tau(t))^2}{\int_{\varDelta_{k,i}}t (t-1) \D
\tau(t)}
   \\
& \hspace{-1.8ex}\overset{\eqref{deltain}}{\Le}
\frac{(\int_0^{\theta_k} (t-1) \D
\tau(t))^2}{\int_0^{\theta_k} t (t-1) \D \tau(t)} +
\sum_{i=2}^{\eta} \int_{\varDelta_{k,i}} \frac{t-1}{t}
\D \tau(t)
   \\
& = \frac{(\int_0^{\theta_k} (t-1) \D
\tau(t))^2}{\int_0^{\theta_k} t (t-1) \D \tau(t)} +
\int_{\theta_{k+1}}^{\infty} \frac{t-1}{t} \D \tau(t),
\quad k\in \nbb.
   \end{align*}
Hence, by applying the Lebesgue monotone and dominated
convergence theorems, we get the inequality ``$\Le$''
in \eqref{oczko1}. This completes the proof of
\eqref{oczko1}.

To prove the remaining part of the conclusion, we
define for every $\eta\in \nbb_2 \cup \{\infty\}$ a
partition $\{\tilde\varDelta_{\eta,i}\}_{i=1}^{\eta}$
of $\supp{\tau}$ by (compare with \eqref{dopor})
   \begin{align*}
\tilde\varDelta_{\eta,i} =
   \begin{cases}
   \{\theta_i\} & \text{if } i \in J_{\eta-1},
\\[1ex]
   \{\theta_{\eta}, \theta_{\eta+1}, \ldots\} &
   \text{if } \eta < \infty \text{ and } i=\eta.
   \end{cases}
   \end{align*}
Then \allowdisplaybreaks
   \begin{align*}
\varLambda(\{\tilde\varDelta_{\eta,i}\}_{i=1}^\eta) =
\sum_{i=1}^{\eta-1} \frac{\theta_i -1}{\theta_i}
\tau(\{\theta_i\}) & +
\frac{(\int_{\theta_{\eta}}^{\infty}(t-1) \D
\tau(t))^2}{\int_{\theta_{\eta}}^{\infty} t (t-1) \D
\tau(t)}
   \\
&\Ge \int_0^{\theta_{\eta-1}} \frac{t-1}{t} \D
\tau(t), \quad \eta \in \nbb_2.
   \end{align*}
Applying the Lebesgue monotone convergence theorem and
using \eqref{oczko2}, we get \eqref{oczko3}. In the
case of $\eta=\infty$, we have
   \begin{align} \label{abc}
\varLambda(\{\tilde\varDelta_{\infty,i}\}_{i=1}^{\infty})
= \int_0^{\infty} \frac{t-1}{t} \D \tau(t).
   \end{align}
This completes the proof of the proposition.
   \end{proof}
   \begin{rem}
It is worth pointing out that under the assumptions of
Proposition \ref{supinf}, the following two
inequalities hold
   \begin{align} \label{strict}
\frac{\big(\int_{\varDelta} (t-1) \D
\tau(t)\big)^2}{\int_{\varDelta} t (t-1) \D \tau(t)}
\Le \sum_{i=1}^n \frac{(\int_{\varDelta_i}(t-1) \D
\tau(t))^2}{\int_{\varDelta_i}t (t-1) \D \tau(t)} \Le
\int_{\varDelta} \frac{t-1}{t} \D \tau(t)
   \end{align}
whenever $\{\varDelta_i\}_{i=1}^n$ is a finite or
infinite partition of a nonempty subset $\varDelta$ of
$\supp{\tau}$ (argue as in the proof of
\eqref{kwaint}). We will show that if, in addition,
$\card{\varDelta_i} \Ge 2$ for at least one $i\in
J_n$, then the second inequality (counting from the
left) in \eqref{strict} is strict. Indeed, otherwise
taking a close look at \eqref{kwaint} shows that the
Cauchy-Schwarz inequality \eqref{deltain} becomes an
equality for every $\varDelta \in \{\varDelta_i\colon
i\in J_n\}$. As a consequence, for every $i\in J_n$,
there exists $\alpha_i \in (0,\infty)$ such that
$\sqrt{(t-1)t} = \alpha_i \sqrt{\frac{t-1}{t}}$ for
every $t \in \varDelta_i$ (recall that $\varDelta_i
\subseteq \supp{\tau}$). This implies that
$\card{\varDelta_i}=1$ for every $i\in J_n$, which
contradicts our assumption.

It follows from the previous paragraph that for every
$\eta \in \nbb \cup \{\infty\}$ and for every
partition $\{\varDelta_i\}_{i=1}^{\eta}$ of
$\supp{\tau}$ such that
$\sup\{\card{\varDelta_i}\colon i \in J_{\eta}\} \Ge
2$ we have
   \begin{align}  \label{abc3}
\varLambda(\{\varDelta_i\}_{i=1}^{\eta}) <
\int_0^\infty \frac{t-1}{t} \D \tau(t).
   \end{align}
However, if $\sup\{\card{\varDelta_i}\colon i \in
J_{\eta}\} =1$ (and consequently $\eta=\infty$), then
the strict inequality in \eqref{abc3} becomes an
equality (cf.\ \eqref{abc}).

Applying \eqref{oczko3} and \eqref{strict}, we deduce
that the sequence
   \begin{align*}
\bigg\{\sup\Big\{\varLambda(\{\varDelta_i\}_{i=1}^{\eta})
\colon \{\varDelta_i\}_{i=1}^{\eta} \text{ satisfies }
\eqref{gs6}\Big\}\bigg\}_{\eta=1}^{\infty}
   \end{align*}
is monotonically increasing to $\int_0^\infty
\frac{t-1}{t} \D \tau(t)$. Hence, there arises the
question whether or not there exists $\eta_0\in \nbb$
at which the above sequence stabilizes.
   \end{rem}
   \subsection*{Acknowledgements} A substantial part
of this paper was written while the second and the
fourth author visited Kyungpook National University
during the spring and the autumn of 2015. They wish to
thank the faculty and the administration of this unit
for their warm hospitality.
   \bibliographystyle{amsalpha}

\begin{thebibliography}{99}
   \bibitem{als-c} W. A. Al-Salam, L. Carlitz,  Some orthogonal
\mbox{$q$-polynomials}, {\em Math. Nachr.} {\bf 30}
(1965), 47-61.
   \bibitem{Apos} T. M. Apostol,  {\em Introduction to analytic number
theory}, Undergraduate Texts in Mathematics,
Springer-Verlag, New York-Heidelberg, 1976.
   \bibitem{Ash} R. B. Ash, {\em Probability and measure theory},
Harcourt/Academic Press, Burlington, 2000.
   \bibitem{b-c-r} C. Berg, J. P. R.
Christensen, P. Ressel, {\em Harmonic Analysis on
Semigroups}, Springer, Berlin, 1984.
   \bibitem{ber-dur} C. Berg, A. J. Dur\'{a}n, The
index of determinacy for measures and the
$\ell^2$-norm of orthonormal polynomials, {\em Trans.
Amer. Math. Soc.} {\bf 347} (1995), 2795-2811.
   \bibitem{ber-thil1}  C. Berg, M. Thill,
Rotation invariant moment problems, {\em Acta Math}.
{\bf 167} (1991), 207-227.
   \bibitem{ber-thil2}  C. Berg, M. Thill,
A density index for the Stieltjes moment problem, {\em
Orthogonal polynomials and their applications (Erice,
1990)}, 185-188, {\em IMACS Ann. Comput. Appl. Math.},
{\bf 9}, Baltzer, Basel, 1991.
   \bibitem{b-v} C. Berg, G. Valent, The Nevanlinna
parametrization for some indeterminate Stieltjes
moment problems associated with birth and death
processes, {\em Methods Appl. Anal.} {\bf 1} (1994),
169-209.
   \bibitem{bis} E. Bishop, Spectral theory for operators on a
Banach space, {\em Trans. Amer. Math. Soc.} {\bf 86}
(1957), 414-445.
   \bibitem{bra} J. Bram, Subnormal operators, {\em Duke Math. J.}
{\bf 22} (1955), 75-94.
   \bibitem{Bu} P. Budzy\'{n}ski, A note on unbounded hyponormal
       composition
operators in $L^2$-spaces, {\em J. Funct. Sp. Appl.}
Volume 2012, doi:10.1155/2012/902853.
   \bibitem{b-d-j-s} P. Budzy\'{n}ski,
P. Dymek, Z. J. Jab{\l}o\'nski, J. Stochel, Subnormal
weighted shifts on directed trees and composition
operators in $L^2$ spaces with non-densely defined
powers, {\em Abs. Appl. Anal.} {\bf 2014} (2014),
Article ID 791817, 6 pp.
   \bibitem{b-j-j-sA} P.\ Budzy\'{n}ski,
Z.\ J.\ Jab{\l}o\'nski, I. B. Jung, J. Stochel,
Unbounded subnormal weighted shifts on directed trees,
{\em J. Math. Anal. Appl.} {\bf 394} (2012), 819-834.
   \bibitem{b-j-j-sB} P.\ Budzy\'{n}ski,
Z.\ J.\ Jab{\l}o\'nski, I. B. Jung, J. Stochel,
Unbounded subnormal weighted shifts on directed trees.
II, {\em J. Math. Anal. Appl.} {\bf 398} (2013),
600-608.
   \bibitem{b-j-j-sC} P.\ Budzy\'{n}ski,
Z.\ J.\ Jab{\l}o\'nski, I. B. Jung, J. Stochel, On
unbounded composition operators in $L^2$-spaces, {\em
Ann. Mat. Pur. Appl.} {\bf 193} (2014), 663-688.
   \bibitem{b-j-j-sS} P.\ Budzy\'{n}ski,
Z.\ J.\ Jab{\l}o\'nski, I. B. Jung, J. Stochel,
Unbounded subnormal composition operators in
$L^2$-spaces, {\em J. Funct. Anal.} {\bf 269} (2015),
2110-2164.
   \bibitem{b-j-j-sSq} P. Budzy\'{n}ski, Z. J.
Jab{\l}o\'nski, I. B.Jung, J. Stochel, A subnormal
weighted shift on a directed tree whose $n$th power
has trivial domain, {\em J. Math. Anal. Appl.} {\bf
435} (2016), 302-314.
   \bibitem{b-j-j-sW} P. Budzy\'{n}ski, Z. J.
Jab{\l}o\'nski, I. B. Jung, J. Stochel, Unbounded
weighted composition operators in $L^2$-spaces,
http://arxiv.org/abs/1310.3542.
   \bibitem{ca-hor} J. T. Campbell, W. E. Hornor,
Seminormal composition operators. {\em J. Operator
Theory} {\bf 29} (1993), 323-343.
   \bibitem{Car} T. Carleman, {\em Les fonctions
quasi analytiques}, Gauthier-Villars, Paris, 1926.
   \bibitem{chi} T. S. Chihara, On determinate
Hamburger moment problems, {\em Pacific J. Math.} {\bf
27} (1968), 475-484.
   \bibitem{chi-b} T. S. Chihara, {\em An Introduction
to Orthogonal Polynomials}, Mathematics and Its Applications,
vol. 13, Gordon and Breach Science Publishers, New York,
1978.
   \bibitem{Con} J. B. Conway, {\em The theory of
subnormal operators}, Mathematical Surveys and
Monographs, {\bf 36}, American Mathematical Society,
Providence, RI, 1991.
   \bibitem{Con2} J. B. Conway,  {\em A course in operator
theory}, Graduate Studies in Mathematics, {\bf 21},
American Mathematical Society, Providence, RI, 2000.
   \bibitem{c-f} J. B. Conway, N. S. Feldman,
The state of subnormal operators, {\em A glimpse at
Hilbert space operators}, 177-194, {\em Oper. Theory
Adv. Appl.}, {\bf 207}, Birkh\"auser Verlag, Basel,
2010.
   \bibitem{c-j-k} J. B. Conway, K. H. Jin,
S. Kouchekian, On unbounded Bergman operators, {\em J.
Math. Anal. Appl.} {\bf 279} (2003), 418-429.
   \bibitem{cu-fi} R. E. Curto, L. A. Fialkow,
Recursively generated weighted shifts and the
subnormal completion problem, II, {\em Integr. Equat.
Oper. Th.} {\bf 18} (1994), 369-426.
   \bibitem{emb} M. R. Embry, A generalization of the
Halmos-Bram criterion for subnormality, {\em Acta Sci.
Math. {\em (}Szeged{\em )}} {\bf 35} (1973), 61-64.
   \bibitem{foi} C. Foia\c{s}, D\'ecompositions en
op\'erateurs et vecteurs propres. I., \'Etudes de ces
d\`ecompositions et leurs rapports avec les
prolongements des op\'erateurs, {\it Rev. Roumaine
Math. Pures Appl.} {\bf 7} (1962), 241-282.
   \bibitem{fug} B. Fuglede, The multidimensional
moment problem, {\em Expo. Math.} {\bf 1} (1983),
47-65.
    \bibitem{g-w} R. Gellar,  L. J. Wallen,
Subnormal weighted shifts and the Halmos-Bram
criterion, {\em Proc. Japan Acad.} {\bf 46} (1970),
375-378.
   \bibitem{hal1} P. R. Halmos, Normal dilations and
extensions of operators, {\em Summa Brasil. Math.}
{\bf 2} (1950), 125-134.
    \bibitem{hal2} P. R. Halmos, Ten problems in Hilbert
space, {\em Bull. Amer. Math. Soc.} {\bf 76} (1970),
887-933.
    \bibitem{Har-Wh} D. Harrington, R. Whitley, Seminormal
composition operators, {\em J. Operator Theory} {\bf
11} (1984), 125-135.
   \bibitem{Hav2} E. K. Haviland, On the momentum
problem for distribution functions in more than one
dimension. II. {\em Amer. J. Math.} {\bf 58} (1936), 164-168.
   \bibitem{Herr} D. Herrero, {\em Subnormal bilateral weighted
shifts}, Notas mimeografiadas, 1971.
   \bibitem{Ism} M. E. H. Ismail,  A queueing model and a set
of orthogonal polynomials, {\em J. Math. Anal. Appl.}
{\bf 108} (1985), 575-594.
   \bibitem{Ism-book}  M. E. H. Ismail, {\em Classical and
quantum orthogonal polynomials in one variable},
Encyclopedia of Mathematics and its Applications, Vol.
98, Cambridge University Press, Cambridge, 2005.
   \bibitem{Ism-Mas} M. E. H. Ismail, D. R. Masson,
$q$-Hermite polynomials, biorthogonal rational
functions, and $q$-beta integrals, {\em Trans. Amer.
Math. Soc.} {\bf 346} (1994), 63-116.
   \bibitem{jab} Z. Jab{\l}onski, Hyperexpansive
composition operators, {\em Math. Proc. Cambridge
Philos. Soc.} {\bf 135} (2003), 513-526.
   \bibitem{j-j-s} Z. J. Jab{\l}o\'nski,  I. B. Jung,
J. Stochel, Weighted shifts on directed trees, {\em
Mem. Amer. Math. Soc.} {\bf 216} (2012), no.\ 1017,
viii+107pp.
   \bibitem{j-j-s0} Z. J. Jab{\l}o\'nski,  I. B. Jung,
J. Stochel, A non-hyponormal operator generating
Stieltjes moment sequences, {\em J. Funct. Anal.} {\bf
262} (2012), 3946-3980.
  \bibitem{jor} P. E. T. Jorgensen, Commutative
algebras of unbounded operators, {\em J. Math. Anal.
Appl.} {\bf 123} (1987), 508-527.
   \bibitem{j-s} I. B. Jung, J. Stochel, Subnormal
operators whose adjoints have rich point spectrum,
{\em J. Funct. Anal.} {\bf 255} (2008), 1797-1816.
   \bibitem{kou} S. Kouchekian,
The density problem for unbounded Bergman operators,
{\em Integr. Equ. Oper. Theory} {\bf 45} (2003),
319-342.
   \bibitem{k-t1} S. Kouchekian, J. E. Thomson,
The density problem for self-commutators of unbounded
Bergman operators, {\em Integr. Equ. Oper. Theory}
{\bf 52} (2005), 135-147.
   \bibitem{k-t2} S. Kouchekian,  J. E. Thomson,
On self-commutators of Toeplitz operators with
rational symbols, {\em Studia Math.} {\bf 179} (2007),
41-47.
   \bibitem{lam} A. Lambert, Subnormality and weighted
shifts, {\em J. London Math. Soc.} {\bf 14} (1976),
476-480.
   \bibitem{lam1} A. Lambert, Subnormal composition operators,
{\em Proc. Amer. Math. Soc.} {\bf 103} (1988),
750-754.
   \bibitem{m-s} G. McDonald, C. Sundberg, On the spectra
of unbounded subnormal operators, {\em Canad. J.
Math.} {\bf 38} (1986), 1135-1148.
   \bibitem{nor} E. Nordgren, Composition operators on Hilbert
spaces, Lecture Notes in Math. {\bf 693},
Springer-Verlag, Berlin 1978, 37-63.
   \bibitem{Ped} H. L. Pedersen, Stieltjes
moment problems and the Friedrichs extension of a
positive definite operator, {\em J. Approx. Theory}
{\bf 83} (1995), 289-307.
   \bibitem{P-S} G. P\'olya,  G. Szeg\"o,
{\em Problems and theorems in analysis, Vol. II,
Theory of functions, zeros, polynomials, determinants,
number theory, geometry}, revised and enlarged
translation by C. E. Billigheimer of the fourth German
edition, Die Grundlehren der Mathematischen
Wissenschaften, Band 216. Springer-Verlag, New
York-Heidelberg, 1976.
   \bibitem{M-R} M. Riesz, Sur le probl\`{e}me des moments
et le th\'{e}or\`{e}me de Parseval correspondant, {\em
Acta Litt. ac Scient. Szeged} {\bf 1} (1923), 209-225.
   \bibitem{Rud} W. Rudin, {\em Real and Complex
Analysis}, McGraw-Hill, New York 1987.
   \bibitem{Sa-Zy} S. Saks, A. Zygmund, {\em Analytic
functions}, Monografie Matematyczne, Tom 28,
Pa\'{n}stwowe Wydawnietwo Naukowe, Warsaw 1965.
   \bibitem{sim} B. Simon, The classical moment problem
as a self-adjoint finite difference operator, {\em
Adv. Math.} {\bf 137} (1998), 82-203.
   \bibitem{Sti} T. Stieltjes,
Recherches sur les fractions continues, {\em Anns.
Fac. Sci. Univ. Toulouse} {\bf 8} (1894-1895),
J1-J122; {\bf 9}, A5-A47.
   \bibitem{sto-a} J. Stochel, Moment functions on
real algebraic sets, {\em Ark. Mat.} {\bf 30} (1992),
133-148.
   \bibitem{sto-t} J. Stochel, Seminormality of operators from
their tensor product, {\it Proc. Amer. Math. Soc.}
{\bf 124} (1996), 135-140.
   \bibitem{2xSt2} J. Stochel, J. B. Stochel,
On the $\varkappa$th root of a Stieltjes moment
sequence, {\em J. Math. Anal. Appl.}, {\bf 396}
(2012), 786-800.
    \bibitem{StSz1} J. Stochel, F. H. Szafraniec,
On normal extensions of unbounded operators. I, {\it
J. Operator Theory} {\bf 14} (1985), 31-55.
    \bibitem{StSz2} J. Stochel and F. H. Szafraniec, On normal
extensions of unbounded operators, II, {\em Acta Sci.
Math. $($Szeged$)$} {\bf 53} (1989), 153-177.
   \bibitem{StSz3} J. Stochel, F. H. Szafraniec,
On normal extensions of unbounded operators. III.
Spectral properties, {\em Publ. RIMS, Kyoto Univ.}
{\bf 25} (1989), 105-139.
   \bibitem{StSz} J. Stochel, F. H. Szafraniec, The
complex moment problem and subnormality: a polar
decomposition approach, {\em J. Funct. Anal.} {\bf
159} (1998), 432-491.
   \bibitem{StB} J. B. Stochel, Weighted quasishifts,
generalized commutation relation and subnormality,
{\em Ann. Univ. Sarav. Ser. Math.} {\bf 3} (1990),
109-128.
   \bibitem{FHSz} F. H. Szafraniec, Sesquilinear
selection of elementary spectral measures and
subnormality, {\em Elementary operators and
applications} (Blaubeuren, 1991), 243-248, {\em World
Sci. Publ., River Edge, NJ}, 1992.
   \bibitem{Sz4} F. H. Szafraniec, On normal extensions
of unbounded operators. IV. A matrix construction,
{\em Operator theory and indefinite inner product
spaces}, 337-350, {\em Oper. Theory Adv. Appl.}, {\bf
163}, Birkh\"auser, Basel, 2006.
   \bibitem{Weid} J. Weidmann, {\em Linear operators
in Hilbert spaces}, Springer-Verlag, Berlin,
Heidelberg, New York, 1980.
   \end{thebibliography}
   
   \end{document}